\documentclass[11pt,a4paper,reqno,oneside,final]{amsart}
\usepackage[a4paper,margin=1in]{geometry}
 
\usepackage[utf8]{inputenc}
\usepackage[T1]{fontenc}
\usepackage[english]{babel}
\usepackage{amsmath, amsfonts, amssymb, amsthm, mathtools}
\usepackage[dvipsnames]{xcolor}
\usepackage[shortlabels]{enumitem}
\setlist[description]{leftmargin=\parindent,itemsep=5pt}
\usepackage[languagenames,fixlanguage]{babelbib}
\usepackage{commath}
\usepackage{textcomp}
\usepackage{mathrsfs,dsfont}
\usepackage{csquotes} 

\usepackage{booktabs,hhline}
\makeatletter
\def\hlinewd#1{%
	\noalign{\ifnum0=`}\fi\hrule \@height #1 \futurelet
	\reserved@a\@xhline}
\makeatother

\theoremstyle{plain}
\newtheorem{theorem}{Theorem}[section]

\newtheorem{lemma}[theorem]{Lemma}
\newtheorem{proposition}[theorem]{Proposition}

\theoremstyle{definition}
\newtheorem{remark}[theorem]{Remark}

\newtheorem{definition}[theorem]{Definition}

\renewcommand{\leq}{\leqslant}

\renewcommand{\geq}{\geqslant}

\renewcommand{\phi}{\varphi}
\renewcommand{\rho}{\varrho}

\newcommand{\real}{\mathds{R}}
\newcommand{\comp}{\mathds{C}}
\newcommand{\nat}{\mathds{N}}
\newcommand{\integer}{\mathds{Z}}
\newcommand{\Pp}{\mathds{P}}
\newcommand{\Ee}{\mathds{E}}

\newcommand{\one}{\mathds{1}}
\newcommand{\bbjedan}{\mathds{1}}
\newcommand{\nubar}{\bar{\nu}}

\newcommand{\calU}{\mathcal{U}}
\newcommand{\calQ}{\mathcal{Q}}
\newcommand{\ducalU}{\widehat{\mathcal{U}}}   
\newcommand{\ducalQ}{\widehat{\mathcal{Q}}}   

\DeclareMathOperator{\id}{id}

\newcommand{\phat}{\widehat{p}}
\newcommand{\Phat}{\widehat{P}}
\newcommand{\hatphi}{\widehat{\phi}}

\newcommand{\duu}{\widehat{u}}
\newcommand{\duq}{\widehat{q}}

\newcommand{\dumu}{\widehat{\mu}}
\newcommand{\dum}{\widehat{m}}

\newcommand{\casymp}[1]{\stackrel{#1}{\asymp}}

\makeatletter
\def\namedlabel#1#2{\begingroup
	#2%
	\def\@currentlabel{#2}%
	\label{#1}\endgroup
}
\makeatother

\numberwithin{equation}{section}

\title[Discrete-time Markov chains killed by potentials]{Heat kernels, intrinsic contractivity and ergodicity\\ of discrete-time Markov chains killed by potentials}

\author[W.~Cygan]{Wojciech Cygan}
\address[W.~Cygan]{
	Faculty of Computer Science and Mathematics, Institute of Mathematics,
	University of Wrocław, Poland}
\email{wojciech.cygan@uwr.edu.pl}

\author[K.~Kaleta]{Kamil Kaleta}
\address[K.~Kaleta]{Faculty of Pure and Applied Mathematics,
	Wrocław University of Science and Technology, Poland}
\email{kamil.kaleta@pwr.edu.pl}

\author[R.L.~Schilling]{Ren\'e L.\ Schilling}
\address[R.L.~Schilling]{TU Dresden, Fakultät Mathematik, 
	Institut für Mathematische Stochastik, 01062 Dresden, Germany}
\email{rene.schilling@tu-dresden.de}

\author[M.~\'{S}liwi\'{n}ski]{Mateusz \'{S}liwi\'{n}ski}
\address[M.~\'{S}liwi\'{n}ski]{Faculty of Pure and Applied Mathematics,
	Wrocław University of Science and Technology, Poland}
\email{mateusz.sliwinski@pwr.edu.pl}

\thanks{Research supported by National Science Centre, Poland, grant no.\ 2019/35/B/ST1/02421, the 6G-life project (BMBF 16KISK001K) and the  Dresden--Leipzig ScaDS.AI centre}

\date{}

\usepackage{todonotes}
\usepackage{diagbox}

\usepackage[unicode=true, pdfstartview={XYZ null null 1.00}, colorlinks=true, linkcolor=blue, urlcolor=blue, citecolor=purple]{hyperref}

\begin{document}

\begin{abstract}
	We study discrete-time Markov chains on countably infinite state spaces, which are perturbed by rather general confining (i.e.\ growing at infinity) potentials. Using a discrete-time analogue of the classical Feynman--Kac formula, we obtain two-sided estimates for the $n$-step heat kernels $u_n(x,y)$ of the perturbed chain. These estimates are of the form $u_n(x,y)\asymp \lambda_0^n\phi_0(x)\widehat\phi_0(y)+F_n(x,y)$, where $\phi_0$ (and $\widehat\phi_0$) are the (dual) eigenfunctions for the lowest eigenvalue $\lambda_0$; the perturbation $F_n(x,y)$ is explicitly given, and it vanishes if either $x$ or $y$ is in a bounded set. The key assumptions are that the chain is uniformly lazy and that the \enquote{direct step property} (DSP) is satisfied. This means that the chain is more likely to move from state $x$ to state $y$ in a single step rather than in two or more steps. Starting from the form of the heat kernel estimate, we define the intrinsic (or ground-state transformed) chains and we introduce time-dependent ultracontractivity notions -- asymptotic and progressive intrinsic ultracontractivity -- which we can link to the growth behaviour of the confining potential; this allows us to consider arbitrarily slow growing potentials. These new notions of ultracontractivity also lead to a characterization of uniform (quasi-)ergodicity of the perturbed and the ground-state transformed Markov chains. At the end of the paper, we give various examples that illustrate how our findings relate to existing models, e.g.\ nearest-neighbour walks on infinite graphs, subordinate processes or non-reversible Markov chains.
\end{abstract}
	

\maketitle
\allowdisplaybreaks

\section{Introduction}
A classical theme in the study of Markov processes is the influence of killing -- this may be caused by the process exiting a certain set or through the presence of a confining (i.e.\ growing) potential. There is a substantial literature on questions like \enquote{heat kernel and resolvent estimates}, \enquote{behaviour of eigenfunctions}, \enquote{ergodic properties} to mention but a few. For continuous-time Markov processes the so-called \emph{intrinsic ultracontractivity} (IUC) which was introduced by Davies and Simon in \cite{Davies-Simon-JFA} became a key technique in the field.
As far as we know, a similar technique has not really been used in the discrete-time or discrete-space setting. 
Our aim is to make IUC techniques from the continuous-time world available in the discrete setting.
To do so, we investigate discrete-time Markov chains taking values in an infinite state space $X$, which are killed by rather general potentials. Using (a discrete version of) the Feynman--Kac formula we obtain an explicit expression for the $n$-step transition kernels $u_n(x,y)$ in terms of the potential $V$ and the unperturbed Markov chain.  
As it stands, the IUC property is rather restrictive (it is stronger than compactness of the transition semigroups and it is linked to fast growth of the potential) and, therefore, not really useful in the discrete setting. If we require IUC to hold only \enquote{as time moves on} -- this is called \emph{progressive IUC}, which is a relaxed version of the more familiar \emph{asymptotic IUC} -- we can treat arbitrarily slow growing potentials. 
In the continuous-time setting the weaker notion of \emph{progressive IUC} was introduced by Kaleta and Schilling in \cite{Kaleta-Schilling-JFA} and it was further discussed in \cite{KSch-ergodic}. The notion of \emph{discrete} aIUC and pIUC allows us to prove results of the following type: 
Denote by $\phi_0$ and $\widehat\phi_0(y)$ the ground states of the Feynman--Kac semigroup and its dual, and by $\lambda_0$ the corresponding eigenvalue. If one of the states $x,y$ is in a finite set, $u_n(x,y)$ is comparable to $\lambda_0^n\phi_0(x)\widehat\phi_0(y)$. If $x,y$ are both outside a finite set, we obtain a similar comparison, but with an explicit correction term, see Theorem \ref{HK:estimates}. The two-sided comparison $u_n(x,y)\asymp \lambda_0^n\phi_0(x)\widehat\phi_0(y)$ remains valid if one of the points $x,y$ is from an (explicitly given) increasing sequence of sets $A_n\uparrow X$, see Theorem \ref{th:pIUC}. Finally, pIUC ensures 
that the  intrinsic/ground-state transformed semigroup is ergodic uniformly along the exhaustion $A_n$
while the original semigroup is quasi-ergodic uniformly along the exhaustion, see Theorem~\ref{pIUC_erg}.
 For aIUC we even get equivalence of these properties uniformly on the whole space, see Theorem \ref{t:ergodic}. Our findings extend known results both from the continuous and discrete settings, since previous authors rely on IUC and aIUC only (which necessitates further assumptions on the Markov chain and/or the killing mechanism) whereas we use pIUC which comes in many situations \enquote{for free}. 

\medskip
Let us give a few more details on the basic setting. Throughout $\left\{Y_n: n \in \nat_0\right\}$ will be a discrete-time Markov chain taking values in a countably infinite state space $X$. The transition between states is governed by a one-step probability kernel (or sub-probability kernel, in which case the Markov chain $\left\{Y_n: n \in \nat_0\right\}$ is non-conservative) $P(x,y)$, $x,y \in X$. Moreover, $V$ is a positive \emph{confining} potential on $X$, i.e.\ there exists a sequence of finite sets $B_n$ such that $B_n \uparrow X$ and $\lim_{n \to \infty} \inf_{x \in B_n^c} V(x) = \infty$, see assumption \eqref{B} below. The main object of our study is the \emph{discrete Feynman–Kac semigroup} $\{\calU_n: n \in \nat_0\}$, which is defined by
\begin{align*} 
	\calU_0 f = f ,\quad  \calU_nf(x) = \Ee^x \left[\prod_{k=0}^{n-1}\frac{1}{V(Y_k)}  f(Y_n) \right], \quad n\geq 1.
\end{align*}
This semigroup describes the evolution of a Markov chain $\left\{Y_n: n \in \nat_0\right\}$ in $X$ under the influence of the external potential $V$.  The potential may satisfy $V<1$ on some finite subset of $X$, which can lead to an increase in the \enquote{mass} of the semigroup. Eventually, however, $V$ is assumed to tend to infinity, which means that the paths are randomly killed with an intensity determined by $V$. We do not impose self-adjointness; as usual, we denote all \enquote{dual} objects by a hat, e.g.\ $\widehat P(x,y)$ for the dual of the transition kernel. Duality is understood in the space $\ell^2(X,\mu)$ where $\mu:X\to (0,\infty)$ is some measure on $X$, see \eqref{cond:duality}.

This model is a discrete-time and discrete-space analogue of the classical Feynman--Kac semigroup associated with continuous-time Markov processes, see the monograph \cite{Demuth-Casteren} by Demuth \& van Casteren. To the best of our knowledge, this is a relatively new concept in the discrete-time setting. Apart from \cite{CKS-ALEA} -- this is a precursor of the present paper -- we are aware of only two further articles, which work with similar concepts:\ Anastassiou \& Bendikov \cite{Anastassiou_Bendikov} and Cs\'aki \cite{Csaki}. In contrast to our paper, these authors consider only $\integer^d$-valued simple random walks, and their focus differs entirely from ours.

\medskip
We assume that the Markov chain $\left\{Y_n: n \in \nat_0\right\}$, resp.\ its transition kernel $P(x,y)$, satisfy the so-called \emph{direct step property} (DSP). Intuitively, this means that the Markov chain prefers to jump directly from one state to another, rather than passing through intermediate states. For technical reasons, we require that the chain is \emph{uniformly lazy} (this means that the family of one-step return probabilities $P(x,x)$, $x\in X$, is bounded away from zero) and that the one-step transition operator is ultracontractive. Our main assumptions are stated as \eqref{A1}-\eqref{A3}. 

This framework covers a wide variety of Markov chains, among them random walks on discrete metric spaces, whose transition probabilities are governed by kernels satisfying certain regularity conditions, see e.g., Bass \& Levin \cite{Bass_Levin} or Murugan \& Saloff-Coste \cite{Murugan_Saloff-Coste,Murugan_Saloff_2}. 
Further examples can be constructed by discrete-time subordination (as in Bendikov \& Saloff-Coste \cite{Bendikov-Saloff}), see in particular \cite[Section 4]{CKS-ALEA} for applications on infinite graphs. Subordination allows one to construct Markov chains, which are generated by functions of fairly general graph Laplacians -- e.g.\ fractional Laplacians or relativistic Laplacians -- and their Feynman--Kac evolutions. Since the DSP requires that $P>0$, this leads to non-local chains which can jump from \emph{every} state $x$ to \emph{any other} state $y$ in one go, without noticing the geometry of the underlying graph. The only link to the geometry is through the potential $V$, whose growth behaviour creates a new \enquote{geometry}, which is expressed by a new exhausting sequence given by $V$. This becomes evident, in particular, in the significantly different distributional properties when compared to finite-range random walks. For the corresponding discrete Feynman--Kac semigroups we observe different regularity: For example, DSP induces hyper- and ultracontractivity of the intrinsic semigroups for sufficiently fast growing confining potentials, whereas these properties typically fail to hold in the nearest-neighbour setting, regardless of the growth rate of the potential, see Section \ref{sec:NNRW}.  In continuous time, the transition semigroups of finite Markov chains are known to be hypercontractive without additional conditions, see e.g.\ Diaconis \& Saloff-Coste \cite{Diaconis-Saloff}. 

Combining Dirichlet form techniques and notions from geometry/manifolds Lenz and co-authors made substantial progress in the study of Laplacians on infinite graphs, see e.g.\ \cite{keller_lenz_wojciechowski_2021} and \cite{Keller_Lenz}. Our results go in a similar direction, leading to various nonlocal models on infinite graphs, but our methods are mainly of a probabilistic nature. We analyze evolutions in discrete time, and the starting point of our considerations is a (sub-)probabilistic kernel/Markov chain.

Intrinsic ultracontractivity was introduced for continuous-time Markov processes and self-adjoint operators by Davies \& Simon \cite{Davies-Simon-JFA}, while hypercontractivity traces back to Nelson \cite{Nelson}. Both properties were first studied in the context of semigroups associated with the Laplacian and second-order differential operators, including transition semigroups of diffusions and classical Schrödinger semigroups, see e.g.\ \cite{Banuelos,Bakry-Gentil-Ledoux,Davies-mon,Davies-Gross-Simon,Metafune-Spina,Ouhabaz-Rhandi} and references therein. For non-self-adjoint operators, intrinsic ultracontractivity appeared for the first time in the papers by Kim \& Song \cite{Kim-Song-2008,Kim-Song-AoP,Kim-Song-FM}.

Non-local operators and the corresponding jump processes (L\'evy- and L\'evy-type processes) \emph{in continuous time} have attracted a lot of attention. Among many other properties, recent studies cover necessary and sufficient conditions for intrinsic contractivity properties, ground state estimates for Schr\"odinger/Feynman–Kac semigroups \cite{Kulczycki-Siudeja,Chen-Wang,Chen_Wang_1,Kaleta-Kulczycki,Kaleta-Lorinczi-AoP} and transition semigroups of processes killed upon exiting unbounded domains \cite{Kwasnicki,Chen-Kim-wang-CMP}, heat kernel and Green function estimates for such models \cite{Kaleta-Schilling-JFA,Chen-Kim-wang-MAn,Chen-Kaleta-Wang}, or applications of intrinsic ultracontractivity to the study of the ergodic properties of continuous-time Markov processes, see e.g.\ \cite{Takeda,Knobloch-Partzsch,Zhang,Kim-Tagawa,Zhang-Li-Liao}. For discrete Markov chains such properties are known only for particular examples and the proofs rely on \emph{ad hoc} methods, tailor-made for each situation. Our motivation was to come up with a general method covering a wide class of examples.

The actual starting point for our work was the paper \cite{Kaleta-Schilling-JFA} (see also \cite{KSch-ergodic} for applications), where the authors derived large-time heat-kernel estimates for a wide class of non-local Schrödinger operators. In this connection the concept of progressive intrinsic ultracontractivity was introduced as a tool to control the behaviour of general compact semigroups, including those that fail to be (asymptotically) intrinsic ultracontractive. The present paper aims to make this novel approach available also in the discrete setting.

\medskip

Our main contributions can be summarized as follows.

\subsection*{Heat kernel estimates and pIUC} Our main results address the problem of the uniform comparability of the  $n$-step transition kernel $u_n(x,y)$ of the discrete (in time and space) Markov chain perturbed by a potential with the product $\lambda_0^n \varphi_0(x) \hatphi_0(y)$ of its eigenvalues and -functions. Theorem \ref{HK:estimates} provides sharp and general two-sided estimates for $u_n(x,y)$. Here we show that $u_n(x,y)$ is comparable with $\lambda_0^n \varphi_0(x) \hatphi_0(y)$ if either $x$ or $y$ belongs to a finite set $B\subset X$ and we identify a precise correction term $F_n(x,y)$ that modifies this factorization if both $x, y$ are outside of $B$. This result holds for general confining potentials and all exhausting sequences of sets $A_n \uparrow X$ such that $V$ is increasing with respect to $(A_n)$, without any additional conditions on the growth rate of $V$. The function $F_n(x,y)$, which takes the form of a \enquote{weighted two-step transition kernel} from $x$ to $y$ \eqref{F-n-formula}, seems to be the most explicit expression that can be achieved at this level of generality. In particular, it is well-suited for further applications in this paper. In some specific cases, where there is more information on the regularity of $P$ and $V$, the expression \eqref{F-n-formula} can be simplified.

Our estimates lead to a rather delicate observation, see Theorem \ref{th:pIUC}: The comparison $u_n(x,y) \asymp \lambda_0^n \varphi_0(x) \hatphi_0(y)$ extends to the case where $n \geq n_0$, for some $n_0 \in \nat$, and at least one spatial variable belongs to $A_n$, while there is no restriction on the other variable. Moreover, the comparison constants in the two-sided estimate are uniform with respect to those $x, y$ and $n\geq n_0$. These two-sided estimates are the discrete-time and discrete-space counterpart of the \emph{progressive} IUC (pIUC) property, which was introduced in \cite{Kaleta-Schilling-JFA} in the continuous setting (see above). The key observation is that the regularity of the (intrinsic) semigroup improves as $n \uparrow \infty$; see Definition \ref{def:piuc} and the discussion below it for a formal statement. We note that, for this result, the exhaustion $(A_n)$ must be chosen carefully to align geometrically with the specific $n$-dependent \enquote{level sets} determined by the potential $V$ and the kernels $P$ and $\Phat$. The pIUC property seems to be the strongest regularity property, which we can expect from a general compact Feynman--Kac semigroup $\{\calU_n: n \in \nat_0\}$ (and its dual) without assuming any conditions on the rate of growth of $V$.

\subsection*{Necessary and sufficient condition for aIUC and IHC} If the exhausting sequence $(A_n)$ needed for pIUC can be chosen in a way that $A_{n_0} = A_{n_0+1} = \ldots = X$, then the two-sided estimate $u_n(x,y) \asymp \lambda_0^n \varphi_0(x) \hatphi_0(y)$ holds true uniformly for all $x, y \in X$ and $n \geq n_0$. This leads to the strongest estimate in the present setting -- the \emph{asymptotic} IUC (aIUC), see Definition \ref{def:aiuc}. In Theorem \ref{aIUC} we prove that this property is \emph{equivalent} to the condition that for every $x_0 \in X$ there is a constant $C =C(x_0,n_0)>0$ such that $(C/V(x)^{n_0-1}) \leq P(x,x_0)\Phat(x_0,x)$, $x \in X$. This shows that aIUC is a very strong regularity property for the semigroup $\{\calU_n: n \in \nat_0\}$, which entails that the rate of growth of the potential $V$ is sufficiently fast. For $n=1$, the estimates $u_1(x,y) \asymp \lambda_0 \varphi_0(x) \hatphi_0(y)$ can hold only if the initial transition probability $P(x,y)$ is of the form $P(x,y)=P(x,y_0)P(x_0,y)/P(x_0,y_0)$ for fixed $x_0,y_0$. For this reason we often require that $n_0>1$, in other words we do not consider IUC (which is aIUC with $n_0=1$).

The aIUC property means that the kernels of the intrinsic or ground-state transformed semigroups $\{\calQ_n: n \in \nat_0\}$, see Definition~\ref{def-intrinsic}, have certain boundedness properties: For sufficiently large $n$ we have that $\calQ_n$ maps any $\ell^p(\nu)$, $1\leq p<\infty$ into $\ell^{\infty}(\nu)$ where $\nu(x) = \varphi_0(x)\hatphi_0(x)\mu(x)$ for the underlying measure $\mu$ on $X$ (see the discussion at the beginning of Section \ref{subsec:paiuc} for more details). In Theorem \ref{thm:aiuc-ihp} we show that $\ell^p(\nu) \mapsto \ell^{\infty}(\nu)$, $1 \leq p <\infty$, can be equivalently replaced by $\ell^p(\nu) \mapsto \ell^q(\nu)$, $1 \leq p< q <\infty$, i.e.\ aIUC is equivalent to the discrete-time variant of the intrinsic hypercontractivity (IHC), see Definition \ref{def:ihc}. 

A surprising situation happens in the case of nearest-neighbour random walks on a graph of finite geometry (see Section \ref{sec:NNRW}): Almost all typical examples of such processes lead to discrete Feynman–Kac semigroups for which neither aIUC nor IHC holds, see Propositions \ref{prop:noaiuc} and \ref{prop:noihc} for formal statements.

\subsection*{Applications to (quasi-)ergodicity} 
IUC is a powerful tool for analysing uniform ergodicity and quasi-ergodicity of continuous-time Markov processes and their Feynman–Kac semigroups, see e.g.\ \cite{Takeda,Knobloch-Partzsch,Zhang,Kim-Tagawa,Zhang-Li-Liao}. 
Our standard reference for quasi-ergodicity is the monograph by Collet \& Mart{\'{\i}}nez \& San Mart{\'{\i}}n \cite{Collet}. 
As discussed above, even the weaker aIUC variant comes with substantial limitations on the potential and relies on strong assumptions about the models under consideration. In this paper, we aim to overcome these restrictions in the discrete-time setting. Building on the analytic results presented above, we propose a novel framework for studying quasi-ergodicity of a general compact discrete Feynman–Kac semigroup $\{\calU_n: n \in \nat_0\}$ and the ergodicity of the corresponding intrinsic (ground-state transformed) semigroup $\{\calQ_n: n \in \nat_0\}$ without imposing any conditions on the growth of the confining potentials; this discussion also extends to the dual semigroups. As before, we focus on the DSP (direct step property) case. Our primary objective is to understand the extent to which the ergodicity properties can be made uniform in space, and to determine their rate of convergence, with explicit dependence on the kernel $P$ and the potential $V$.

We study this problem for all spaces $\ell^p$, i.e.\ for $1 \leq p \leq \infty$. Theorem~\ref{thm:unif-ergodic} establishes that the semigroup $\{\calQ_n : n \in \nat_0\}$ is uniformly ergodic on $\ell^\infty(\nu)$ -- this is the smallest of these spaces -- where $\nu$ denotes the stationary measure associated with the intrinsic semigroup. Moreover, the convergence is geometric. This follows from an application of the well-known \emph{Doeblin condition}. Particularly striking is the fact that this result is driven primarily by the DSP property of the kernel $P$, while $V$ may be any confining  potential, regardless of how slowly it grows at infinity. In contrast, no such statement holds for the quasi-ergodicity of the original Feynman–Kac semigroup; see Remark~\ref{rem:neg_l_infty}.

Our main result, Theorem~\ref{pIUC_erg}, focuses on (quasi-)ergodicity in $\ell^1$, which is the most intricate case. It shows that pIUC, relative to a fixed exhaustion $(A_n)$, yields uniform quasi-ergodicity of the semigroup $\{\calU_n: n \in \nat_0\}$ and uniform ergodicity of the intrinsic semigroup $\{\calQ_n: n \in \nat_0\}$ (and their duals) along the exhaustion $(A_n)$, with an explicit convergence rate $\sup_{x \in A_n^c} (1/V(x))$. The proof of this theorem relies on Proposition \ref{kappa_lem}, which explains -- in a more abstract way -- the connection between uniform ergodicity of the intrinsic semigroup and uniform quasi-ergodicity of the original discrete Feynman--Kac semigroup along the same exhaustion in the $\ell^1$ setting. It reveals that both concepts are, in fact, equivalent, and that the corresponding rates of convergence to the stationary (resp.\ quasi-stationary) measure coincide.

We note that intrinsic semigroups are a powerful tool, as they are typically much more regular than the original semigroups. For instance, they can be used to identify the ergodic rate of the \emph{killed} process. In this context, Proposition~\ref{kappa_lem} can be compared to a general result by Diaconis \emph{et al.} \cite[Theorem 7.9]{Diaconis}, which describes the convergence rate to a quasi-stationary measure for a Dirichlet-type Markov chain killed upon exiting a finite set. Our result can be viewed as a uniform analogue of this result in a \enquote{soft-killing} setting, where paths are killed outside a finite set with an intensity governed by a confining potential. It is worth pointing out that we establish the equivalence between the ergodicity of the intrinsic semigroup and the quasi-ergodicity of the original Feynman–Kac semigroup, with matching rates of convergence. In contrast, \cite{Diaconis} proves only one direction. This highlights that Proposition~\ref{kappa_lem} may be of independent interest. As far as we know, such a result is unknown in the continuous-time case. For example, the recent paper \cite{KSch-ergodic} does not consider intrinsic semigroups at all.

Yet another link between ergodicity and the regularizing effects of the intrinsic semigroups is given in Theorem \ref{t:ergodic}. We show that in the aIUC regime both the intrinsic and the original semigroup are uniformly ergodic (resp., quasi-ergodic) in appropriate $\ell^1$ spaces with geometric rates of convergence.  The classical Riesz-Thorin interpolation argument allows us to show that the pIUC property can be used to obtain uniform quasi-ergodicity of the intrinsic semigroup with geometric rate of convergence in the space $\ell^p(\nu) $, for any parameter $p>1$, which may be arbitrarily close to $1$, see Proposition~\ref{prop-riesz-thorin}.

\subsection*{Overview} Let us briefly discuss the structure of the present paper. In Section \ref{sec:Prel} we collect all necessary definitions, introduce the underlying Markov chain, the dual Markov chain, and we formulate the basic assumptions on both processes and potentials.  We define the corresponding discrete Feynman--Kac semigroups. Under the assumption that the potential is confining, these semigroups turn out to be compact, and this allows us to study their ground state eigenfunctions. We also examine the ultracontractivity of the Feynman--Kac semigroups and obtain a useful  formula for its heat kernel. 
In Section \ref{sec:HK-est}, we present the estimates for the heat kernels in the DSP settings. 
Section \ref{sec:IUC} is devoted to the asymptotic/progressive intrinsic ultracontractivity (a/pIUC) and hypercontractivity (IHC) properties. 
In Section \ref{sec:NNRW}, we discuss discrete Feynman--Kac semigroups derived from nearest-neighbour random walks, which evolve on a given graph of finite geometry.  
Section \ref{sec:ergodic} is devoted to the study of ergodic properties of discrete (in time and space) Feynman--Kac semigroups. 

In the closing Section \ref{sec:ex} we illustrate our main results by considering numerous examples. Through various choices of  sub-Markov kernels and potentials we can observe different regularizing effects and rates of ergodic convergence. A summary of our findings is contained in Table~\ref{Table}. 

\section{Preliminaries}\label{sec:Prel}
In this section we introduce discrete-time Feynman--Kac semigroups and discuss some properties of the underlying Markov chains. 

\subsection*{Markov chains and duality}
Let $X$ be a countably infinite set and $P : X \times X \to [0,1]$ a (sub-)probability kernel, that is
\begin{align} \label{def:prob_kernel}
    \sum_{y \in X} P(x,y) \leq 1 \ \ \text{for every} \ x \in X. 
\end{align}
For every sub-probability kernel $P$, there is a time-homogeneous Markov chain $\left\{Y_n: n \in \nat_0\right\}$ on a suitable probability space $(\Omega, \mathcal F, \Pp)$, with values in $X$ and one-step transition probability given by 
\begin{gather*}
	\Pp(Y_{n+1}=y \mid Y_n=x) = P(x,y),\quad x,y\in X.
\end{gather*}
Throughout we use the following standard notation $\Pp^x(Y_n =y) := \Pp(Y_n=y \mid  Y_0=x)$ for the measure of the process starting at $x \in X$; the corresponding expected value is denoted by $\Ee^x$. If, for every $x\in X$, the sum in \eqref{def:prob_kernel} is equal to $1$, then $\Pp^x$ is a probability measure, and the process $\left\{Y_n: n \in \nat_0\right\}$ is \emph{conservative} with \emph{infinite life-time}. Otherwise, it can be interpreted as a \emph{killed} process with \emph{finite life-time}. It is always possible to add an extra \emph{cemetery} point $\partial$ to the state space and to extend $P$ to $(X\cup\{\partial\}) \times (X\cup\{\partial\})$ so that the extension of $\Pp^x$ to $X\cup\{\partial\}$ becomes a probability measure. We will not follow this path, but deal with (sub-)probability kernels $P(x,y)$. 

We are interested in non-self-adjoint operators, so we do not require reversibility of $P$. Instead, we assume that there is a further, possibly different, (sub-)probability kernel $\Phat \colon X \times X \to [0,1]$ and a measure $\mu:X \to (0,\infty)$ such that the following \emph{duality condition} holds:
\begin{align}\label{cond:duality}
	\mu(x)P(x,y) = \mu(y)\Phat (y,x),\quad x,y\in X,
\end{align} 
Summing over $x \in X$ on both sides of \eqref{cond:duality}, we get 
\begin{align} \label{eq:sub-invar}
	\sum_{x \in X} \mu(x)P(x,y) \leq \mu(y), \quad y \in X,
\end{align}
and if $\Phat$ is conservative, then we have the equality 
\begin{align} \label{eq:invar}
	\sum_{x \in X} \mu(x)P(x,y) = \mu(y), \quad y \in X.
\end{align}
This means that $\mu$ is an invariant measure of $P$. The same is true for $\Phat$ when we sum over $y \in X$ in \eqref{cond:duality}. Conversely, if we know that there exists a positive measure $\mu$ on $X$ such that \eqref{eq:sub-invar} holds, then 
\begin{gather*}
	\Phat (x,y):= \frac{\mu(y)}{\mu(x)} P(y,x)
\end{gather*}
defines a (sub-)probability kernel that satisfies the duality relation \eqref{cond:duality}. Moreover, if \eqref{eq:invar} holds, then $\Phat$ is conservative.

The Markov chain $\left\{\widehat{Y}_n: n \in \nat_0\right\}$ whose transition function is given by $\Phat$ can be obtained from $\left\{Y_n: n \in \nat_0\right\}$ through time reversal. If 
\begin{gather*}
	\Phat (x,y)= P(x,y), \quad x, y \in X,
\end{gather*}
then $\left\{Y_n: n \in \nat_0\right\}$ is said to be \emph{reversible}. In this case the equality \eqref{cond:duality} is called the \emph{detailed balance condition}.

The one-step transition \emph{densities} with respect to the measure $\mu$ are denoted by
\begin{align}\label{eq:dens}
	p(x,y) = \frac{P(x,y)}{\mu(y)}
	\quad\text{and}\quad
	\phat (x,y)=\frac{\Phat (x,y)}{\mu(y)}.
\end{align}
Note that \eqref{cond:duality} is equivalent to the equality $\phat (x,y)=p(y,x)$.

\subsection*{Discrete time Feynman--Kac semigroups}
A \emph{lower bounded potential} is a function $V:X \to (0,\infty)$ such that 
\begin{align}\label{potential_lower_bound}
	\inf_{x \in X} V(x) = V_{-} > 0.
\end{align}

\medskip
\begin{center}\bfseries
Unless mentioned otherwise, all potentials considered in this paper\\are assumed to be lower bounded.
\end{center}
\medskip

\noindent
We define operator semigroups $\{\calU_n: n \in \nat_0\}$ and $\{\ducalU_n: n \in \nat_0\}$ by
\begin{align} \label{eq:FK-discrete}
	\calU_0 f = f ,\quad  \calU_nf(x) = \Ee^x \left[\prod_{k=0}^{n-1}\frac{1}{V(Y_k)}  f(Y_n) \right], \quad n\geq 1,
\end{align}
and 
\begin{align} \label{eq:FK-discrete_dual}
	\ducalU_0 f = f ,\quad  \ducalU_nf(x) = \Ee^x \left[\prod_{k=1}^{n}\frac{1}{V(\widehat Y_k)}  f(\widehat Y_n) \right], \quad n\geq 1,
\end{align}
for any admissible function $f:X\to\comp$. Observe that $\calU_n= \calU^{n}$ and $\ducalU_n = \ducalU ^n$, $n\geq 1$,
where $\calU^{n}$ and $\ducalU^n$ denote the $n$\textsuperscript{th} powers of the operators 
\begin{align}\label{def:calU}
	\calU f(x) = \frac{1}{V(x)} \sum_{y \in X} P(x,y)f(y), \quad x \in X,\\
\intertext{and}\label{dual_semigroup}
	\ducalU f(x) = \sum_{y\in X}\Phat (x,y)\frac{1}{V(y)}f(y),\quad x\in X,
\end{align}
respectively. 

The multiplicative functionals $\prod_{k=0}^{n-1} V(Y_k)^{-1}$ and $\prod_{k=1}^{n} V(\widehat Y_k)^{-1}$ appearing inside the expectations in \eqref{eq:FK-discrete} and \eqref{eq:FK-discrete_dual} are the discrete-time counterparts of the classical Feynman--Kac functionals known from the theory of Brownian motion and more general continuous-time Markov processes, see the monograph by Demuth \& van Casteren \cite{Demuth-Casteren}. Therefore, we call $\{\calU_n: n \in \nat_0\}$ and $\{\ducalU_n: n \in \nat_0\}$ the \emph{discrete time Feynman--Kac semigroups} with potential $V$ associated with the chains $\{Y_n: n \in \nat_0\}$ and $\{\widetilde{Y}_n: n \in \nat_0\}$, respectively. 

Both $\calU$ and $\ducalU$ are bounded linear operators on $\ell^p(X,\mu)$ for every $1 \leq p \leq \infty$. Indeed, for any $f \in \ell^p(X,\mu)$, 
\begin{align}\label{bdd_op}
	|\calU f(x)|\leq \frac{1}{V_{-}} \sum_{y\in X} |f(y)|P(x,y) = \frac{1}{V_{-}} \sum_{y\in X} |f(y)|p(x,y)\mu(y),\\
\label{bdd_du_op}
	|\ducalU f(x)|\leq \frac{1}{V_{-}} \sum_{y\in X} |f(y)|\Phat(x,y) = \frac{1}{V_{-}} \sum_{y\in X} |f(y)|\phat(x,y)\mu(y),
\end{align}
and it is easy to see that the integral operators defined by $P$ and $\Phat$ are contractions on every $\ell^p(X,\mu)$. From \eqref{bdd_op} and \eqref{bdd_du_op} we also get that the condition
\begin{align}\label{p_sup_cond}
	\sup_{x,y\in X}p(x,y)<\infty,
\end{align} 
ensures that $\calU, \ducalU: \ell^p(X,\mu) \to \ell^\infty(X,\mu)$ are bounded for every $1 \leq p < \infty$. 

Moreover, the operators $\ducalU$ and $\calU$ are dual (i.e.\ adjoint) in $\ell^2(\mu)$. Indeed, for every $f,g \in \ell^2(\mu)$ Fubini's theorem and the duality condition \eqref{cond:duality} prove
\begin{align*}
	\sum_{x \in X} g(x) \overline{\calU f(x)} \mu(x) 
	& = \sum_{x \in X} \sum_{y \in X} \frac{g(x)}{V(x)} \mu(x)P(x,y) \overline{f(y)} \\
    & = \sum_{y \in X} \sum_{x \in X}  \frac{1}{V(x)} \Phat (y,x) g(x)\overline{f(y)}  \mu(y) \\
	&= \sum_{y \in X} \ducalU g(y) \overline{f(y)} \mu(y).
\end{align*}
These properties are inherited by $\calU_n$ and $\ducalU_n$ for $n > 1$. 

Because of \eqref{def:calU}, \eqref{dual_semigroup} and \eqref{eq:dens}, both $\calU$ and $\ducalU$ are integral operators with kernels
\begin{gather*}
	u(x,y) = \frac{1}{V(x)} p(x,y) 
	\quad\text{and}\quad 
	\duu(x,y) = \phat(x,y) \frac{1}{V(y)} = \frac{1}{V(y)}p(y,x) = u(y,x).
\end{gather*}
It is straightforward to check by induction that
\begin{gather*}
	\calU_n f(x) 
	= \sum_{y\in X} u_n(x,y) f(y)\mu (y) 
	\quad\text{and}\quad 
	\ducalU_n f(x) = \sum_{y\in X} \duu_n(x,y) f(y)\mu (y), \quad n \in \nat_0,
\end{gather*}
where 
\begin{gather} \label{eq:kernel_0_1}
	u_0(x,y) = \bbjedan_{\{x\}}(y) = \begin{cases} 1,& x=y\\ 0, & x\neq y\end{cases},
	\qquad
	u_1(x,y)=u(x,y) = \frac{1}{V(x)} p(x,y),\\
\label{eq:kernel_n}
	u_n(x,y)
	= \frac{1}{V(x)} \sum_{z_{1},\ldots ,z_{n-1}\in X} p(x,z_{1})\frac{\mu(z_1)}{V(z_{1})}\cdot  \ldots \cdot p(z_{n-2}, z_{n-1}) \frac{\mu(z_{n-1})}{V(z_{n-1})}p(z_{n-1}, y), \quad n>1,
\end{gather}
and
\begin{gather*}
	\duu_n(x,y) = u_n(y,x),\quad n \in \nat_0.
\end{gather*}

\subsection*{Notation}
Most of our notation is standard and is, if needed, explained locally. Throughout, we write $f(x)\asymp g(x)$, $x\in A$, if there is a constant $C>0$ such that $C^{-1}g(x)\leq f(x) \leq C g(x)$, for $x\in A$. We use $\casymp{c}$ to highlight the comparability constant. By $\ell^p(X,\mu)$, $1\leq p\leq\infty$ we denote the usual space of $p$-summable sequences, and whenever there is no ambiguity about the underlying space $X$, we simply write $\ell^p (\mu)$.


\section{Heat Kernel estimates}\label{sec:HK-est}

In this section we prove estimates of $u_n(x,y)$ for the class of Markov chains that enjoy the so-called direct step property (DSP). The $n$-step transition probabilities and densities are recursively defined as 
\begin{gather*}
	P_1(x,y) = P(x,y), \quad P_n(x,y) = \sum_{z \in X} P_{n-1}(x,z)P(z,y), \quad n>1;\\
	p_1(x,y) = p(x,y), \quad p_n(x,y) = \sum_{z \in X} p_{n-1}(x,z)p(z,y)\mu (z), \quad n>1.
\end{gather*}
Throughout this paper we will make the following assumptions on the transition probabilities:
\begin{itemize}\itemsep=5pt
\item[\bfseries(\namedlabel{A1}{A1})] 
	$P(x,y) >0$ for all $x,y \in X$ and there exists a constant $C_* \geq 1$ such that
	\begin{align}\label{eq:DSP_for_P}
		P_2(x,y) \leq C_* P(x,y), \quad x, y \in X.
	\end{align} 
\item[\bfseries(\namedlabel{A2}{A2})] 
	$C_{-}= \inf_{x\in X}P(x,x)>0;$

\item[\bfseries(\namedlabel{A3}{A3})] 
	$\sup_{x,y\in X}\big(P(x,y)/\mu(y)\big)<\infty$, i.e.\ \eqref{p_sup_cond} holds. 
\end{itemize}

Since the assumptions \eqref{A1}--\eqref{A3} are symmetric in $x$ and $y$, \eqref{cond:duality} implies that they are automatically true for $\Phat(x,y)$.

\medskip 
We have already discussed the  assumption \eqref{A3} in the previous section -- it ensures the ultracontractivity of the Feynman--Kac semigroup. The assumption \eqref{A2} means that the Markov chain defined by $P$ is uniformly \enquote{lazy} since the set $\{P(x,x):x\in X\}$ is bounded away from zero. This is a technical condition, which looks very restrictive, but it is typically satisfied by Markov chains with DSP, which we consider in this paper. Finally, \eqref{A1} is fundamental for our study. It is equivalent to the condition that $p(x,y) >0$, for all $x,y \in X$, and there exists a constant $C_*\geq 1$ such that
\begin{align}\label{eq:DSP}
	p_2(x,y) \leq C_* p(x,y), \quad x, y \in X.
\end{align}
The estimates \eqref{eq:DSP_for_P} and \eqref{eq:DSP} are called \enquote{direct step property}. The name derives from the following observation: $P_2(x,y)$ is the probability to move from state $x$ to state $y$ in two consecutive steps, and -- up to a uniform multiplicative constant $C_*$ -- it is dominated by the probability $P(x,y)=P_1(x,y)$ to go from $x$ to $y$ in a single step; this is non-trivial for infinite state spaces $X$.

Clearly, this property extends to  all $n$-step transition probabilities. Under \eqref{eq:DSP_for_P} the condition that $P$ is positive is in fact equivalent to the weaker condition that the Markov chain associated with $P$ is irreducible, i.e.\ for every $x, y \in X$ there exists some $n \in \nat$ such that $P_n(x,y) > 0$.

\subsection*{Estimates for general potentials} 

We are going to present general bounds on the kernel $u_n(x,y)$ under the DSP assumption \eqref{A1}. These estimates do not require any regularity assumptions on the potential $V$.
\begin{proposition}\label{thm:bounds-1}
	Assume \eqref{A1} and \eqref{A2}. Let $V$ be a potential satisfying \eqref{potential_lower_bound}. For every $x,y \in X$ and $n \geq 2$ we have
	\begin{gather}\label{u_n-lower-bound}
		u_n(x,y) \geq c_1^{n-2}\frac{1}{V(x)}\sum_{z\in X}p(x,z)p(z,y)\frac{\mu(z)}{V(z)^{n-1}}
	\end{gather}
	and
	\begin{gather}\label{eq:estm}
		u_n(x,y) \leq c_2^{n-2}\frac{1}{V(x)}\sum_{z\in X}p(x,z)p(z,y)\frac{\mu(z)}{V(z)^{n-1}},
	\end{gather}
	with $c_1=C_{-}$ and $c_2=2C_*$.
\end{proposition}

\begin{proof}
	If we keep in the definition \eqref{eq:kernel_n} of $u_n(x,y)$ only the term with $z_1=z_2=\ldots=z_{n-1}$, we obtain
	\begin{align*}
	u_n(x,y)
		&\geq \frac 1{V(x)}\sum_{z\in X}p(x,z)(p(z,z)\mu(z))^{n-2}\frac{\mu(z)}{V(z)^{n-1}}p(z,y)\\
		&\geq C_{-}^{n-2}\frac{1}{V(x)}\sum_{z\in X}p(x,z)p(z,y)\frac{\mu(z)}{V(z)^{n-1}};
	\end{align*}
	the last inequality follows from \eqref{A2}. 
	
	In order to find the upper bound, we use induction. For $n=2$ \eqref{eq:estm} (with equality) follows directly from the definition of $u_2(x,y)$ as $c_2=2C_* > 1$. Consider the following decomposition
	\begin{align*}
		X \times X 
		&= \left\{(z,w)\in X \times X: V(z)\geq V(w)\right\}\cup \left\{(z,w)\in X \times X: V(z)<V(w)\right\}\\
		& =: D_1\cup D_2.
	\end{align*}
	Suppose that \eqref{eq:estm} holds for some $n>2$ with the constant $c_2=2C_* > 1$. Then we have
	\begin{align*}
	u_{n+1}(x,y) 
	&= \frac 1{V(x)} \sum_{z\in X}p(x,z)\mu(z)u_n(z,y)\\
	&\leq c_2^{n-2} \frac 1{V(x)}\sum_{(z,w) \in D_1}p(x,z)\frac{\mu(z)}{V(z)}p(z,w)\frac{\mu(w)}{V(w)^{n-1}}p(w,y)\\
	&\qquad \mbox{} + c_2^{n-2} \frac 1{V(x)}\sum_{(z,w) \in D_2}p(x,z)\frac{\mu(z)}{V(z)}p(z,w)\frac{\mu(w)}{V(w)^{n-1}}p(w,y)\\
	&\leq c_2^{n-2}\frac 1{V(x)}\sum_{w \in X}\left(\sum_{z \in X}p(x,z)\mu(z)p(z,w)\right)\frac{\mu(w)}{V(w)^{n}}p(w,y)\\
	&\qquad \mbox{} + c_2^{n-2}\frac 1{V(x)}\sum_{z \in X}p(x,z)\frac{\mu(z)}{V(z)^{n}}\left(\sum_{w \in X}p(z,w)\mu(w)p(w,y)\right)\\
	&\leq 2\,C_*\,c_2^{n-2}\frac{1}{V(x)}\sum_{z\in X}p(x,z)\frac{\mu(z)}{V(z)^{n}}p(z,y),
	\end{align*}
	where the last inequality follows from \eqref{eq:DSP}. This completes the induction step, hence the proof.
\end{proof}

\subsection*{Estimates for confining potentials} 

In the remainder of this section we assume that \eqref{A1}--\eqref{A3} hold. 
\begin{definition} 
	A potential $V:X\to (0,\infty)$ is \emph{confining}, if it satisfies
	\begin{itemize}
	\item[\bfseries(\namedlabel{B}{B})] 
	\text{for every $M \in \nat$ there exists a finite set $B_M \subset X$ such that $V(x) \geq M, \ x \in B_M^c.$}
	\end{itemize}
\end{definition} 

The condition \eqref{B} ensures that the operators $\calU_n$, $\ducalU_n$, $n \in \nat$, are compact in $\ell^2(\mu)$, see e.g.\ \cite[Lemma 3.7]{CKS-ALEA}. Due to compactness, the spectra of these operators, excluding zero, are discrete. Moreover, it follows from Jentzsch's theorem \cite[Theorem V.6.6.]{Schaefer} that 
\begin{gather*}
	\lambda_0:= \sup_{\lambda \in \sigma(\calU)} |\lambda| = \sup_{\lambda \in \sigma(\ducalU)} |\lambda| >0
\end{gather*}
is a common simple eigenvalue of both $\calU$ and $\ducalU$. The corresponding eigenfunctions $\phi_0$, $\widehat{\phi}_0$, which we call \emph{ground states}, are strictly positive on $X$. Throughout the paper we assume that the eigenfunctions are normalized:
\begin{align}\label{ass:normalization}
	\left\|\phi_0\right\|_{\ell^2(\mu)}=\left\|\hatphi_0\right\|_{\ell^2(\mu)}=1.
\end{align}
Clearly, we have
\begin{gather*}
	\calU_n \phi_0 = \lambda_0^n \phi_0 
	\quad\text{and}\quad 
	\ducalU_n \widehat{\phi}_0 = \lambda_0^n  \widehat{\phi}_0, \quad n \in \nat.
\end{gather*}
Recall that under \eqref{A3} the semigroups $\{\calU_n: n\in \nat\}$, $\{\ducalU_n: n\in \nat\}$ are ultracontractive. Consequently, the ground states $\phi_0$ and $\widehat{\phi}_0$ are bounded functions. 

Since $\lambda_0^{-1}\calU\phi_0 = \phi_0$, the ground state $\phi_0$ is  $(\calU^{\lambda_0}-\id)$-harmonic in $X$ (here $\calU^{\lambda_0}:= \lambda_0^{-1}\calU$ -- this amounts to changing the potential $V\rightsquigarrow \lambda_0^{-1}V$), we can apply \cite[Theorem~2.2 \& Proposition~2.4]{CKS-ALEA} to show that for every fixed $x_0 \in X$ there exists a constant $c_1=c_1(x_0,V,P)$ such that
\begin{align}\label{phi_0_bound}
	\phi_0(x) \casymp{c_1} \frac{p(x,x_0)}{V(x)},\quad x\in X. 
\end{align}
More precisely, it is shown in \cite{CKS-ALEA} that there exists a finite set $B \subset X$ such that $x_0 \in B$ and the two-sided estimates \eqref{phi_0_bound} hold for all $x \in B^c$. Since $B$ is finite and $\phi_0$ and $p$ are strictly positive and bounded, we can extend these estimates to $x\in B$, hence to $x\in X$; this may require a larger comparability constant. 

Since $\widehat{\phi}_0$ is the ground state of $\ducalU$, the function $\widehat{\phi}_0/V$ is a strictly positive eigenfunction of the operator $\widetilde \calU$ corresponding to the eigenvalue $\lambda_0$, where
\begin{gather*}
	\widetilde \calU g(x) = \frac{1}{V(x)} \sum_{y \in X}  \Phat(x,y)g(y).
\end{gather*}
This observation and the fact that the kernel $\phat$ also satisfies \eqref{A1}-\eqref{A3} allow us to use \eqref{phi_0_bound} for $\widehat{\phi}_0/V$. This gives for every fixed $x_0 \in X$ the following estimate:
\begin{align}\label{phi_hat_bound}
	\widehat{\phi}_0(x) \casymp{c_2} \widehat{p}(x,x_0), \quad x\in X,
\end{align}
with a constant $c_2=c_2(x_0,V,\Phat)$.
 
For our main result in this section we need the following definition.
\begin{definition}\label{def:exh_fam}
\begin{enumerate}
\item
	A sequence of sets $(A_n)_{n\in \nat}$ in $X$ is called \emph{exhausting}, if $A_n\subseteq A_{n+1}$ and $\bigcup_n A_n = X$. 
\item
	For any $x\in X$ we define $ \alpha(x) :=\min\left\{n\in \nat : x\in A_n\right\}$, and we call it the \emph{first appearance function}.
\item
	A potential $V:X\to\real$ is said to be \emph{increasing} with respect to the exhausting sequence $(A_n)_{n\in \nat}$, if $V(x)\leq V(y)$ whenever $\alpha(x) < \alpha(y)$ (or equivalently, $A_{\alpha(x)}\subsetneq A_{\alpha(y)}$).
\end{enumerate}
\end{definition} 
If $V$ is a confining potential, then there is always an exhausting sequence $(A_n)_{n\in \nat}$ such that all $A_n$ are finite sets, and $V$ is increasing with respect to the sequence $(A_n)_{n\in \nat}$.  If $(v_n)$ is an increasing sequence, which represents the range of the potential $V$, then one can take $A_n = \left\{x \in X: V(x) \leq v_n \right\}$ (or simply, $A_n = \left\{x \in X: V(x) \leq n \right\}$, depending on  the particular application).

\begin{theorem}\label{HK:estimates}
	Assume \eqref{A1}--\eqref{A3} and \eqref{B}. Let $(A_n)_{n\in \nat}$ be an exhausting sequence such that $V$ is increasing with respect to it, and let $B \subset X$ be an arbitrary finite set. Then there exists a constant $\widetilde C >0$,  independent of $n\in\nat$, such that the following estimates hold.
\begin{enumerate}
\item\label{HK:estimates-1} 
	If $x\in B$ or $ y\in B$, then 
	\begin{align}\label{inner_estimate}
		u_n(x,y)\casymp{\widetilde C} \lambda_0^n \phi_0(x)\widehat{\phi}_0(y),\quad n\geq 1;
	\end{align}
\item\label{HK:estimates-2} 
	If $x,y\in B^c$, then
	\begin{align}\label{outer_estimate}
		u_n(x,y) \casymp{\widetilde C} \left( \lambda_0^n \phi_0(x)\widehat{\phi}_0(y) + C^n F_n(x,y)\right)\quad n\geq 1,
	\end{align}
	where $C =\max\big\{3C_*,C_{-}^{-1}\big\}$ for the upper bound and $C = \min\big\{C_*^{-1},C_{-}\big\}$ for the lower bound, and
	\begin{align}\label{F-n-formula}
		F_n(x,y)= \frac{1}{V(x)}\sum_{z\in B^c\cap A_{\alpha(x)}\cap A_{\alpha(y)}} p(x,z)p(z,y)\frac{1}{V(z)^{n-1}}\mu (z).
	\end{align}
\end{enumerate}
In particular, we have
\begin{gather}\label{full_lower_estimate}
	u_n(x,y) \geq \widetilde C^{-1} \lambda_0^n\phi_0(x)\widehat\phi_0(y)\quad\text{for all\ \ } x,y\in X.
\end{gather}
\end{theorem}

\begin{proof}
Fix a finite $B \subset X$ and a reference point $x_0 \in B$. 

\medskip\emph{Proof of \ref{HK:estimates-1}}: 
Observe that by \cite[Lemma 2.1]{CKS-ALEA} and by the fact that $0 < \min_{y \in B} \mu(y) \leq \max_{y \in B} \mu(y) < \infty$, we have
\begin{gather*}
	u(x,y) 
	=  \frac{1}{V(x)}p(x,y)
	\asymp 
	\begin{cases}
	\frac{1}{V(x)}p(x,x_0),\quad y\in B,\\
	p(x_0,y),\qquad \quad x\in B.
	\end{cases}
\end{gather*}
It follows from \eqref{phi_0_bound} and \eqref{phi_hat_bound} that
\begin{gather*}
	u(x,y )
	\asymp 
	\begin{cases}
	\phi_0(x),\quad y\in B,\\
	\widehat{\phi}_0(y),\quad  x\in B.
	\end{cases}
\end{gather*}
Since the ground states are strictly positive and bounded on $X$, we see that 
\begin{align*}
	\phi_0 (x)
	&\asymp \lambda_0\phi_0(x)\widehat{\phi}_0(y),\quad y\in B,\\
	\widehat{\phi}_0(y) 
	&\asymp \lambda_0 \phi_0(x)\widehat{\phi}_0(y),\quad x\in B.
\end{align*}
Hence, there is a constant $\widetilde C$ such that \eqref{inner_estimate} holds for $n=1$. In order to show \eqref{inner_estimate} for $n > 1$, we use induction. If $x \in B$, then we have
\begin{align*}
	u_{n+1}(x,y) = \sum_{z\in X} u_n(x,z)u(z,y)\mu(z)
	&\casymp{\widetilde C} \lambda_0^n \sum_{z\in X}\phi_0(x)\widehat{\phi}_0(z)u(z,y)\mu(z) \\
	&= \lambda_0^n\phi_0(x)\widehat{\calU}\widehat{\phi}_0(y)
	= \lambda_0^{n+1}\phi_0(x) \widehat{\phi}_0(y).
\end{align*}
The argument for the case $y\in B$ is similar. This completes the proof of \ref{HK:estimates-1}.

\medskip\emph{Proof of the upper bound in \ref{HK:estimates-2}}: By \eqref{A2}, we have
\begin{align*}
	\frac{C_{-}}{V(x)}p(x,y)\leq \frac{1}{V(x)}p(x,y)p(y,y)\mu(y)\leq F_1(x,y).
\end{align*}
It follows that 
\begin{align*}
	u(x,y)\leq \frac{1}{C_{-}}F_1(x,y) \leq \widetilde C C F_1(x,y),
\end{align*}
which is the upper bound for $n=1$. For $n>1$ we use again induction. Assume that the lower bound in \eqref{outer_estimate} holds for some $n \geq 1$. For the induction step $n\rightsquigarrow n+1$ it is enough to consider the case $A_{\alpha(y)}\subseteq A_{\alpha(x)}$, the opposite case can be treated similarly. We decompose the kernel $u_{n+1}(x,y)$ as follows
\begin{align*}
	u_{n+1}(x,y)
	&= \left( \sum_{z\in B}+\sum_{z\in B^c}\right) u(x,z)u_n(z,y)\mu(z)
	=I(x,y)+J(x,y).
\end{align*}
By \ref{HK:estimates-1}, 
\begin{align}\label{I_bound}
	I(x,y)&\leq \widetilde C \lambda_0^n\widehat{\phi}_0(y)\sum_{z\in B}u(x,z) \phi_0(z)\mu(z).
\end{align}
For $J(x,y)$ we apply the induction hypothesis:
\begin{align*}
	J(x,y)
	&\leq \widetilde C \sum_{z\in B^c}u(x,z)\left( C^n F_n(z,y)+\lambda_0^n\phi_0(z)\widehat{\phi}_0(y)\right) \mu(z)\\
    &= \widetilde C \left(C^n  \sum_{z\in B^c \cap A_{\alpha(y)} }u(x,z) F_n(z,y) \mu(z) 
			                        + C^n \sum_{z\in B^c \cap A_{\alpha(y)}^c }u(x,z) F_n(z,y)\mu(z)\right.  \\
	&\qquad\qquad \left.\mbox{} + \lambda_0^n \widehat{\phi}_0(y) \sum_{z\in B^c} u(x,z)\phi_0(z) \mu(z)\right)\\
	&= \widetilde C \left( J_1(x,y) + J_2(x,y) + J_3(x,y)\right).
\end{align*}
Using a similar decomposition as in the proof of Proposition \ref{thm:bounds-1} and \eqref{eq:DSP}, we get
\begin{align*}
	&J_1(x,y) 
	= \frac{C^{n}}{V(x)} \sum_{z\in B^c \cap A_{\alpha(y)}}p(x,z) \frac{\mu(z)}{V(z)} \sum_{w\in B^c\cap A_{\alpha(z)}\cap A_{\alpha(y)}}  p(z,w)\frac{\mu(w)}{V(w)^{n-1}}p(w,y)\\
    &\quad\leq \frac{C^{n}}{V(x)} \sum_{z,w\in B^c \cap A_{\alpha(y)}} \bigg(\one_{\left\{V(w) \leq V(z)\right\}} + \one_{\left\{V(w) > V(z)\right\}}\bigg)p(x,z) \frac{\mu(z)}{V(z)}p(z,w)\frac{\mu(w)}{V(w)^{n-1}}p(w,y)\\
	&\quad\leq \frac{C^{n}}{V(x)} \left( \sum_{w\in B^c\cap A_{\alpha(y)}}\left(\sum_{z\in X}p(x,z)p(z,w)\mu(z)\right)p(w,y)\frac{\mu(w)}{V(w)^{n}}\right.\\
	&\qquad\qquad\qquad\qquad \left.\mbox{} +  \sum_{z\in B^c\cap A_{\alpha(y)}}p(x,z)\frac{\mu(z)}{V(z)^{n}}\left(\sum_{w\in X}p(z,w)p(w,y)\mu(w)\right)\right)\\
	&\quad\leq 2C_*\frac{C^{n}}{V(x)}\sum_{w\in B^c\cap A_{\alpha(y)}}p(x,w)\frac{\mu(w)}{V(w)^{n}}p(w,y)\\
	&\quad=  2C_* C^{n}F_{n+1}(x,y).
\end{align*}
Moreover, 
\begin{align*}
	J_2(x,y) 
	&= C^n\sum_{z\in B^c \cap A_{\alpha(y)}^c} \frac 1{V(x)} p(x,z) \frac 1{V(z)} \sum_{w\in B^c\cap A_{\alpha(z)}\cap A_{\alpha(y)}}  p(z,w) \frac 1{V(w)^{n-1}} p(w,y)\mu(w)\mu(z)\\
	&\leq C^{n}V(x)^{-1} \sum_{w\in B^c\cap A_{\alpha(y)}} \left(\sum_{z\in X}p(x,z)p(z,w)\mu(z)\right) p(w,y) \frac 1{V(w)^{n}}\mu(w)\\
	&\leq C^{n}C_* \frac 1{V(x)}\sum_{w\in B^c\cap A_{\alpha(y)}}p(x,w)p(w,y) \frac 1{V(w)^{n}}\mu(w)\\
	&= C_*C^{n}F_{n+1}(x,y);
\end{align*}
in the first inequality we use the monotonicity of $V$ (note that $z \in A_{\alpha(y)}^c$ and $w\in A_{\alpha(z)}\cap A_{\alpha(y)} = A_{\alpha(y)}$ imply $A_{\alpha(w)}\subsetneq A_{\alpha(z)}$), and in the second inequality we use again \eqref{eq:DSP}. 

Since $J_3(x,y) = \lambda_0^n \widehat\phi_0(y)\calU_1\phi_0(x)= \lambda_0^{n+1} \widehat\phi_0(y)\phi_0(x)$, the sum $J_1(x,y)+J_2(x,y)+J_3(x,y)$ yields the upper bound in \eqref{outer_estimate}. 

\medskip\emph{Proof of the lower bound in \ref{HK:estimates-2}}: From \eqref{eq:DSP} we see
\begin{align*}
	p(x,y)
	&\geq \frac{1}{C_*}\sum_{z\in X}p(x,z)p(z,y)\mu(z)\\
	&\geq  \frac{1}{C_*}p(x,x_0)p(x_0,y)\mu(x_0)\\
	&= \frac{1}{C_*} p(x,x_0)\widehat{p}(y,x_0)\mu(x_0).
\end{align*}
Using \eqref{phi_0_bound} and \eqref{phi_hat_bound}, and increasing the constant $\widetilde C$ (if necessary), we get for $x,y\in B^c$,
\begin{align*}
	u(x,y)\geq  \frac{2}{ \widetilde C}\lambda_0\,\phi_0(x)\widehat{\phi}_0(y).
\end{align*}
On the other hand, using \eqref{eq:DSP} once again, 
\begin{align*}
	u(x,y)\geq \frac{1}{C_*}F_1(x,y) \geq \frac{2C}{\widetilde C} F_1(x,y).
\end{align*}
Consequently,
\begin{gather*}
	u(x,y) \geq \frac{1}{\widetilde C}\left(\lambda_0\,\phi_0(x)\widehat{\phi}_0(y) + C F_1(x,y)\right),
\end{gather*}
which is the lower bound for $n=1$. We will now use induction. Assume that the bound is true for some $n\geq 1$. We write
\begin{align*}
	u_{n+1}(x,y)  = \left( \sum_{z\in B}+\sum_{z\in B^c}\right) u_n(x,z)u(z,y)\mu(z).
\end{align*}
For the first sum we can use the lower bound from \ref{HK:estimates-1} for $u_n(x,z)$,  for the second sum we use the induction hypothesis: 
\begin{align*}
	u_{n+1}(x,y)  
	&\geq \widetilde C^{-1} \lambda_0^n \phi_0(x) \sum_{z\in B} \widehat{\phi}_0(z)u(z,y)\mu(z) \\
    &\qquad\qquad\mbox{} + \widetilde C^{-1} \sum_{z\in B^c} \big(\lambda_0^n\phi_0(x)\widehat{\phi}_0(z) + C^n F_n(x,z)\big)u(z,y)\mu(z)\\
    &= \widetilde C^{-1} \left(\lambda_0^n \phi_0(x) \sum_{z\in X} \widehat{\phi}_0(z)u(z,y)\mu(z) +  C^n\sum_{z\in B^c} F_n(x,z)u(z,y)\mu(z)\right) \\
	&= \widetilde C^{-1} \left(\lambda_0^{n+1} \phi_0(x) \widehat{\phi}_0(y) +  C^n\sum_{z\in B^c} F_n(x,z)u(z,y)\mu(z)\right).
\end{align*}
Furthermore,
\begin{align*}
	\sum_{z\in B^c} F_n(x,z)u(z,y)\mu(z) 
	&= \frac{1}{V(x)}\sum_{z\in B^c} \sum_{w\in B^c\cap A_{\alpha(x)}\cap A_{\alpha(z)}} p(x,w)\frac{\mu(w)}{V(w)^{n-1}}p(w,z)\frac{1}{V(z)}p(z,y)\mu(z) \\
    &\geq \frac{1}{V(x)}\sum_{z\in B^c \cap A_{\alpha(x)}} p(x,z)p(z,z)\mu(z)\frac{1}{V(z)^n}p(z,y)\mu(z) \\
	&\geq  \frac{C_{-}}{V(x)}\sum_{z\in B^c \cap A_{\alpha(x)} \cap A_{\alpha(y)}} p(x,z)p(z,y)\frac{\mu(z)}{V(z)^n} \\
	&= C F_{n+1}(x,y);
\end{align*}
in the second line we keep only the terms with $w=z$, and in the third line we use \eqref{A2}. This gives the claimed lower bound and completes the proof. 
\end{proof}

\section{Intrinsic contractivity properties}\label{sec:IUC}

In this section we continue our discussion on the behaviour of the kernels $u_n(x,y)$ and connect them with smoothing properties of intrinsic semigroups. We find sharp necessary and sufficient conditions for various intrinsic contractivity properties.

\subsection{Intrinsic semigroups}

We  introduce a new base measure by setting
\begin{gather}\label{eq:v-measure}
	\nu(x):=\phi_0(x)\hatphi_0(x)\mu(x), \quad x \in X.
\end{gather}
The normalization condition \eqref{ass:normalization} for $\phi_0, \widehat\phi_0$ combined with the Cauchy--Schwarz inequality implies that $\nu(x)$ is a sub-probability measure on $X$. 
\begin{definition}\label{def-intrinsic}
	The \emph{intrinsic} (or \emph{ground state transformed}) \emph{semigroups} $\{\calQ_n: n\in \nat_0\}$ and $\{\ducalQ_n: n\in \nat_0\}$ related to the semigroups $\{\calU_n: n\in \nat_0\}$ and $\{\widehat{\calU}_n: n\in \nat_0\}$, respectively, are given by
	\begin{gather*}
		\calQ_n\,f(x)
		= \frac{1}{\lambda_0^{n}\phi_0(x)}\calU_n (f \phi_0)(x) 
		= \sum_{y \in X}\frac{u_n(x,y)}{\lambda_0^{n}\phi_0(x)\hatphi_0(y)}f(y)\nu(y),\\
		\ducalQ_n\,f(x)
		= \frac{1}{\lambda_0^{n}\hatphi_0(x)} \widehat{\calU}_n (f \hatphi_0)(x)
		= \sum_{y \in X}\frac{\duu_n(x,y)}{\lambda_0^{n}\hatphi_0(x)\phi_0(y)}f(y)\nu(y),
	\end{gather*}
	for all admissible functions $f:X \to \comp$. 
\end{definition}
The semigroups $\calQ_n$ and $\ducalQ_n$ are conservative and contractive on $\ell^p(\nu)$ for every $ p \in [1,\infty]$, i.e.\ 
\begin{align}\label{eq:contr} 
	\calQ_n \one_X 
	= \ducalQ_n \one_X 
	= \one_X 
	\quad\text{and}\quad 
	\Vert\calQ_n f\Vert_{\ell^p(\nu)}, \Vert \ducalQ_n f \Vert_{\ell^p(\nu)} 
	\leq \Vert f\Vert_{\ell^p(\nu)}, 
	\quad f \in \ell^p(\nu), \; n \in \nat_0.
\end{align}
This follows from the eigenequations for the operators $\calU_n, \widehat{\calU}_n$, Jensen's inequality and the fact that the measure $y\mapsto u_n(x,y)\nu(y)/(\lambda_0^n\phi_0(x)\widehat\phi_0(y))$ is a probability measure as $\calQ_n\one_X = \one_X$. It is convenient to use the following shorthand notation for the integral kernels
\begin{gather*}
	q_n(x,y)=\frac{u_n(x,y)}{\lambda_0^{n}\phi_0(x)\hatphi_0(y)} 
	\quad\text{and}\quad 
	\duq_n(x,y)=\frac{\duu_n(x,y)}{\lambda_0^{n}\hatphi_0(x)\phi_0(y)}.
\end{gather*}
Clearly, 
\begin{align}\label{eq:q-dual}
	\duq_n(x,y) = q_n(y,x),
\end{align}
and $\calQ_n, \ducalQ_n$ are  dual (i.e.\ adjoint) operators in $\ell^2(\nu)$ for every $n \in \nat$. Observe that $\nu$ is an invariant measure of the semigroups $\{\calQ_n: n\in \nat_0\}$ and $\{\ducalQ_n: n\in \nat_0\}$, i.e.\ 
\begin{gather*}
	\nu(\calQ_nf) = \nu(f) 
	\quad\text{and}\quad 
	\nu(\ducalQ_nf) = \nu(f), \quad n \in \nat_0, \; f \in \ell^p(\nu), \; p \in [1,\infty].
\end{gather*}

\subsection{Progressive and asymptotic intrinsic ultracontractivity} \label{subsec:paiuc}

Intrinsic ultracontractivity was originally introduced in \cite{Davies-Simon-JFA} as a regularity property of continuous-time semigroups of self-adjoint operators. More recent works \cite{Kaleta-Lorinczi-SPA, Kaleta-Lorinczi-AoP} studied a weaker, asymptotic version of this property, which holds for large times only. In the present paper we extend this notion to the case of possibly non-reversible  and not necessarily self-adjoint evolutions with discrete time and infinite discrete state spaces. 

\begin{definition}[\bfseries aIUC] \label{def:aiuc}
	The semigroup $\{ \calU_n : n\in \nat\}$ is \emph{asymptotically intrinsically ultracontractive} (aIUC,  for short) if the intrinsic semigroup $\{\calQ_n: n\in \nat_0\}$ is ultracontractive for large times, i.e.\ there exist some $n_0 \in \nat$ such that
	\begin{align}\label{def:IUC_bdd}
		\calQ_{n_0}: \ell^1(\nu) \longrightarrow \ell^{\infty}(\nu) \ \text{is a bounded  operator}.
	\end{align}
\end{definition}
Asymptotic intrinsic ultracontractivity can be equivalently  expressed by
\begin{align}\label{def:IUC_kernel}
	K_{n_0}:= \sup_{x,y \in X} q_{n_0} (x,y) < \infty.
\end{align}
Moreover, $K_{n_0} = \Vert\calQ_{n_0}\Vert_{\ell^1(\nu) \to \ell^{\infty}(\nu)}$. If we want to highlight the specific value of $n_0$, we write aIUC($n_0$). 

Clearly, these boundedness properties extend to larger  values of $n$: if \eqref{def:IUC_bdd} and \eqref{def:IUC_kernel} are true for some $n_0 \in \nat$, then they hold true for all $n \geq n_0$ because of the semigroup property. In fact,  we may rearrange \eqref{def:IUC_kernel}  and get
\begin{gather*}
	u_{n_0}(x,y) \leq K_{n_0} \phi_0 (x) \widehat{\phi}_0(y), \quad x,y\in X.
\end{gather*} 
Thus, it is immediately clear that
\begin{align}\label{def:IUC_kernel_new}
	u_{n}(x,y) \leq K_{n_0} \lambda_0^n \phi_0 (x) \widehat{\phi}_0(y), \quad x,y\in X, \; n \geq n_0.
\end{align}
From Theorem \ref{HK:estimates} we know that a similar lower bound for $u_{n}(x,y)$ is always true in the present setting. This means that aIUC is a very strong regularity property for the initial semigroup $\{ \calU_n : n\in \nat_0\}$ which implies a full factorization with respect to time and space. 

The aIUC property of the dual semigroup $\{\widehat{\calU}_n: n\in \nat_0\}$ is defined in the same way  using the dual semigroup $\{\ducalQ_n: n\in \nat_0\}$. By \eqref{eq:q-dual}, 
\begin{gather*}
	\Vert\ducalQ_{n_0}\Vert_{\ell^1(\nu) \to \ell^{\infty}(\nu)} 
	= \Vert\calQ_{n_0}\Vert_{\ell^1(\nu) \to \ell^{\infty}(\nu)}
\end{gather*}
and, therefore, 
\begin{align}\label{eq:equiv_aiuc}
	\{\calU_n : n\in \nat\} \text{\ \ is aIUC($n_0$)} \iff \{\widehat{\calU}_n: n\in \nat_0\} \text{\ \ is aIUC($n_0$)}.
\end{align}
Since $\nu$ is a finite measure on $X$, we have the natural inclusions
\begin{align} \label{eq:inclusions-1}
\ell^q(\nu) \subset \ell^p(\nu), \quad 1 \leq p \leq q \leq \infty.
\end{align}
Thus, both ${\calQ}_n$ and ${\ducalQ}_n$ are bounded  operators from $\ell^p(\nu)$ to $\ell^{\infty}(\nu)$, for any $p \geq 1$ and $n \geq n_0$.  Conversely, if the operators ${\calQ}_{n_0}, {\ducalQ}_{n_0} : \ell^2(\nu) \to \ell^{\infty}(\nu)$ are bounded, then, by duality, 
\begin{gather*}
	\Vert\ducalQ_{n_0}\Vert_{\ell^1(\nu) \to \ell^2(\nu)} = \Vert\calQ_{n_0}\Vert_{\ell^2(\nu) \to \ell^{\infty}(\nu)} < \infty
	\quad\text{and}\quad
	\Vert\calQ_{n_0}\Vert_{\ell^1(\nu) \to \ell^2(\nu)} = \Vert\ducalQ_{n_0}\Vert_{\ell^2(\nu) \to \ell^{\infty}(\nu)} < \infty,
\end{gather*}
which means that aIUC($2n_0$) holds. 

If a semigroup fails to be aIUC, it is usually very hard to describe its regularity and to control its large-time behaviour, including the heat kernel. In Theorem \ref{aIUC} below we give necessary and sufficient conditions for aIUC to hold. It shows that this property is very restrictive as it requires a sufficiently fast growth of the potential at infinity. 

In the DSP case, however, we can introduce the following weaker smoothing property, which arises from our general estimates obtained in Theorem \ref{HK:estimates}. We call this new condition \emph{progressive intrinsic ultracontractivity} (pIUC,  for short)  mimicking the definition for continuous time semigroups of self-adjoint operators in \cite{Kaleta-Schilling-JFA}.

\begin{definition}[\bfseries pIUC]\label{def:piuc}
	The semigroup $\{\calU_n : n\in \nat_0 \}$ is called \emph{progressively intrinsically ultracontractive} (pIUC, for short), if there exists an exhausting sequence of sets $(A_n)_{n\in \nat}$ in $X$,  some $n_0 \in \nat$, and a constant $c>0$ such that 
	\begin{align} \label{eq:pIUC}
		u_{n}(x,y)\leq c\lambda_0^n\phi_0(x){\hatphi}_0(y), \quad  n \geq n_0, \; x\in A_n, \; y \in X. 
	\end{align}
\end{definition}
The definition of pIUC of the dual semigroup $\{\widehat{\calU}_n: n\in \nat_0\}$ is analogous. If we want to highlight the specific value of $n_0$, we write pIUC($n_0$).

The constant $c$ appearing in \eqref{eq:pIUC} does not depend on $x,y \in X$ or $n \geq n_0$, which means that the factorization property extends to larger and larger sets $A_n$ with respect to $x$. In that sense, regularity of the semigroup improves progressively as $n \to \infty$. Indeed, \eqref{eq:pIUC} is equivalent to the property that
\begin{gather*}
	\calQ_{n}: \ell^1(X,\nu) \to \ell^{\infty}(A_n,\nu) \text{\ \ is a bounded operator for every $n\geq n_0$}
\end{gather*}
and
\begin{gather*}
	\sup_{n \geq n_0} \Vert\calQ_{n}\Vert_{\ell^1(X,\nu) \to \ell^{\infty}(A_n,\nu)} < \infty.
\end{gather*}
Observe that aIUC always implies pIUC: indeed, consider the trivial exhaustion $(A_n)$ such that $A_{n_0}= A_{n_0+1}=\ldots = X$. 

We are going to show that in the present setting pIUC \emph{always} holds for \emph{any} increasing potential -- no matter how slowly $V$ grows at infinity. This is one of the main differences between pIUC and aIUC. Of course, the structure of the exhausting sequence $(A_n)_{n\in \nat}$ and the speed of propagation will depend on $V$.

\begin{theorem} \label{th:pIUC}
	Assume \eqref{A1}--\eqref{A3} and \eqref{B}. Let $(B_n)_{n\in \nat}$ be an arbitrary exhausting sequence consisting of finite sets such that $V$ is increasing with respect to $(B_n)_{n\in \nat}$. Set $B_\infty := X$, let $x_0\in X$ be a fixed reference point and let $C$ be the constant from estimate \eqref{outer_estimate} in Theorem \ref{HK:estimates}. For an arbitrary finite set $D \supseteq \left\{z\in X: \lambda_0\,V(z) \leq C\right\}$ we define 
	\begin{gather*}
		D_n 
		= D \cup \left\{z\in X: \lambda_0\,V(z) \geq C\left(\frac{1}{P(z,x_0)P(x_0,z)}\right)^{\frac{1}{n-1}} \right\}, \quad n \geq 2,
	\intertext{and}
		A_1 = B_1, \quad A_n = B_{l(n)}, \quad n \geq 2,
	\intertext{where}
		l(n):= \sup \left\{k \in \nat: B_k \subseteq D_n \right\}, \qquad \left(\sup \emptyset := 1\right). 
	\end{gather*}
	Then $(A_n)_{n\in \nat}$ is an exhausting sequence and there exists a constant $c>0$ such that
	\begin{align} \label{eq:pIUC_xy}
		u_{n}(x,y)\leq c\lambda_0^n\phi_0(x){\hatphi}_0(y),\quad \text{whenever\ \ } x\in A_n \text{\ or\ } y\in A_n, \; n \in \nat. 
	\end{align}
	Equivalently, the semigroups $\{\calU_n : n\in \nat_0\}$ and $\{\widehat{\calU}_n: n\in \nat_0\}$ are pIUC with respect to $(A_n)_{n\in \nat}$. 
\end{theorem}
\begin{proof}
Both $(B_n)_{n\in \nat}$ and $(D_n)_{n \geq 2}$ are exhausting  sequences in $X$, and all $B_n$'s are finite. Consequently, $l(n)$ increases to $\infty$ as $n \to \infty$, and $(A_n)_{n\in \nat}$ is exhausting as well. Furthermore, the monotonicity of $V$ with respect to $(A_n)_{n\in \nat}$ is inherited from $(B_n)_{n\in \nat}$.

By Theorem \ref{HK:estimates}, there is a constant $c_1>0$ such that for any finite set $B$ and for all $x,y\in X$,
\begin{align} \label{eq:aux_piuc}
	u_n(x,y) 
	\leq c_1\left( \lambda_0^n \phi_0(x)\widehat{\phi}_0(y) + \frac{C^n}{V(x)}\sum_{z\in B^c\cap A_{\alpha(x)}\cap A_{\alpha(y)}} p(x,z)p(z,y)\frac{\mu(z)}{V(z)^{n-1}}\right).
\end{align}
Set $B=A_1 \cup D$ and suppose that $x\in A_n$. If $n=1$, then the assertion follows trivially from \eqref{eq:aux_piuc} as the sum on the right hand side disappears. Assume that $n \geq 2$. Clearly, $A_{\alpha(x)}\subseteq A_n \subseteq D_n$. Hence, by the definition of the set $D_n$,
\begin{align*}
	u_n(x,y) 
	&\leq c_1\left( \lambda_0^n \phi_0(x)\widehat{\phi}_0(y) + \frac{C^n}{V(x)}\sum_{z\in B^c\cap A_n} p(x,z)p(z,y)\frac{1}{V(z)^{n-1}}\mu (z)\right)\\
	&\leq c_1\left( \lambda_0^n \phi_0(x)\widehat{\phi}_0(y) + \frac{C\,\lambda_0^{n-1}}{V(x)}\sum_{z\in B^c \cap D_n} P(x,z)P(z,x_0)P(x_0,z) P(z,y) \mu(y)^{-1}\right)\\
	&\leq c_1 \lambda_0^n \phi_0(x)\widehat{\phi}_0(y) + c_2\,\lambda_0^{n}\frac{p(x,x_0)}{V(x)}p(x_0,y) \leq c_3 \lambda_0^n \phi_0(x)\widehat{\phi}_0(y),
\end{align*}
where we use the DSP property \eqref{A1} first to get $P(x,z)P(z,x_0)\leq C_*P(x,x_0)$ and  again for  $\sum_{z} P(x_0,z)P(z,y)\leq C_*P(x_0,y)$; then we combine this estimate with \eqref{phi_0_bound} and \eqref{phi_hat_bound}. If $y \in A_n$, then $A_{\alpha(y)}\subseteq A_n$ and the  same argument can be used.  The proof is finished.
\end{proof}

We  will  now establish a necessary and sufficient condition for aIUC to hold. The form of this condition is suggested by the structure of the sets $A_n$ appearing in Theorem \ref{th:pIUC}. 

\begin{theorem}\label{aIUC}
	Assume \eqref{A1}--\eqref{A3} and \eqref{B}. Let $n_0 \in \left\{2,3,\dots\right\}$.  The following conditions are equivalent: 
	\begin{enumerate}
		\item\label{aIUC-i} the semigroup $\{\calU_n : n\in \nat_0\}$ is \textup{aIUC($n_0$)};
		\item\label{aIUC-ii} the semigroup $\{\widehat{\calU}_n: n\in \nat_0\}$ is \textup{aIUC($n_0$)};
		\item\label{aIUC-iii} for every fixed $x_0 \in X$ there is a constant $C=C(n_0, x_0)>0$ such that 
	\begin{gather}\label{eq:V}
		\frac{C}{V(x)^{n_0-1}}\leq  P(x,x_0)P(x_0,x), \quad x \in X.
	\end{gather}
\end{enumerate}
\end{theorem}

\begin{proof}
The equivalence of \ref{aIUC-i} and \ref{aIUC-ii} is just \eqref{eq:equiv_aiuc}. We are going to show \ref{aIUC-i}$\Rightarrow$\ref{aIUC-iii}. Fix an arbitrary $x_0 \in X$ and note that \ref{aIUC-i} and the estimates \eqref{phi_0_bound}, \eqref{phi_hat_bound} imply that there is a constant $c_1=c_1(n_0,x_0)>0$ such that
\begin{gather}\label{eq:rightside}
	u_{n_0}(x,x)\mu(x)\leq \frac{c_1}{V(x)} P(x,x_0)P(x_0,x), \quad x \in X. 
\end{gather}
On the other hand, we have
\begin{align*}
	u_{n_0}(x,x) \mu(x)
	&= \frac{1}{V(x)}\sum_{z_1,\ldots,z_{n_0-1}\in X} p(x,z_1)\frac{\mu(z_1)}{V(z_1)}p(z_1,z_2)\ldots\frac{\mu(z_{n_0-1})}{V(z_{n_0-1})}p(z_{n_0-1},x) \mu(x)\\
	&\geq \frac{(p(x,x)\mu(x))^{n_0}}{V(x)^{n_0}} \\
	&\geq \frac{C_{-}^{n_0}}{V(x)^{n_0}},
\end{align*}
where we use the assumption \eqref{A2}. Combining the above inequality with \eqref{eq:rightside} immediately gives \ref{aIUC-iii}. 

Now we show the implication \ref{aIUC-iii}$\Rightarrow$\ref{aIUC-i}. By \eqref{eq:estm}, there is a constant $c_2=c_2(n_0)>0$ such that 
\begin{gather*}
	u_{n_0}(x,y)
	\leq\frac{c_2}{V(x)}\sum_{z\in X}p(x,z)\frac{p(z,y)}{V(z)^{n_0-1}}\mu(z).
\end{gather*}
With \eqref{eq:V} and two applications of the DSP condition \eqref{A1} (similar to the last lines of the proof of Theorem \ref{th:pIUC}, but in the form \eqref{eq:DSP}), combined with \eqref{phi_0_bound} and \eqref{phi_hat_bound}, we get
\begin{align*}
	u_{n_0}(x,y)
	&\leq c_3\sum_{z\in X}\frac{p(x,z)p(z,x_0)}{V(x)}p(x_0,z)p(z,y)\mu(z)\\
	&\leq c_4\phi_0(x)\widehat{\phi}_0(y),
\end{align*}
for some constants $c_3, c_4>0$ depending only on $n_0$ and $x_0$.
\end{proof}

We will need a refined version of aIUC given by the \emph{asymptotic ground state domination} (aGSD, for short). 
\begin{definition} 
	 Let $\lambda_0$ and $\phi_0$ be the ground state eigenvalue and eigenfunction of the semigroup $\{\calU_n : n\in \nat_0\}$. The semigroup $\{\calU_n : n\in \nat_0\}$ is \emph{asymptotically ground state dominated}, aGSD($n_0$), if there exist $n_0 \in \nat$ and a constant $C=C(n_0)>0$ such that
	\begin{align}\label{eq:aGSD}
		\calU_{n_0}\one(x)\leq C \phi_0(x),\quad x\in X.
	\end{align}
\end{definition}
A similar definition applies for $\{\widehat{\calU}_n: n\in \nat_0\}$, with $\phi_0$ replaced by $\widehat{\phi}_0$.

Since $\phi_0$ is an eigenfunction, we can iterate aGSD($n_0$) and get
\begin{align}\label{eq:aGSD-ext}
	\calU_{n}\one(x)\leq \widetilde C \lambda_0^{n} \phi_0(x),\quad x\in X, \; n \geq n_0,
\end{align}
where $\widetilde C = C/\lambda_0^{n_0}$.   

\begin{proposition}\label{prop:aIUCandaGSD}
Assume \eqref{A1}--\eqref{A3} and \eqref{B}. Then the following implications hold:
\begin{enumerate}
	\item\label{prop:aIUCandaGSD-1} 
		\textup{aIUC($n_0$)} $\implies \{\calU_n : n\in \nat_0\}$ and $\{\widehat{\calU}_n: n\in \nat_0\}$ are \textup{aGSD($n_0$)};
	\item\label{prop:aIUCandaGSD-2} 
		$\{\calU_n : n\in \nat_0\}$ and $\{\widehat{\calU}_n: n\in \nat_0\}$ are \textup{aGSD($n_0$)} $\implies$ \textup{aIUC($2n_0+1$)}.
\end{enumerate}
\end{proposition}

\begin{proof}
In order to show \ref{prop:aIUCandaGSD-1} we observe that by \eqref{phi_0_bound}, \eqref{phi_hat_bound} we have $\phi_0, \widehat{\phi}_0 \in \ell^1(\mu)$ and
\begin{gather*}
	\calU_{n_0}\one(x) 
	= \sum_{y\in X}u_{n_0}(x,y)\mu(y) \leq c \phi_0(x)\sum_{y\in X}\widehat{\phi}_0(y)\mu(y)
	\leq c \Vert \widehat{\phi}_0\Vert_{\ell^1(\mu)}\phi_0(x),\\
	\ducalU_{n_0}\one(x) 
	= \sum_{y\in X}u_{n_0}(y,x)\mu(y) \leq c \widehat{\phi}_0(x)\sum_{y\in X}\phi_0(y)\mu(y)
	\leq c \Vert \phi_0\Vert_{\ell^1(\mu)}\widehat{\phi}_0(x).
\end{gather*}

\noindent
For \ref{prop:aIUCandaGSD-2} we can use the semigroup property and \eqref{potential_lower_bound}, \eqref{p_sup_cond} and write
\begin{align*}
	u_{2n_0+1}(x,y)
	&\leq\sum_{z, w \in X}u_{n_0}(x,z)\mu(z)\frac{p(z,w)}{V(z)} u_{n_0}(w,y)\mu(w) \\
    &\leq c_1 \calU_{n_0} \, \one(x) \ducalU_{n_0}\one(y) \\
    &\leq c_2 \phi_0(x)\widehat{\phi}_0(y).
\qedhere
\end{align*}
\end{proof}

\subsection{Equivalence of asymptotic intrinsic ultra- and hypercontractivity}

Our next goal is to show that under the DSP property aIUC is equivalent to intrinsic hypercontractivity. This fact has already been observed in the case of continuous (in time and space) Feynman--Kac semigroups corresponding to L\'{e}vy generators in \cite{Kaleta-Kwasnicki-Lorinczi-JST} and \cite{Kaleta-Schilling-JFA}; in general this equivalence is not true. 

In the present purely discrete setting we use the following definition. 
\begin{definition} \label{def:ihc}
	The semigroup $\{\calU_n : n\in \nat_0\}$ is called \emph{intrinsically hypercontractive} (IHC,  for  short), if the corresponding intrinsic semigroup $\{\calQ_n : n\in \nat_0\}$ is hypercontractive, i.e.\ for every $1 < p < q < \infty$ 
	\begin{gather*}
		\text{there exists $n=n_{p,q} \in \nat$ such that $\calQ_{n}: \ell^p(X,\nu) \to \ell^{q}(X,\nu)$ is a bounded  operator.} 
	\end{gather*}
\end{definition}
It is easy to see that the boundedness of $\calQ_{n}$ extends to \emph{every} $n \geq n_{p,q}$. Using a standard duality argument, we see that $\widehat\calQ_n : (\ell^q)^*(X,\nu) \simeq \ell^{q'}(X,\nu)\to(\ell^p)^*(X,\nu)\simeq \ell^{p'}(X,\nu)$ where $1 < q' = \frac{q}{q-1} < \frac{p}{p-1}=p'<\infty$. Consequently, 
\begin{gather*}
	\{\calU_n : n\in \nat_0\} \text{\ \ is IHC} \iff \{\widehat{\calU}_n: n\in \nat_0\} \text{\ \ is IHC}.
\end{gather*}
Moreover, because of \eqref{eq:inclusions-1},
\begin{gather*}
	\text{aIUC} \implies \text{IHC}
\end{gather*}
and, using the fact that $\calQ_n$ and $\widehat\calQ_n$ are dual in  $\ell^2(X,\nu)$, 
IHC is equivalent to the following property:
\begin{gather*}
\text{for every $p \in (2,\infty)$ there is $n=n_p \in \nat$ such that ${\calQ}_{n}, {\ducalQ}_{n} : \ell^2(X,\nu) \to \ell^p(X,\nu)$ are bounded. }
\end{gather*}
Our next proposition provides a necessary condition for IHC. 

\begin{proposition}\label{th:IHC}
	Assume \eqref{A1}--\eqref{A3} and \eqref{B}. Let $p \in (2,\infty)$ and $n \in \nat$. If the operator $\calQ_n: \ell^2(\nu) \to \ell^p(\nu)$ is bounded, then there exists a constant $c=c(p)>0$ such that
	\begin{align}\label{cond:V-for-AIUC}
		\frac{c}{V(x)^{\frac{2pn}{p-2}}}\leq \phi_0(x)\hatphi_0(x)\mu(x),\quad x\in X.
	\end{align} 
\end{proposition}
\begin{proof}
	We fix $p\in (2,\infty)$ and $n\geq 1$ such that the operator $\calQ_n: \ell^2(\nu) \to \ell^p(\nu)$ is bounded. If we evaluate $\calQ_n$ at the function $y\mapsto\one_{\{x\}}(y)$
	we obtain that for a constant $c_1>0$
	\begin{align}\label{eq:p-norm-1}
		\sum_{y \in X}|\calQ_n\,\one_{\{x\}}(y)|^p\nu(y) 
		= \Vert \calQ_n\,\one_{\{x\}} \Vert_{\ell^p(\nu)}^p 
		\leq c_1(\phi_0(x)\hatphi_0(x)\mu(x))^{\frac{p}{2}}.
	\end{align}
	By the definition of $\calQ_n$ and the fact that the expression under the sum is nonnegative, the left-hand side of \eqref{eq:p-norm-1} is equal to 
	\begin{align*}
		\sum_{y\in X} \left|\sum_{z \in X}\frac{u_n(y,z)}{\lambda_0^{n}\phi_0(y)\hatphi_0(z)}\,\one_{\{x\}}(z)\nu(z)\right|^p\nu(y)
		=\sum_{y \in X}\left(\frac{u_n(y,x)}{\lambda_0^{n}\phi_0(y)\hatphi_0(x)}\,\nu(x)\right)^p\nu(y).
	\end{align*} 
	If we keep in the sum only the term with $y=x$, then we obtain
	\begin{align*}
		\sum_{y \in X}\left(\frac{u_n(y,x)}{\lambda_0^n\phi_0(y)\hatphi_0(x)}\nu(x)\right)^p\nu(y) 
		&\geq \left(\frac{u_n(x,x)}{\lambda_0^n\phi_0(x)\hatphi_0(x)}\nu(x)\right)^p\nu(x)\\
		&=\left(\frac{u_n(x,x)}{\lambda_0^n}\right)^p\phi_0(x)\hatphi(x)\mu(x)^{p+1},
	\end{align*}
	where the equality follows from the definition of the measure $\nu$. Combining this with \eqref{eq:p-norm-1} yields
	\begin{gather*}
		\left(\frac{u_n(x,x)}{\lambda_0^n}\right)^p\phi_0(x)\hatphi(x)\mu(x)^{p+1}
		\leq c_1(\phi_0(x)\hatphi_0(x)\mu(x))^{\frac{p}{2}}.
	\end{gather*}
	If we set $c_2=\lambda_0^{np}\,c_1$, then
	\begin{gather*}
		u_n(x,x)^p\leq c_2(\phi_0(x)\hatphi_0(x))^{\frac{p}{2}-1}\mu(x)^{-\frac{p}{2}-1},
	\intertext{which implies}
		u_n(x,x)\mu(x)\leq c_2^{\frac{1}{p}}(\phi_0(x)\hatphi_0(x)\mu(x))^{\frac{1}{2}-\frac{1}{p}}.
	\end{gather*}
	The estimate \eqref{u_n-lower-bound} together with the assumption \eqref{A2} imply that there is a constant $c_3>0$ such that
	\begin{align*}
		u_n(x,x)\mu(x) 
		&\geq \frac{c_3^{n-2}}{V(x)}\sum_{y\in X}p(x,y)\frac{\mu(y)}{V(y)^{n-1}}p(y,x)\mu(x)\\
		&\geq \frac{c_3^{n-2}}{V(x)^n}\big(p(x,x)\mu(x)\big)^2
		\geq \frac{c_3^{n-2}C_-^2}{V(x)^n}.
	\end{align*}
	Finally, we conclude that
	\begin{gather*}
		\frac{1}{V(x)^{n\cdot\frac{2p}{p-2}}} 
		\leq c_4\phi_0(x)\hatphi_0(x)\mu(x),
	\end{gather*}
	for some $c_4=c_4(n,p)>0$ and all $x\in X$. This completes the proof.
\end{proof}

We can now derive the final result of this section.

\begin{theorem} \label{thm:aiuc-ihp}
	Under \eqref{A1}--\eqref{A3} and \eqref{B} the aIUC and IHC properties are equivalent.
\end{theorem}
\begin{proof}
	It is clearly enough to show that IHC implies aIUC. This is now easy, as by Proposition \ref{th:IHC} the IHC property implies \eqref{cond:V-for-AIUC}, for some $p >2$ and $n \in \nat$. In view of the estimates \eqref{phi_0_bound}, \eqref{phi_hat_bound} and Theorem \ref{aIUC}, this leads to aIUC($n_0$) for sufficiently large $n_0 \in \nat$.
\end{proof}

\section{Nearest-neighbour random walks}\label{sec:NNRW}
In this section we consider finite range nearest-neighbour random walks on a graph with vertex set $X$. We obtain general sufficient conditions under which the associated Feynman--Kac semigroups do \textbf{not  enjoy}   the aIUC or IHC properties.  

We start by defining a graph structure over the state space $X$. The graph $G=(X,E)$ over $X$ (the points in $X$ form the set of \emph{vertices}) is defined by specifying the set of \emph{edges} $E \subset \left\{\{x,y\}: x,y \in X \right\}$: Two vertices $x, y \in X$ are connected by an edge in $G$ if, and only if, $\{x,y\} \in E$. In this case we call $x$ and $y$ \emph{neighbours} and write $x \sim y$; note that $\{x,y\} = \{y,x\}$, so $x\sim y\iff y\sim x$. 
The graph $G$
is of \emph{finite geometry} if $\# \left\{y \in X: x \sim y \right\} < \infty$, for all $x \in X$, i.e.\ the number of neighbours of an arbitrary vertex $x \in X$ is finite. The graph $G$ is \emph{connected}, if for all $x,y \in X$, $x \neq y$, there exists a sequence $\left\{x_i\right\}_{i=0}^n \subset X$ with $x_0=x$, $x_n=y$ such that $x_{i-1} \sim x_i,$ for $i = 1,\ldots,n$. This means that any two different vertices $x$ and $y$ are connected by a path in $G$. The \emph{length} of the path is the number of edges belonging to that path. Every shortest path connecting two different vertices $x$ and $y$ is called a \emph{geodesic path} between $x$ and $y$. For the rest of this section we assume that $G$ is a connected graph of finite geometry.
The  fact  that $G$ is a connected graph allows us to define the natural \emph{graph} (\emph{geodesic}) \emph{distance} $d$ in $X$.
More precisely, $d(x,y)$ is the length of the geodesic path connecting $x$ and $y$. As $G$ is of  finite geometry, every open geodesic ball $B_r(x) = \left\{y \in X: d(x,y) < r\right\}$ is finite, and since $X$ is infinite, the metric space $(X,d)$ is unbounded. 

We consider a probability kernel $P:X \times X \to [0,\infty)$ such that for every two vertices $x, y \in X$,
\begin{align} \label{eq:nnrw_def}
	P(x,y)>0 \iff x \sim y .
\end{align}
The corresponding $P$-Markov chain is called a \emph{nearest-neighbour random walk} on the graph $G$. We will need an additional regularity assumption on the kernel $P(x,y)$, which coincides with the so-called $p_0$-condition imposed in \cite{Grigoryan-Telcs} (cf.\ \cite[Definition 2.1.1]{Kumagai}):
\begin{align} \label{eq:nnrw_add_ass} 
	\inf \left\{P(x,y): x,y \in X, \ x \sim y\right\} > 0.
\end{align}
We also assume that there exists a probability kernel $\Phat \colon X \times X \to [0,1]$ and a measure $\mu:X \to (0,\infty)$ such that the duality relation \eqref{cond:duality} holds.

In this section, we restrict our attention to the class of isotropic and increasing potentials $V$, i.e.\ we assume that there exists some $x_0 \in X$ such that 
\begin{align} \label{eq:V_isotropic} 
	V(x) = V(y), \text{\ \ if}\ d(x,x_0) = d(y,x_0),
	\quad\text{and}\quad 
	V(x) \geq V(y), \text{\ \ if}\ d(x,x_0) \geq d(y,x_0).
\end{align}

We are going to show that under \eqref{eq:nnrw_add_ass} a nearest-neighbour random walk does not enjoy the aIUC property. 
\begin{proposition} \label{prop:noaiuc}
	 Let $(X,E)$ be a  connected graph of finite geometry and $P$ a transition kernel as described above. Denote 
	by $\{\calU_n: n\in\nat_0\}$ the Feynman--Kac semigroup where $V$ is a confining potential satisfying \eqref{eq:V_isotropic}. If the ground state of the dual semigroup $\widehat{\phi}_0\in \ell^1(X,\mu)$, and if $P$ satisfies \eqref{eq:nnrw_add_ass}, then $\{\calU_n: n\in\nat_0\}$ does not have the aIUC property. 
\end{proposition}
\begin{proof}
	Suppose, to the contrary, that there exist $n_0 \in \nat$ and a constant $c=c(n_0)>0$ such that
	\begin{gather}\label{AIUC-local}
		u_{n_0}(x,y) \leq c\, \phi_0(x)\,\widehat{\phi}_0(y),\quad x,y\in X.
	\end{gather}
	Since $\lambda_0^{-1}\calU\phi_0 = \phi_0$, the ground state $\phi_0$ is harmonic with respect to the operator $\calU^{\lambda_0} - \id:= \lambda_0^{-1}\calU - \id$ (this amounts to  changing the potential $V\rightsquigarrow \lambda_0^{-1}V$), we can apply \cite[Theorem 2.7]{CKS-ALEA}. Together with \eqref{AIUC-local} and  the assumption $\widehat{\phi}_0\in \ell^1(\mu)$ we see that for some $c_1>0$,	
	\begin{gather}\label{eq:contr1}
		\calU_{n_0}\one(x)  
		= \sum_{y\in X}u_{n_0}(x,y)\mu (y) 
		\leq c\Vert \widehat{\phi}_0\Vert_{\ell^1 (\mu)}\phi_0(x) 
		\leq c_1 \prod_{k=0}^n \frac{1}{V(x_k)},
	\end{gather}
	for any path $x_0\to \dots \to x_n=x$ that connects a fixed reference point $x_0\in X$ with $x$. By \eqref{eq:kernel_n} we obtain for a constant $c_2>0$, 
	\begin{align*}
		\calU_{n_0}\one(x) 
		&= \sum_{y\in X} \sum_{z_1,\ldots,z_{n_0-1}\in X} \frac{p(x,z_{1})\mu (z_1)}{V(x)V(z_1)}\cdot \ldots \cdot
			\frac{p(z_{n_0-2},z_{n_0-1})\mu(z_{n_0-1})}{V(z_{n_0-1})} p(z_{n_0-1},y)\mu(y) \\
		&=\sum_{z_1,\ldots,z_{n_0-1}\in X} \frac{p(x,z_{1})\mu (z_1)}{V(x)V(z_1)}\cdot \ldots \cdot
			\frac{p(z_{n_0-2},z_{n_0-1})}{V(z_{n_0-1})} \mu(z_{n_0-1}) \\
		&\geq \frac{p(x,x_{n-1}) \mu(x_{n-1})}{V(x)V(x_{n-1})}\cdot \ldots \cdot
			\frac{p(x_{n-n_0+2},x_{n-n_0+1})\mu(x_{n-n_0+1})}{V(x_{n-n_0+1})}\\
		&\geq c_2^{n_0-1}\prod_{k=n-n_0+1}^{n}\frac{1}{V(x_k)}, 
	\end{align*}
	where the second equality follows  from  Tonelli's theorem and the fact that $P$ is a probability kernel, the first inequality is just a reduction of the sum to the path $z_1=x_{n-1},\ldots ,z_{n_0 -1} = x_{n-n_0+1}$, and the last inequality follows with \eqref{eq:nnrw_add_ass}. Combining \eqref{eq:contr1} with the inequality above yields
	\begin{gather*}
		0<\frac{c_2^{n_0-1}}{c_1}\leq  \prod_{k=0}^{n-n_0}\frac{1}{V(x_k)}.
	\end{gather*}
	If we let $n$ tend to infinity, we arrive at a contradiction as $V$ is a confining potential which  means that  the right-hand side converges to zero. 
\end{proof}

Our next aim is to show that for some class of discrete Feynman--Kac semigroups related to nearest-neighbour random walks satisfying \eqref{eq:nnrw_add_ass} the IHC property typically fails to hold. We first observe that the condition \eqref{eq:V_isotropic} can be equivalently stated in the following way: There exist $x_0 \in X$ and an increasing profile function $W:\nat_0 \to (0,\infty)$ such that $V(x) = W(d(x_0,x))$, for any $x \in X$. We will assume that
\begin{gather}\label{asm:cmp}
	W(k) \asymp W(k+1)\text{\ \ for large $k$, and\ \ } \lim_{k\to\infty} W(k) = \infty,
\end{gather} 
which in turn implies that $V$ is confining.  We will also need the condition that the measure $\mu$ is bounded, that is
\begin{gather}\label{walk_mu}
	\sup_{x \in X} \mu(x)<\infty.
\end{gather}

\begin{proposition} \label{prop:noihc}
	 Let $(X,E)$ be a  connected graph of finite geometry and $P$ a transition kernel as described above. 
	Under the assumptions \eqref{eq:nnrw_add_ass}, \eqref{asm:cmp} and \eqref{walk_mu}, the Feynman--Kac semigroup $\{\calU_n:n\in \nat_0\}$ related to the kernel $P$ does not enjoy the IHC property.
\end{proposition}

\begin{proof}
Suppose, to the contrary, that $\{\calU_n:n\in \nat_0\}$ satisfies IHC property, that is for some $p>2$ there exists $n_0\in \nat$ and a constant $c=c(n_0)>0$ such that
\begin{gather}\label{bnd_cnd}
	\Vert \calQ_{n_0}\,f \Vert_{\ell^p(\nu)}  \leq c\, \Vert f \Vert_{\ell^2(\nu)}.
\end{gather}
The measure $\nu(x)$ is given by \eqref{eq:v-measure}. We fix a reference point $x_0\in X$.
Let $z_0\in X$ be a point such that $d(z_0,x_0)=n>n_0$ and let $z_1\in X$ be such that $z_0\sim z_1$ and $d(z_1,x_0)=n+1$. The existence of $z_0$ for each $n>n_0$ is due to the structural assumptions on the graph $G$. 
Taking $f=\one_{\{z_0\}}$, the inequality \eqref{bnd_cnd} yields
\begin{gather}\label{eq:RS}
	\Vert \calQ_{n_0}\,\one_{\{z_0\}} \Vert_{\ell^p(\nu)}^p 
	\leq c\,(\phi_0(z_0)\hatphi_0(z_0)\mu(z_0))^{\frac{p}{2}}.
\end{gather}
On the other hand, we have
\begin{gather}\label{eq:LS}\begin{split}
	\Vert \calQ_{n_0}\,\one_{\{z_0\}} \Vert_{\ell^p(\nu)}^p 
	&=\sum_{x\in X} \left|\sum_{y \in X} \frac{u_{n_0}(x,y)}{\lambda_0^{n_0}\phi_0(x)\hatphi_0(y)}\, \one_{\{z_0\}}(y)\phi_0(y)\hatphi_0(y)\mu(y)\right|^p\,\phi_0(x)\hatphi_0(x)\mu(x)\\
	&=\sum_{x\in X} \left|\frac{u_{n_0}(x,z_0)}{\lambda_0^{n_0}\phi_0(x)}\,\phi_0(z_0)\mu(z_0)\right|^p\,\phi_0(x)\hatphi_0(x)\mu(x)\\
	&\geq \left|\frac{u_{n_0}(z_0,z_0)}{\lambda_0^{n_0}}\,\mu(z_0)\right|^p\,\phi_0(z_0)\hatphi_0(z_0)\mu(z_0)\\
	&=\frac{1}{\lambda_0^{p\,n_0}}\left(u_{n_0}(z_0,z_0)\mu(z_0)\right)^p\,\phi_0(z_0)\hatphi_0(z_0)\mu(z_0).
\end{split}\end{gather}
Combining the estimates \eqref{eq:RS} and \eqref{eq:LS}, and rearranging the terms yields 
\begin{gather*}
	\frac{1}{\lambda_0^{p\,n_0}}\left(u_{n_0}(z_0,z_0)\mu(z_0)\right)^p
	\leq c\,(\phi_0(z_0)\hatphi_0(z_0)\mu(z_0))^{\frac{p}{2}-1} .
\end{gather*}
Because of \eqref{walk_mu}, there is a constant $c_1=c_1(n_0)>0$ such that
\begin{gather}\label{halfway}
	\frac{1}{\lambda_0^{p\,n_0}}\left(u_{n_0}(z_0,z_0)\mu(z_0)\right)^p 
	\leq c_1\,(\phi_0(z_0)\hatphi_0(z_0))^{\frac{p}{2}-1}.
\end{gather}
In the next step we use the fact that ground state eigenfunctions are harmonic (with respect to $\calU^{\lambda_0} - \id := \lambda_0^{-1}\calU - \id$ and the dual operator), and we apply \cite[Theorem 2.7]{CKS-ALEA} to the geodesic path from $x_0$ to $ z_0$ of length $n$. This implies that, for a constant $c_2>0$, 
\begin{gather}\label{UB-V}
	(\phi_0(z_0)\hatphi_0(z_0))^{\frac{p}{2}-1} \leq c_2\left(\prod_{k=0}^{n} \frac{1}{V(x_k)}\right)^{p-2}.
\end{gather}
 We  can find a lower bound for the expression $u_{n_0}(z_0,z_0)\mu(z_0)$ by restricting the  $n_0$-fold sum appearing in the definition of $u_{n_0}(z_0,z_0)$  to a single path. Without loss of generality, we may assume  that $n_0$ is even. We consider the path $z_0 \to z_1 \to z_0 \to z_1 \to \ldots \to z_0$. By  \eqref{eq:nnrw_add_ass}, \eqref{walk_mu} and \eqref{asm:cmp}, we have
\begin{gather}\label{LB-V} \begin{split}
	u_{n_0}(z_0,z_0) 
	&= V(z_0)^{-1} \sum_{z_{1},\ldots ,z_{n_0-1}\in X} \frac{p(z_0,z_{1})\mu(z_1)}{V(z_{1})} \dots \frac{p(z_{n_0-2}, z_{n_0-1})\mu(z_{n_0-1})}{V(z_{n_0-1})}p(z_{n_0-1}, z_0)\\
	&\geq \left( p(z_0,z_1)\mu(z_1)\right)^{n_0/2}  V(z_0)^{-n_0/2}\left( p(z_1,z_0)\mu(z_0)\right)^{n_0/2}V(z_{1})^{-n_0/2}\mu(z_0)^{-1}\\
	&= P(z_0,z_1)^{n_0/2}   P(z_1,z_0)^{n_0/2}  V(z_0)^{-n_0/2}V(z_{1})^{-n_0/2}\mu(z_0)^{-1}  \\
	&\geq c_3^{n_0}V(z_0)^{-n_0}\mu(z_0)^{-1},
\end{split}\end{gather}
for some constant $c_3>0$. Combining \eqref{halfway}, \eqref{UB-V} and \eqref{LB-V} yields
\begin{gather*}
	\left(\prod_{k=0}^{n} V(x_k)\right)^{p-2}\left(\frac{1}{V(z_0)}\right)^{n_0\,p}\leq c_4,
\end{gather*}
for a constant $c_4=c_4(n_0)>0$. We arrive at
\begin{gather*}
	\left(\prod_{k=0}^{n-\left\lceil\frac{n_0\,p}{p-2}\right\rceil-1} V(x_k)\right)^{p-2}\left(\prod_{k=n-\left\lceil\frac{n_0\,p}{p-2}\right\rceil}^n V(x_k)\right)^{p-2}\left(\frac{1}{V(z_0)}\right)^{n_0\,p}
	\leq c_4.
\end{gather*}
The first factor on the left-hand side can become arbitrarily large, while the product of the two remaining factors is bounded away from zero, due to the  assumption \eqref{asm:cmp}. Thus, we have reached a contradiction.
\end{proof}


\section{Ergodicity and quasi-ergodicity}\label{sec:ergodic}

In this section we study the ergodic behaviour of the intrinsic semigroup $\{\calQ_n:n\in \nat_0\}$ (and its dual semigroup $\{\ducalQ_n:n\in \nat_0\}$) and the related quasi-ergodic behaviour of the discrete Feynman--Kac semigroup $\{\calU_n:n\in \nat_0\}$ (and its dual semigroup $\{\ducalU_n:n\in \nat_0\}$) with confining potentials in the DSP framework. We also investigate connections between ergodicity properties and the regularizing effects of the semigroups, which we have considered in previous sections. 

\subsection{Stationary and quasi-stationary measures}
Let $\nu$ be the measure given by \eqref{eq:v-measure}. In this section, we will use the following shorthand notation
$\ell^p (\nu) = \ell^p (X,\nu)$, $1 \leq p \leq \infty$ and 
\begin{align*}
	\|\nu\| = \Vert \nu \Vert_{\mathrm{TV}} = \sum_{x\in X}\nu(x),\qquad 
	\nubar (f)=\frac{1}{\|\nu\|}\sum_{y\in X}f(y)\nu(y).
\end{align*} 
Recall that the measure $\nubar$ is the stationary probability measure for the intrinsic semigroups, i.e.\ $\nubar(\calQ_n f) =\nubar(f)$ and $\nubar(\ducalQ_n f) =\nubar(f)$, for any $f\in \ell^1 (\nu)$ and $n \in \nat_0$. 

As before, we will  assume that the ground state eigenfunctions $\phi_0$ and $\hatphi_0$ are normalized, i.e.\ condition \eqref{ass:normalization} is satisfied. Combining the estimates \eqref{phi_0_bound}--\eqref{phi_hat_bound} with our assumption \eqref{B}, the inequality \eqref{eq:sub-invar} shows that $\| \phi_0\|_{\ell^1(\mu)}<\infty$ and $\|\hatphi_0\|_{\ell^1(\mu)} <\infty$.  This allows us to introduce the following probability measures
\begin{gather}\label{eq:measures-ergodic}
	m(E) := \frac{1}{\left\|\hatphi_0\right\|_{\ell^1(\mu)}}\sum_{y\in E}\hatphi_0(y)\mu(y) 
	\quad \text{and} \quad 
	\dum(E) := \frac{1}{\left\|\phi_0\right\|_{\ell^1(\mu)}}\sum_{y\in E}\phi_0(y)\mu(y), \quad E \subset X.
\end{gather}
These measures are \textit{quasi-stationary} measures of the semigroups $\{\calU_n:n\in \nat_0\}$ and  $\{\ducalU_n:n\in \nat_0\}$, respectively. This means that for any $n\in \nat_0$
\begin{align*}
\frac{m(\calU_n f)}{m(\calU_n\bbjedan)} = m( f), \quad f\in \ell^1 (m), \quad \text{and} \quad \frac{\dum(\ducalU_n f)}{\dum(\ducalU_n\bbjedan)} = \dum( f), \quad f\in \ell^1 (\dum),
\end{align*}
or, equivalently,
\begin{align*}
	m(\calU_n f) = \lambda_0^n m(f),\quad f\in\ell^1(m) 
	\quad\text{and}\quad
	\dum(\ducalU_n f) = \lambda_0^n \dum(f), \quad f\in\ell^1(\dum).
\end{align*} 
This is due to the fact that the discrete Feynman--Kac semigroups need not be conservative, so we have to normalize their action (see e.g.\ \eqref{kappa_U}) when studying their ergodic behaviour; this finally results in a quasi-ergodic behaviour of the semigroups.

Since $\nu$ is a finite measure, the corresponding $\ell^p$ spaces satisfy natural (strict) inclusions, i.e.\ one has
\begin{align} \label{eq:inclusions}
\ell^{\infty}(\nu) \subsetneq \ell^p(\nu) \subsetneq \ell^q(\nu) \subsetneq \ell^1(\nu), \quad \text{whenever} \quad 1 < q < p < \infty,
\end{align}
and the same is true for the $\ell^p$-spaces of the probability measures $m, \dum$. We will analyse the uniform ergodic and quasi-ergodic properties on various $\ell^p$ spaces, mainly $\ell^{\infty}$ and $\ell^1$, and explain how they depend on the regularity of the semigroups.

\subsection{Uniform ergodicity of intrinsic semigroups on $\ell^\infty (\nu)$ } \label{sec:unif_ergod}
We start our discussion with the smallest space:\ $\ell^\infty(\nu)$. Recall that the Markov semigroup $\{\calQ_n:n\in \nat_0\}$ is called \emph{uniformly ergodic} on $\ell^\infty (\nu)$ (or simply \emph{uniformly ergodic} in the sense of the definition in \cite[p.\ 393]{Tweedy}) if  there is  a stationary probability measure $\pi$ on $X$ such that
\begin{align}\label{Q-unif-ergodic}
	\lim_{n\to \infty} \sup_{x\in X}\sup_{\Vert g\Vert_{\ell^\infty (\nu)}\leq 1}
	\left\vert \calQ_ng(x) - \pi (g) \right\vert =0.
\end{align} 
The condition \eqref{Q-unif-ergodic} can be equivalently formulated as
\begin{gather*}
	\lim_{n\to\infty}\sup_{x\in X}\Vert \calQ_n(x,\cdot)-\pi\Vert_{\mathrm{TV}} = 0.
\end{gather*}

We will show that under the assumptions \eqref{A1}, \eqref{A3} and \eqref{B} the intrinsic semigroups $\{\calQ_n:n\in \nat_0\}$, $\{\ducalQ_n:n\in \nat_0\}$ are uniformly ergodic on $\ell^\infty (\nu)$ and that the rate of convergence to their quasi-stationary measures is geometric. It turns out that in this case the key assumption is the DSP property of the free kernels $P(x,y)$,  while  the potential $V$ is  may be  an arbitrary confining potential, and it does not matter how slowly $V$ grows.

In order to show uniform ergodicity with a geometric rate of convergence, we will apply \emph{Doeblin's condition}. More precisely, Theorem 16.0.2 from \cite{Tweedy} asserts that the Markov chain with transition kernel $\calQ$ is uniformly geometrically ergodic if, and only if, it is aperiodic and there exist a probability measure $\rho$ on $X$, some $\epsilon, \delta>0$ and $k\in \nat$, such that for all sets $A\subset X$ with $\rho (A)>\epsilon$  we have 
\begin{gather*}
	\inf_{x\in X} \calQ_k(x,A)>\delta.
\end{gather*}
In the proof of the following theorem we show that under our assumptions, Doeblin's condition for the semigroups $\{\calQ_n:n\in \nat_0\}$ and $\{\ducalQ_n:n\in \nat_0\}$ holds with measure $\nubar = \nu/\|\nu\|$ defined above, and that it is the only stationary measure for either semigroup in the sense of \eqref{Q-unif-ergodic}.

\begin{theorem}\label{thm:unif-ergodic}
	Assume \eqref{A1}, \eqref{A3} and \eqref{B}. 
	Then there exist $\kappa \in (0,1)$ and $C>0$ such that
\begin{align} \label{Q-unif-ergodic_geom_rate}
\sup_{x\in X}\sup_{\Vert g\Vert_{\ell^\infty (\nu)}\leq 1}
\left\vert \calQ_ng(x) - \nubar (g) \right\vert 
\leq C \kappa^{n}, \quad n \in \nat.
\end{align}
Moreover, $\nubar$ is the only stationary probability measure of the semigroup $\{\calQ_n:n\in \nat_0\}$. 
The same statement holds for the  dual  semigroup $\{\ducalQ_n:n\in \nat_0\}$. 
\end{theorem}

\begin{proof}
	The aperiodicity of the $\mathcal{Q}$-Markov chain follows at once from the fact that $P(x,y)>0$ for all $x,y$. Fix a reference point $x_0\in X$. We apply \eqref{phi_0_bound}, \eqref{phi_hat_bound} and the duality condition \eqref{cond:duality} to get 
\begin{align*}
	q_1(x,y)
	=\frac{1}{V(x)}\frac{p(x,y)}{\lambda_0\phi_0(x)\hatphi_0(y)}
	&\geq \frac{p(x,y)}{c\,\lambda_0\,P(x,x_0)\widehat{P}(y,x_0)}\\
	&=\frac{P(x,y)}{c\, \lambda_0\,P(x,x_0)P(x_0,y)\mu(x_0)}
	\geq  \frac{1}{c\, C_*\, \lambda_0\, \mu(x_0)},
\end{align*}
where the last inequality follows from \eqref{A1}.
We conclude that, for any $A\subset X$,
\begin{gather*}
	\calQ_1\bbjedan_A(x)
	= \sum_{y \in A}q_1(x,y)\nu(y)
	\geq \frac{1}{c\, C_*\, \lambda_0\mu(x_0)} \nu(A) = \frac{\|\nu\|}{c\, C_*\, \lambda_0\mu(x_0)} \nubar(A),
\end{gather*}
which implies Doeblin's condition for the probability measure $\nubar$. Therefore, there exists a probability measure $\pi$ on $X$ such that the uniform convergence in \eqref{Q-unif-ergodic} holds, and the rate of convergence is geometric. Consequently, for every $x \in X$, 
\begin{align} \label{eq:pi_is_nu}
	\nubar(x) = \nubar(\bbjedan_{\{x\}}) = \nubar(\calQ_n\bbjedan_{\{x\}}) \to\pi(x) \quad \text{as} \quad n \to \infty. 
\end{align}
This shows that $\pi =\nubar$, leading to \eqref{Q-unif-ergodic_geom_rate}.

The proof for the dual chain is the same as we have $\duq_1(x,y) = q_1(y,x)$. The uniqueness of the stationary measure $\nubar$ follows directly from the uniform convergence \eqref{Q-unif-ergodic_geom_rate} with the same argument as in \eqref{eq:pi_is_nu}.
\end{proof}

In  view  of Theorem \ref{thm:unif-ergodic} it is natural to ask whether our current (mild) assumptions ensure that the initial discrete Feynman--Kac semigroups $\{\calU_n:n\in \nat_0\}$, $\{\ducalU_n:n\in \nat_0\}$ are also uniformly quasi-ergodic on $\ell^{\infty} (m)$, resp., $\ell^{\infty} (\dum)$. The answer to this question is negative -- this will be clarified in the next section where we discuss the ergodicity on $\ell^1$ spaces, see Remark \ref{rem:neg_l_infty}.

\subsection{Equivalence between uniform ergodicity and quasi-ergodicity on $\ell^1$}

In this and the next paragraph we study ergodicity of the intrinsic semigroups and quasi-ergodicity of the original discrete Feynman--Kac semigroups in the $\ell^1$ setting, and their connection to the regularizing properties of both semigroups. 

We begin with a general result, which reveals an interesting connection between the \emph{progressive uniform ergodicity} of the intrinsic semigroup $\{\calQ_n:n\in \nat_0\}$ and the \emph{progressive uniform quasi-ergodicity} of the original Feynman--Kac semigroup $\{\calU_n:n\in \nat_0\}$. More precisely, we will show that the uniform ergodicity along a given exhaustion $(A_n)$ of the intrinsic semigroup $\{\calQ_n:n\in \nat_0\}$ on $\ell^1(\nu)$ is equivalent to the uniform quasi-ergodicity of the original Feynman--Kac semigroup $\{\calU_n:n\in \nat_0\}$ on $\ell^1(m)$ along the same exhaustion. It is remarkable that the ergodic convergence rates are equivalent as well. The same statement remains true for the dual semigroups.

\begin{proposition}\label{kappa_lem}
Assume \eqref{A3} and \eqref{B}. Let $\kappa: \nat \to [0,\infty)$ be a given rate function, let $(A_n)$ be an arbitrary sequence of subsets of $X$, and $n_0 \in \nat$. Consider the following two statements:
\begin{enumerate}
\item\label{kappa_lem-a} 
	There exists a constant $C_1>0$ such that for all $n \geq n_0$, 
	\begin{gather}\label{kappa_Q}
		\sup_{x\in A_{n}} \sup_{ \Vert f\Vert_{\ell^1 (\nu)}\leq 1}\left|\calQ_n\,f(x)-\nubar(f)\right|\leq C_1\,  \kappa(n).
	\end{gather}

\item\label{kappa_lem-b} 
	There exists a constant $C_2>0$ such that for all $n \geq n_0$, 
	\begin{gather}\label{kappa_U}
		\sup_{x\in A_{n}} \sup_{ \Vert f\Vert_{\ell^1 (m)}\leq 1} \left|\frac{\calU_n\,f(x)}{\calU_n\bbjedan(x)}-m(f)\right|\leq C_2\, \kappa(n).
	\end{gather} 
\end{enumerate} 
Under the above assumptions, condition \ref{kappa_lem-a} implies \ref{kappa_lem-b}. If, in addition, $C_3:=\sup_{n \geq n_0} C_2 \kappa(n) < 1$, then \ref{kappa_lem-b} implies \ref{kappa_lem-a}.

Moreover, the same implications  remain  true if $\calQ_n$ is replaced with $\ducalQ_n$ in \ref{kappa_lem-a}, and $\calU_n$, $m$ with $\ducalU_n$, $\dum$ in \ref{kappa_lem-b}. 
\end{proposition}

\begin{proof}  
We restrict our attention to the semigroups $\{\calQ_n:n\in \nat_0\}$ and $\{\calU_n:n\in \nat_0\}$ as the proof for  dual semigroups  is similar. 

\smallskip\noindent\ref{kappa_lem-a}$\Rightarrow$\ref{kappa_lem-b}:
We fix an arbitrary $\tilde{f}\in \ell^1(m)$ and set $f:=\tilde{f}/\phi_0$. One easily checks that 
\begin{align}\label{eq:help12}
	\| f\|_{\ell^1(\nu)} = \|\hatphi_0\|_{\ell^1(\mu)}\|\tilde{f}\|_{\ell^1(m)},
\end{align}
hence $f\in \ell^1(\nu)$. For any $x\in A_n$, we rewrite the expression appearing on the left-hand side of \eqref{kappa_U} as follows
\begin{align*}
	\frac{\calU_n\,\tilde{f}(x)}{\calU_n\bbjedan(x)}-m(\tilde{f})
	&= \frac{\lambda_0^{-n}\calU_n\,\tilde{f}(x)-\lambda_0^{-n}\calU_n\bbjedan(x)\,m(\tilde{f})}{\lambda_0^{-n}\calU_n\bbjedan(x)}\\
	&=\frac{\lambda_0^{-n}\calU_n\,\tilde{f}(x)-\frac{\phi_0(x)}{\|\nu\|} \Vert \hatphi_0\Vert_{\ell^1(\mu)}m(\tilde f)}{\lambda_0^{-n}\calU_n\bbjedan(x)}\\
	&\qquad \mbox{}+\frac{\frac{\phi_0(x)}{\|\nu\|} \Vert \hatphi_0\Vert_{\ell^1(\mu)}m(\tilde f)-\lambda_0^{-n}\calU_n\bbjedan(x)m(\tilde{f})}{\lambda_0^{-n}\calU_n\bbjedan(x)}.
\end{align*}
For the first summand we have
\begin{align*}
	\left\vert\frac{\lambda_0^{-n}\calU_n\,\tilde{f}(x)-\frac{\phi_0(x)}{\|\nu\|}\sum_{y\in X} \tilde{f}(y)\hatphi_0(y)\mu(y)}{\lambda_0^{-n}\calU_n\bbjedan(x)} \right\vert
	&=\frac{\phi_0(x)}{\lambda_0^{-n}\calU_n\bbjedan(x)}\left\vert \frac{\calU_n\,\tilde{f}(x)}{\lambda_0^{n}\phi_0(x)}
	   -\frac{1}{\|\nu\|}\sum_{y\in X}\tilde{f}(y)\hatphi_0(y)\mu(y)\right\vert\\
	&=\frac{\phi_0(x)}{\lambda_0^{-n}\calU_n\bbjedan(x)}|\calQ_n\,f(x)-\nubar (f)|\\
	&\leq  C_1 \left\|\phi_0\right\|_{\ell^\infty (\mu)} \|\hatphi_0\|_{\ell^1(\mu)} \| \tilde{f}\|_{\ell^1(m)} \, \kappa (n);
\end{align*}
in the last line we use \eqref{kappa_Q} together with \eqref{eq:help12} and the bound
\begin{gather}\label{bound-khelp1}
	\lambda_0^{n}\phi_0(x)
	= \calU_n\,\phi_0(x)
	\leq \left\|\phi_0\right\|_{\ell^\infty(\mu)}\calU_n\bbjedan(x).
\end{gather}
We proceed with the estimate of the second summand. We have
\begin{align*}
	\frac{\frac{\phi_0(x)}{\|\nu\|} \Vert \hatphi_0 \Vert_{\ell^1(\mu)}m(\tilde f) -\lambda_0^{-n}\calU_n\bbjedan(x)m(\tilde{f})}{\lambda_0^{-n}\calU_n\bbjedan(x)}
	&= \frac{\phi_0(x)}{\lambda_0^{-n}\calU_n\bbjedan(x)}\, m(\tilde{f})
	   \left(\frac{1}{\|\nu\|}\sum_{y\in X}\hatphi_0(y)\mu(y) - \frac{\calU_n\,\bbjedan(x)}{\lambda_0^{n}\phi_0(x)}\right)\\
	&=\frac{\phi_0(x)}{\lambda_0^{-n}\calU_n\bbjedan(x)}\,  m(\tilde{f}) \left(\nubar (1/\phi_0) -  \calQ_n(1/\phi_0)(x)\right).
\end{align*}
Using the assumption \eqref{kappa_Q} for the function $1/\phi_0$ gives
\begin{gather}\label{kappa-bound-help2} 
	\left\vert \nubar \left(\frac{1}{\phi_0}\right)-\calQ_n\left(\frac{1}{\phi_0}\right)(x)\right\vert \leq  C_1  \|\hatphi_0\|_{\ell^1(\mu)} \kappa(n),\quad x\in A_n.
\end{gather}
Further, we observe that
\begin{gather}\label{m-norm-bound-help3}
	|m(\tilde{f})|\leq \frac{\Vert f\Vert_{\ell^1(\nu)}}{\Vert \hatphi_0\Vert_{\ell^1(\mu)}}.
\end{gather} 
If we combine \eqref{bound-khelp1}, \eqref{kappa-bound-help2} and  \eqref{m-norm-bound-help3} with \eqref{eq:help12}, then we obtain
\begin{align*}
	\frac{\phi_0(x)m(\tilde{f})}{\lambda_0^{-n}\calU_k\bbjedan(x)}\left\vert \nubar \left(\frac{1}{\phi_0}\right)-\calQ_n\left(\frac{1}{\phi_0}\right)(x)\right\vert
	&\leq C_1   \|\phi_0\|_{\ell^\infty(\mu)}  \|\hatphi_0\|_{\ell^1(\mu)}  \|\tilde{f}\|_{\ell^1(m)}\, \kappa(n),
\end{align*}
finishing the proof of the implication \ref{kappa_lem-a}$\Rightarrow$\ref{kappa_lem-b}.

\smallskip\noindent\ref{kappa_lem-b}$\Rightarrow$\ref{kappa_lem-a}:
Recall that for this implication we additionally assume that $C_3=\sup_{n \geq n_0} C_2 \kappa(n) < 1$.
For an arbitrary $f \in \ell^1(\nu)$ we  write  $\tilde{f}:= f \phi_0$. Clearly, $\tilde{f} \in \ell^1(m)$ and \eqref{eq:help12} holds. For any $x\in A_n$, we can write
\begin{align*}
	\calQ_n\,f(x)-\nubar (f) 
	&= \frac{\lambda_0^{-n}\calU_n\,\tilde{f}(x)}{\phi_0(x)} - \nubar (f)\\
  	&= \frac{\lambda_0^{-n}\calU_n\bbjedan(x)}{\phi_0(x)} \left(\frac{\calU_n\,\tilde{f}(x)}{\calU_n\bbjedan(x)} - m(\tilde{f}) + m(\tilde{f})-\nubar (f) \frac{\phi_0(x)}{\lambda_0^{-n}\calU_n\bbjedan(x)}\right)
\intertext{and, since $m(\tilde{f}) = \nubar (f) (\|\nu\|/\Vert \hatphi_0\Vert_{\ell^1(\mu)}) = \nubar (f) m(\phi_0)$ and $\phi_0 = \lambda_0^{-n}\calU_n \phi_0$, this equality can be further transformed to}
	\calQ_n\,f(x)-\nubar (f) 
  	&= \frac{\lambda_0^{-n}\calU_n\bbjedan(x)}{\phi_0(x)} \left(\frac{\calU_n\,\tilde{f}(x)}{\calU_n\bbjedan(x)} - m(\tilde{f})\right) 
	  + \nubar (f) \frac{\lambda_0^{-n}\calU_n\bbjedan(x)}{\phi_0(x)} \left(m(\phi_0)-\frac{\calU_n\phi_0(x)}{\calU_n\bbjedan(x)}\right).
\end{align*}
If we take $f=\phi_0$ in \eqref{kappa_U} and use $C_3=\sup_{n \geq n_0} C_2 \kappa(n) < 1$, we see that for $n \geq n_0$, 
\begin{gather*}
	\frac{\lambda_0^{-n}\calU_n\bbjedan(x)}{\phi_0(x)}  \leq c, \quad x \in A_n,
\end{gather*}
with $c = (m(\phi_0)(1-C_3))^{-1}$. 
Indeed, from \eqref{kappa_U} we have for $n \geq n_0$
\begin{gather*}
	m(\phi_0)-\frac{\calU_n\phi_0(x)}{\calU_n\one(x)}\leq C_2\kappa(n) m(\phi_0) \leq C_3 m(\phi_0),
\intertext{and so}
	m(\phi_0)(1-C_3)\leq \frac{\calU_n\phi_0(x)}{\calU_n\one(x)} = \frac{\lambda_0^n\phi_0(x)}{\calU_n\one(x)}.
\end{gather*}
	
Consequently, using \eqref{kappa_U} for both $\tilde{f}$ and $\phi_0$, and \eqref{eq:help12}, we obtain
\begin{align*}
	|\calQ_n\,f(x)-\nubar (f)| 
	&\leq c \left|\frac{\calU_n\,\tilde{f}(x)}{\calU_n\bbjedan(x)} - m(\tilde{f})\right|
	   + c \, \nubar (f) \left|\frac{\calU_n\phi_0(x)}{\calU_n\bbjedan(x)}-m(\phi_0)\right| \\
	&\leq c \, C_2 \left(\Vert \tilde{f}\Vert_{\ell^1(m)} + \frac{\Vert f \Vert_{\ell^1(\nu)}}{\|\nu\|}\Vert \phi_0\Vert_{\ell^1(m)}\right) \kappa(n)\\
	&= \frac{2 \, c \, C_2 }{\Vert \hatphi_0\Vert_{\ell^1(\mu)}} \Vert f \Vert_{\ell^1(\nu)} \kappa(n),
\end{align*}
for $x \in A_n$ and sufficiently large $n$. This completes the proof of the implication \ref{kappa_lem-b}$\Rightarrow$\ref{kappa_lem-a}.
\end{proof}

\begin{remark}\label{rem:neg_l_infty}
We can now come back to the issue raised at the end of Section \ref{sec:unif_ergod}. If we analyse the proof of Proposition \ref{kappa_lem}, especially the implication \ref{kappa_lem-a}$\Rightarrow$\ref{kappa_lem-b}, we conclude that, in order to obtain the quasi-ergodicity of the initial discrete Feynman--Kac semigroups $\{\calU_n:n\in \nat_0\}$ and $\{\ducalU_n:n\in \nat_0\}$ on the spaces  $\ell^{\infty} (m)$ and $\ell^{\infty} (\dum)$, respectively, we would need to prove ergodicity of the intrinsic semigroups on a much larger space than $\ell^{\infty} (\nu)$. More precisely, if (say) $\tilde{f} \in \ell^{\infty} (m)$, then one needs to know that $\calQ_n(\tilde{f}/\phi_0)$ and $\calQ_n(1/\phi_0)$ converge, but the trouble is that $\tilde{f}/\phi_0, 1/\phi_0 \notin \ell^{\infty} (\nu)$. This means that the general assumptions in Theorem \ref{thm:unif-ergodic} are not strong enough to ensure uniform quasi-ergodicity on $\ell^{\infty} (m)$, resp., $\ell^{\infty} (\dum)$. 
\end{remark}

\subsection{Necessary and sufficient conditions for uniform (quasi-)ergodicity on $\ell^1$}

We first prove that pIUC implies uniform ergodicity of the intrinsic semigroups on $\ell^1(\nu)$ with a rate function $\kappa(n)$, which is determined by the decay rate of the potential along a given exhausting family. In view of Proposition \ref{kappa_lem} this is equivalent to the uniform quasi-ergodicity along the same exhaustion of the original Feynman--Kac semigroups. 

Before we state this result, we recall a lemma, which describes the rate of convergence of the kernel $u_n(x,y)$ of the compact operator $\calU_n$. This type of result is known in a much more general context of continuous time semigroups, see e.g.\ \cite{Kim-Song-2008}, \cite{Zhang} and references therein. Now we need a version with a space rate on the right-hand side, i.e.\ a discrete time counterpart of the result from \cite{KSch-ergodic}.

\begin{lemma}\label{rho_lem}
	Assume \eqref{A3} and \eqref{B}. There exist $C>0$ and $\rho \in (0,1)$ such that for all $x,y \in X$,  $n \geq 3$, 	
	and all $k,l,j \in\nat_0$ such that $k+l+j=n$ we have
\begin{gather}\label{eq:klm}
	\left|\|\nu\| \lambda_0^{-n} \, u_n(x,y)-\phi_0(x)\hatphi_0(y)\right|
	\leq C\rho^k\left(\lambda_0^{-l}\, \calU_{l}\bbjedan(x)\right)\left(\lambda_0^{-j}\, \ducalU_{j}\bbjedan(y)\right).
\end{gather}
\end{lemma}
As the proof in the discrete setting does not much differ from that in \cite{Zhang, KSch-ergodic}, we do not include it here.  
 
According to Theorem \ref{th:pIUC} the pIUC property holds for both semigroups $\{\calU_n:n\in \nat_0\}$ and $\{\ducalU_n:n\in \nat_0\}$  if we assume  \eqref{A1}--\eqref{A3} and \eqref{B}. This shows the wide applicability of the next theorem.    

\begin{theorem}\label{pIUC_erg}
Assume \eqref{A1}, \eqref{A3} and \eqref{B}. Then the following statements hold. 
\begin{enumerate}
\item\label{pIUC_erg-a} 
	If $n_0 \in \nat$ and $(A_n)$ is an exhausting sequence for which both semigroups $\{\calU_n:n\in \nat_0\}$, $\{\ducalU_n:n\in \nat_0\}$ are pIUC\textup{($n_0$)}, then there exist $C_1, C_2 >0$ and $\rho\in (0,1)$ such that for all $n > 2n_0+1$ and $k,l,j \in \nat_0$ such that $l,j \geq n_0$ and $k+l+j=n-1$ we have 
	\begin{gather}\label{pIUC_sup}
		\sup_{x\in A_{l}} \, \sup_{ \Vert f\Vert_{\ell^1 (\nu)}\leq 1} \left|\calQ_n\,f(x)-\nubar (f)\right|
		\leq C_1 \left( \sup_{z\in A^c_{j}}\frac{1}{V(z)}+\rho^{k}\right)
	\intertext{and }\label{pIUC_quasi_erg}
		\sup_{x\in A_{l}}\sup_{ \Vert f\Vert_{\ell^1 (m)}\leq 1}\left|\frac{\calU_n\,f(x)}{\calU_n\bbjedan(x)}-m(f)\right|
		\leq C_2 \left( \sup_{z\in A^c_{j}}\frac{1}{V(z)}+\rho^{k}\right).
	\end{gather}
	These assertions remain valid for the dual setting, i.e.\ if one replaces $\calQ_n$ with $\ducalQ_n$ in \eqref{pIUC_sup}, and $\calU_n$ and $m$ with $\ducalU_n$ and $\dum$ in \eqref{pIUC_quasi_erg}. 

\item\label{pIUC_erg-b} 
	If there is an exhausting sequence $(A_n)$ such that  
	\begin{gather}\label{eq:e_aux}
		\lim_{n \to \infty} \sup_{x\in A_{n}} \sup_{ \Vert f\Vert_{\ell^1 (\nu)}\leq 1} \left|\calQ_n\,f(x)-\nubar (f)\right| = 0, 
		\quad f \in \ell^1 (\nu),
	\intertext{or}\label{eq:q-e_aux}
		\lim_{n \to \infty} \sup_{x\in A_{n}} \sup_{ \Vert f\Vert_{\ell^1 (m)}\leq 1} \left|\frac{\calU_n\,f(x)}{\calU_n\bbjedan(x)}-m(f)\right| = 0, \quad f \in \ell^1 (m),
	\end{gather}
	holds, then the semigroup $\{\calU_n:n\in \nat_0\}$ is pIUC with respect to the exhaustion $(A_{n-1})$. This assertion remains true, if one replaces $\calQ_n, \calU_n, m$ with $\ducalQ_n, \ducalU_n, \dum$.
\end{enumerate}
\end{theorem}

\begin{proof}
We first consider part \ref{pIUC_erg-a}. In view of Proposition \ref{kappa_lem}, it suffices to prove \eqref{pIUC_sup}. We start with the semigroup $\{\calQ_n:n\in \nat_0\}$. Fix $k,l,m \in\nat_0$ such that $l,j \geq n_0$ and $k+l+j=n-1$. For any $x\in A_l$ and $f\in \ell^1(\nu)$  we have
\begin{align*}
	\left|\calQ_n\,f(x)-\nubar (f)\right|
	&=\left|\calQ_{n-1}\calQ_1\,f(x)-\nubar (\calQ_{1}f)\right|\\
	&=\Big|\sum_{y\in X}\left(q_{n-1}(x,y)-\frac{1}{\|\nu\|}\right)\calQ_{1}f(y)\nu(y)\Big|.
\end{align*}
Now we split the sum $\sum_{y\in X}$ into two parts $\sum_{y\in A_j} + \sum_{y\in A_j^c}$ and we estimate both sums separately. We start with the first part
\begin{align*}
	&\left|\sum_{y\in A_{j}} \frac{1}{\|\nu\|\phi_0(x)\hatphi_0(y)} \left(\frac{\|\nu\|u_{n-1}(x,y)}{\lambda_0^{n-1}}-\phi_0(x)\hatphi_0(y)\right) \calQ_{1}f(y)\nu(y)\right|\\
	&\qquad\leq \sum_{y\in A_{j}} \frac{1}{\|\nu\|\phi_0(x)\hatphi_0(y)}\left|\frac{\|\nu\|u_{n-1}(x,y)}{\lambda_0^{n-1}}-\phi_0(x)\hatphi_0(y)\right|
	\calQ_{1}\left(\left|f\right|\right)(y)\nu(y)\\
	&\qquad\leq c\rho^k \frac{\calU_l\bbjedan(x)}{\lambda_0^l\phi_0(x)} \sum_{y\in A_{j}} \frac{\ducalU_j\bbjedan(y)}{\lambda_0^j\hatphi_0(y)} \calQ_{1}\left(\left|f\right|\right)(y)\nu(y)\\
	&\qquad\leq c_1\rho^k\sum_{y \in A_{j}}\calQ_{1}\left(\left|f\right|\right)(y)\nu(y)\leq  c_1\rho^k\left\|f\right\|_{\ell^1(\nu)},
\end{align*}
where we applied Lemma \ref{rho_lem} and the pIUC property of both semigroups, combined with the fact that $\nu$ is the stationary measure of the intrinsic semigroup.

We proceed with the estimate of the second sum. Combining pIUC and the triangle inequality imply that there is a constant $c>0$ such that  for every  $x\in A_l$ and $y\in X$, $\left\vert q_{n}(x,y) - \|\nu\|^{-1}\right\vert \leq c$, hence
\begin{align*}
	\left|\sum_{y\in A_{j}^c}\left(q_{n-1}(x,y)-\frac{1}{\|\nu\|}\right)\calQ_{1}f(y)\nu(y)\right|
	&\quad\leq c\sum_{y\in A_{j}^c}\calQ_{1}\left(|f|\right)(y)\nu(y)\\
	&\quad=c\sum_{z\in X}\ducalQ_{1}\left(\bbjedan_{A_j^c}\right)(z)|f(z)|\nu (z)\\
	&\quad\leq c \Vert f\Vert_{\ell^1(\nu)} \Vert\ducalQ_{1}\left(\bbjedan_{A_{j}^c}\right)\Vert_{\ell^\infty (\nu)} .
\end{align*}
Let us finally  show that there is a constant $c_1>0$ such that
\begin{align}\label{eq:aux_piuc1}
	\Vert\ducalQ_{1}\bbjedan_{A_{j}^c}\Vert_{\ell^\infty (\nu)} 
	\leq c_1 \sup_{y\in A_j^c}\frac{1}{V(y)}.
\end{align}
To this end, we observe that, for any $z\in X$,
\begin{align*} 
	\ducalQ_{1}\left(\bbjedan_{A_{j}^c}\right)(z)
	&= \sum_{y\in X}\duq_{1}(z,y)\bbjedan_{A_{j}^c}(y)\nu(y)\\
	&=\frac{1}{\lambda_0 \hatphi_0(z)}\sum_{y\in X}u_{1}(y,z)\bbjedan_{A_{j}^c}(y) \hatphi_0(y)\mu(y) \\
	&\leq\left(\sup_{y\in A_j^c}\frac{1}{V(y)}\right)\cdot\frac{1}{\lambda_0 \hatphi_0(z)}\sum_{y\in X}p(y,z) \one_{A_j^c}(y) \hatphi_0(y)\mu(y). 
\end{align*}
Further, by \eqref{eq:DSP} combined with \eqref{phi_hat_bound}, for a fixed reference point $x_0\in X$ we have
\begin{align*}
	\frac{1}{\lambda_0 \hatphi_0(z)}\sum_{y\in X}p(y,z)\bbjedan_{A_{j}^c}(y) \hatphi_0(y)\mu(y)
	&\leq \frac{1}{\lambda_0 \hatphi_0(z)}\sum_{y\in X}p(y,z)\widehat{p}(y,x_0)\mu(y)\\
	&\leq C_{*} \frac{1}{\lambda_0 \hatphi_0(z)}\widehat{p}(z,x_0)\leq c_2,
\end{align*}
where the penultimate bound is justified by \eqref{A1}, and this completes the proof for $\{\calQ_n:n\in \nat_0\}$. 

The proof of \eqref{pIUC_sup} for the semigroup $\{\ducalQ_n:n\in \nat_0\}$ is very similar. The only difference is in the argument leading to \eqref{eq:aux_piuc1} for the operator $\calQ_1$: Fix $x_0\in X$. By \eqref{phi_0_bound} and \eqref{eq:DSP}, we have
\begin{align*} 
	\calQ_{1}\left(\bbjedan_{A_{j}^c}\right)(z)
	& = \frac{1}{\lambda_0 \phi_0(z)}\sum_{y\in X}u_{1}(z,y)\bbjedan_{A_{j}^c}(y) \phi_0(y)\mu(y) \\
	&\leq\left(\sup_{y\in A_j^c}\frac{1}{V(y)}\right)\cdot\frac{c}{\lambda_0 p(z,x_0)}\sum_{y\in X}p(z,y) p(y,x_0)\mu(y)
	    \leq c\, C_* \sup_{y\in A_j^c}\frac{1}{V(y)}. 
\end{align*} 
This completes the proof of Part \ref{pIUC_erg-a}.

\medskip
Let us show Part \ref{pIUC_erg-b}. The equivalence of \eqref{eq:e_aux} and \eqref{eq:q-e_aux} follows from Proposition \ref{kappa_lem}. Using \eqref{eq:q-e_aux} with $f=\phi_0$, we can find $n_0 \in \nat$ such that
\begin{gather*}
	\sup_{x\in A_n} \left|\frac{\lambda_0^n\phi_0(x)}{\calU_n\bbjedan(x)}-\frac{\|\nu\|}{\left\|\hatphi_0\right\|_{\ell^1(\mu)}}\right|
	\leq \frac12\left\|\phi_0\right\|_{\ell^1(m)} = \frac12 \frac{\|\nu\|}{\left\|\hatphi_0\right\|_{\ell^1(\mu)}}, \quad n \geq n_0-1,
\end{gather*}
and  using \eqref{eq:q-e_aux} once more with $f=f_y = u_1(\cdot,y)$, $y \in X$, yields 
\begin{gather*}
	\sup_{x\in A_{n-1}} \left|\frac{u_n(x,y)}{\calU_{n-1}\bbjedan(x)}-\frac{\lambda_0 \hatphi_0(y)}{\left\|\hatphi_0\right\|_{\ell^1(\mu)}}\right|
	\leq \frac12\left\|u_1(\cdot,y)\right\|_{\ell^1(m)}  = \frac12 \frac{\lambda_0 \hatphi_0(y)}{\left\|\hatphi_0\right\|_{\ell^1(\mu)}}, \quad n \geq n_0, \ y \in X.
\end{gather*} 
(we may use the same $n_0$ in both conditions). These estimates entail 
\begin{gather*}
	\calU_n\bbjedan(x) \leq c_1 \lambda_0^n \phi_0(x),\quad x\in A_n, \ n \geq n_0-1,
\end{gather*}
and
\begin{gather*}
	u_n(x,y) \leq c_2 \lambda_0 \calU_{n-1}\bbjedan(x) \hatphi_0(y) \leq c_1 c_2 \lambda_0^n \phi_0(x) \hatphi_0(y) ,\quad n \geq n_0, \; x\in A_{n-1}, \; y \in X.
\end{gather*}
This is pIUC of the semigroup $\{\calU_n:n\in \nat_0\}$ for the shifted exhausting sequence $(A_{n-1})$. The proof for the  dual  semigroup is the same. 
\end{proof}

Our final goal in this section is to establish the equivalence between the aIUC property for the semigroup $\{\calU_n:n\in \nat_0\}$, the uniform geometric ergodicity of the intrinsic semigroup $\{\calQ_n:n\in \nat_0\}$ on $\ell^1(\nu)$, and the uniform geometric quasi-ergodicity of $\{\calU_n:n\in \nat_0\}$ on $\ell^1(m)$. All of these properties are, moreover, equivalent to the same statements for the  dual  semigroups.

\begin{theorem}\label{t:ergodic}
Assume \eqref{A1}, \eqref{A3} and \eqref{B}. 
The following three conditions are equivalent:
\begin{enumerate}
\item\label{t:ergodic-1} 
	The semigroup $\{\calU_n:n\in \nat_0\}$ is aIUC;
\item\label{t:ergodic-2} 
	The semigroup $\{\calQ_n:n\in \nat_0\}$ is geometrically uniformly ergodic on $\ell^1(\nu)$, that is there exist $C_1>0$ and $\rho \in (0,1)$ such that for all $f\in \ell^1(\nu)$ and for sufficiently large $n$ it holds
	\begin{gather}\label{eq:aIUC-ineq}
		\sup_{x\in X} \left|\calQ_n\,f(x)-\nubar (f)\right|
		\leq C_1 \rho^{n}\left\|f\right\|_{\ell^1(\nu)};
	\end{gather}
\item\label{t:ergodic-3} 
	The semigroup $\{\calU_n:n\in \nat_0\}$ is geometrically uniformly quasi-ergodic on $\ell^1(m)$, that is there exist $C_2>0$ and $\rho\in (0,1)$ such that for all $f\in \ell^1(m)$ and for sufficiently large $n$ it holds
	\begin{gather}\label{quasi_erg}
		\sup_{x\in X} \left|\frac{\calU_n\,f(x)}{\calU_n\bbjedan(x)}-m(f)\right|
		\leq C_2 \rho^n\Vert f\Vert_{\ell^1(m)}.
	\end{gather}
\end{enumerate} 
These equivalences remain true if one replaces $\calQ_n, \calU_n, m$ with $\ducalQ_n, \ducalU_n, \dum$. Moreover, any of the conditions \ref{t:ergodic-1}, \ref{t:ergodic-2}, and \ref{t:ergodic-3} is equivalent to the  corresponding statements for the  dual  semigroups $\{\ducalU_n:n\in \nat_0\}$ and $\{\ducalQ_n:n\in \nat_0\}$ and $\widehat m$.
\end{theorem}

\begin{proof}
Since aIUC implies pIUC we can apply Proposition \ref{kappa_lem} and Theorem \ref{pIUC_erg}.\ref{pIUC_erg-a} to conclude that \ref{t:ergodic-1}$\Rightarrow$\ref{t:ergodic-2}$\Rightarrow$\ref{t:ergodic-3}. The implication \ref{t:ergodic-3}$\Rightarrow$\ref{t:ergodic-1} follows from Theorem \ref{pIUC_erg}.\ref{pIUC_erg-b}. From \eqref{eq:equiv_aiuc} we conclude that we have both the equivalence of \ref{t:ergodic-1}--\ref{t:ergodic-3} for the  dual  semigroups and the equivalence of all statements for the original and the  dual  semigroups.
\end{proof}

So far, we have studied  the uniform ergodic and quasi-ergodic properties on $\ell^\infty$ and $\ell^1$, the two endpoints in the chain \eqref{eq:inclusions}. Due to these inclusions, the results from Theorems \ref{pIUC_erg}, \ref{t:ergodic} extend from $\ell^1(\nu)$, $\ell^1(m)$, $\ell^1(\dum)$ to $\ell^p(\nu)$, $\ell^p(m)$, $\ell^p(\dum)$, $1 < p \leq \infty$ with the same ($\ell^1$-)rates of convergence. 
However, as we will see in Section \ref{sec:ex}, the rate of convergence in the progressive case (Theorem \ref{pIUC_erg}.\ref{pIUC_erg-a}) may be much slower than geometric. We close this section by noting that the progressive uniform ergodicity of the intrinsic semigroups on $\ell^p(\nu)$, for $p$ arbitrarily close to $1$, is still geometric. This is obtained by interpolation. 

\begin{proposition}\label{prop-riesz-thorin}
	Let \eqref{A1}, \eqref{A3} and \eqref{B} hold. Assume that both semigroups $\{\calU_n:n\in \nat_0\}$ and $\{\ducalU_n:n\in \nat_0\}$ are pIUC with respect to some exhausting sequence $(A_n)$. Fix $p \in (1,\infty)$. Then there exist $C >0$ and $\rho = \rho(p) \in (0,1)$ such that for sufficiently large $n \in \nat$ we have
	\begin{gather*}
		 \sup_{x\in A_{n-1}} \sup_{ \Vert f\Vert_{\ell^p (\nu)}\leq 1} \left|\calQ_n\,f(x)-\nubar (f)\right|
		\leq C \rho^n.
	\end{gather*}
	This statement remains true for the  dual  semigroup $\{\ducalQ_n:n\in \nat_0\}$. 
\end{proposition}

\begin{proof}
We consider the operator $\mathcal{T}_n f = \calQ_n f - \nubar(f)$. By Theorem \ref{thm:unif-ergodic}, for sufficiently large $n$, $\mathcal{T}_n : \ell^{\infty}(X,\nu) \to \ell^\infty(X,\nu)$ is a bounded operator and $\ell^\infty(X,\nu) \hookrightarrow \ell^{\infty}(A_{n-1},\nu)$ is the canonical embedding, so $\Vert\mathcal{T}_n\Vert_{\ell^{\infty}(X,\nu) \to \ell^{\infty}(A_{n-1},\nu)} \leq C_1 \kappa^n$, for some $C_1>0$ and $\kappa \in (0,1)$. Moreover, by the proof of Theorem \ref{pIUC_erg}.\ref{pIUC_erg-a}, for sufficiently large $n$, the operator $\mathcal{T}_n : \ell^1(X,\nu) \to \ell^{\infty}(A_{n-1},\nu)$ is bounded as well, and $\Vert\mathcal{T}_n\Vert_{\ell^1(X,\nu) \to \ell^{\infty}(A_{n-1},\nu)} \leq C_2$. By the Riesz--Thorin interpolation theorem \cite[Ch.~4, Cor.~1.8]{Bennett-Sharpley}, $\mathcal{T}_n : \ell^p(X,\nu) \to \ell^{\infty}(A_{n-1},\nu)$ is also a bounded operator for such $n$'s and 
\begin{gather*}
	\Vert\mathcal{T}_n\Vert_{\ell^p(X,\nu) \to \ell^{\infty}(A_{n-1},\nu)} 
	\leq C_1^{1-1/p} (\kappa^n)^{1-1/p} C_2^{1/p} 
	=:  C_1^{1-1/p} C_2^{1/p} \rho^n,
\end{gather*}
which yields the desired estimate.
\end{proof}

\section{Examples} \label{sec:ex}

In this section we illustrate the results obtained above for some specific classes of discrete-time processes and confining potentials. The discussion is split into three parts. In Section \ref{sebsec:ex1} we discuss polynomially and exponentially decaying kernels $P$, and exponentially, polynomially and logarithmically growing confining potentials $V$. We show how the (intrinsic) contractivity properties of the discrete-time semigroups depend on the decay rate of $P$ and the growth rate of $V$. In Section \ref{sebsec:ex2} we  provide  examples of probability kernels $P$ that satisfy our assumptions. Finally, in Section \ref{sebsec:ex3}, we justify our claims from Section \ref{sebsec:ex1}. 

\subsection{ Overview} \label{sebsec:ex1} 

Let $x_0 \in X$ and let $(B_n)_{n\in \nat}$ be an exhausting sequence of $X$. We assume that all sets $B_n$ are finite and satisfy $B_{n+1} \setminus B_n \neq \emptyset$, $n \in \nat$. Further, we assume that for all $ n \in \nat$ and $ x \in X,$
\begin{gather} \label{eq:P_R}
	P(x,x_0) \asymp P(x_0,x) \asymp R(1) \cdot \one_{B_1}(x) + \sum_{n=1}^{\infty} R(n+1) \cdot \one_{B_{n+1} \setminus B_n}(x) = R(\alpha (x)),
\end{gather}
where $R:\nat \to (0,1]$ is a decreasing \emph{profile} function and $\alpha (x)$ is the first appearance function, see Definition \ref{def:exh_fam}.
Below we consider two types of profiles:
\begin{itemize}
\item $R(n) = n^{-\beta}$, 
\item $R(n) = e^{-\kappa n} n^{-\beta}$,
\end{itemize}
where $\beta>0$ and $\kappa>0$. The kernels $P$ corresponding to such profiles are referred to as \emph{polynomial} and \emph{exponential} one-step transition probabilities with respect to the exhaustion $(B_n)_{n\in \nat}$. The sets $B_n$ may grow at an arbitrary speed, but if $\# B_n \asymp n^{d}$ for some $d>0$ and large values of $n$ (e.g.\ $X = \integer^d$ and the $B_n$ are metric balls), then the names \emph{polynomial} and \emph{exponential} kernels can be understood in a proper sense. If $(X,\delta)$ is a metric space such that the distance $\delta:X \times X \to \nat$ is a surjective function, then (typically) $B_n= \left\{x \in X: \delta(x,x_0) \leq n \right\}$, $n \in \nat$. In general, the kernel $P$ may be independent of the metric structure of the space. We assume that the parameter $\beta>0$ and the sequence $(B_n)_{n\in \nat}$ are such that \eqref{A1}--\eqref{A3} hold; this will be discussed in detail for concrete examples in Section \ref{sebsec:ex2}.

In the examples we will consider confining potentials $V$ of the form
\begin{gather} \label{eq:V_W}
	V(x) = W(1) \cdot \one_{B_1}(x) +  \sum_{n=1}^{\infty}W(n+1) \cdot \one_{B_{n+1} \setminus B_n}(x)
	=W(\alpha(x)), \quad n \in \nat, \; x \in X,
\end{gather}
where the increasing profile $W:\nat \to (0,\infty)$ takes one of the following three different shapes:
\begin{itemize} 
\item $W(n) = \log^{\rho} (1+n)$,
\item $W(n) = n^{\rho}$,
\item $W(n) = \exp(\rho \, n)$,
\end{itemize}
for some $\rho>0$. Clearly, these potentials are increasing with respect to the sequence $(B_n)_{n\in \nat}$.

The main results for  such  kernels $P$ and potentials $V$, including the aIUC/pIUC properties and (quasi\nobreakdash-)ergodicity, are  summarized in Table \ref{Table}. The detailed justification of these examples, for each combination of $R$ and $W$, is deferred to Section \ref{sebsec:ex3}. 
\newpage
\begin{table}[h]\centering 
	\caption{Overview on results for various types of kernels $P$ and potentials $V$}\label{Table}
	\begin{tabular}{@{}c%
			@{\hskip 20pt}
			>{\def\fullwidthdisplay{\displayindent0pt \displaywidth\hsize}}p{.375\textwidth}%
			@{\hskip 40pt}
			>{\def\fullwidthdisplay{\displayindent0pt \displaywidth\hsize}}p{.375\textwidth}%
			@{}}\hlinewd{1.5pt}\midrule
		$\displaystyle W(n)$ 
		& \mbox{}\hfill$\displaystyle R(n) = n^{-\beta}$\hfill\mbox{} 
		& \mbox{}\hfill$\displaystyle R(n) = e^{-\kappa n} n^{-\beta}$\hfill\mbox{} 
		\\\toprule
		$\displaystyle \log^{\rho} (1+n)$ 
		& {pIUC with respect to the exhaustion $(B_{k(n)})_{n\in \nat}$, where $k(1) = 1$ and 
		\begin{gather*}
				k(n)=\left\lfloor 2^{\frac{-1}{2\beta}}\left(\frac{(n-1)\rho}{\widetilde C^{1/\rho}}\right)^{\frac{(n-1)\rho}{2\beta}}\right\rfloor \vee 1, 
		\end{gather*} for $n \geq 2$ and some $\widetilde C>0$. 
		\newline\medskip\par 
		The rate of convergence in Theorem \ref{pIUC_erg} for the exhaustion $(B_{k(n)})_{n\in \nat}$ is
		\begin{gather*}
			\sup_{z\in B_{k(m)}^c}\frac{1}{V(z)} \asymp \left(\frac{1}{m \log m}\right)^{\rho}
		\end{gather*} 
		for $m$ large enough.} 
		& {pIUC with respect to the exhaustion $(B_{k(n)})_{n\in \nat}$, where $k(1) = 1$ and 
		\begin{gather*}
			k(n)=\left\lfloor \frac{(n-1)\rho \log a}{2\kappa+\epsilon}\right\rfloor \vee 1, 
			\vphantom{k(n)=\left\lfloor 2^{\frac{-1}{2\beta}}\left(\frac{\gamma(n-1)}{\widetilde C^{1/\gamma}}\right)^{\frac{\gamma(n-1)}{2\beta}}\right\rfloor \vee 1,}
		\end{gather*} 
		for $n \geq 2$, $\epsilon >0$ and some $a>1$. 
		\newline\medskip\par 
		The rate of convergence in Theorem \ref{pIUC_erg} for the exhaustion $(B_{k(n)})_{n\in \nat}$ is
		\begin{gather*}
			\sup_{z\in B_{k(m)}^c}\frac{1}{V(z)} \asymp \left(\frac{1}{\log m}\right)^{\rho},
		\end{gather*}
		for $m$ large enough.} 
		\\ \midrule
		$\displaystyle n^\rho$ 
		& {aIUC; \newline Theorem \ref{t:ergodic} applies and the rate of convergence is geometric.} 
		& {pIUC with respect to the exhaustion $(B_{k(n)})_{n\in \nat}$, where $k(1) = 1$ and 
			\begin{gather*}
				k(n)=\left\lfloor \frac{(n-1)\rho\log\left(\frac{(n-1)\rho}{\widetilde C^{1/\rho}}\right)}{2\kappa+\epsilon}\right\rfloor \vee 1,
			\end{gather*} 
			for $n \geq 2$, $\epsilon >0$ and some $\widetilde C>0$. 
			\newline\medskip\par 
			The rate of convergence in Theorem \ref{pIUC_erg} for the exhaustion $(B_{k(n)})_{n\in \nat}$ is
		\begin{gather*}
			\sup_{z\in B_{k(m)}^c}\frac{1}{V(z)} \asymp \left(\frac{1}{m \log m}\right)^{\rho},
		\end{gather*}
		for $m$ large enough.}
 		\\ \midrule
		$\displaystyle \exp(\rho n)$ 
		& {aIUC; \newline Theorem \ref{t:ergodic} applies and the rate of convergence is geometric.} 
		& {aIUC; \newline Theorem \ref{t:ergodic} applies and the rate of convergence is geometric.} \\[3pt]\hlinewd{1.5pt}
	\end{tabular}
\end{table}

\subsection{Examples of transition kernels satisfying our assumptions} \label{sebsec:ex2} 

The examples presented below are constructed using a general setup. The starting point is an infinite, not necessarily symmetric,  matrix $\left(\mu(x,y)\right)_{x,y\in X}$ with non-negative entries and such that
\begin{gather}\label{ass-i-ii}
	\forall x\in X\::\: 0 < \sum_{y\in X}\mu(x,y) = \sum_{y\in X} \mu(y,x) < \infty
\end{gather}
By $\dumu(x,y):=\mu(y,x)$ we denote the transposed (adjoint) matrix and we define
\begin{gather*}
	\mu(x)  := \sum_{y\in X}\mu(x,y),
	\qquad
 \dumu(x):	= \sum_{y\in X}\dumu(x,y).
\end{gather*}
The probability kernels $P(x,y)$ and $\Phat(x,y)$ are then defined by
\begin{gather*}
	P(x,y) :=\frac{\mu(x,y)}{\mu(x)}
	\quad\text{and}\quad 
	\Phat(x,y) := \frac{\dumu(x,y)}{\dumu(x)}.
\end{gather*}
By  construction  we have $\mu(x) = \dumu(x)$, $x \in X$, and therefore the measure $\mu$ and the kernels $P, \Phat$ satisfy the duality  relation  \eqref{cond:duality}. Note that if $\mu(x,y) = \mu(y,x) = \dumu(x,y)$, then $\Phat = P$ and \eqref{cond:duality} reduces to the detailed balance condition.
 
\subsubsection{Markov chains on discrete metric spaces} \label{sec:discrete-metric} Let $(X,\delta)$ be a metric space such that the distance $\delta:X \times X \to \nat$ is a surjective function, and let $J, K:[0,\infty)\to (0,\infty)$ be decreasing functions such that 
\begin{gather*}
	J(r)\leq c_1 J(2r),\quad r>0,
\intertext{and}
	K(r)K(s) \leq c_2 K(r+s), \quad r,s >0,
\end{gather*}
for some constants $c_1, c_2>0$. Moreover, we assume that there are constants $c_3, c_4, c_5$ such that 
\begin{align}\label{eq:integrability}
0 < c_3 \leq \sum_{y \in X} J(\delta(x,y))K(\delta(x,y)) \leq c_4 \sum_{y \in X} J(\delta(x,y))  \leq c_5 < \infty, \quad x \in X. 
\end{align}
We define
\begin{align*}
\mu(x,y):= J(\delta(x,y)) K(\delta(x,y)),\quad x,y\in X.
\end{align*}
Observe that $\mu(x,y) = \mu(y,x)$, $x,y \in X$  holds  and \eqref{ass-i-ii} is satisfied. Moreover, \eqref{A1} follows from \cite[Proposition 4.1, Corollary 4.2]{CKS-ALEA} and \eqref{eq:integrability}, \eqref{A2} holds because of $J(0)K(0) >0$ and the upper bound in \eqref{eq:integrability}, and \eqref{A3} is a consequence of boundedness of $J, K$ (being decreasing functions) and the lower bound in \eqref{eq:integrability}.

One can fix $x_0 \in X$ and take $R(n) := (\left\|J\right\|_{\infty} \left\|K\right\|_{\infty})^{-1}  J(n)K(n)$ and $B_n:= B_n(x_0)$ in \eqref{eq:P_R}, where $B_n(x)= \left\{z \in X: \delta(z,x) \leq n \right\}$. For instance, if $\# B_n(x) \asymp n^{d}$ for some $d \geq 1$ (uniformly in $x \in X$, at least for large values of $n$), then by taking $J(r) = r^{-\beta}$, $K(r) = e^{-\kappa r}$, for $\beta > d$ and $\kappa \geq 0$, we obtain the transition kernels with the shapes discussed in Section \ref{sebsec:ex1}. Such an arrangement already provides the regularity needed for \eqref{eq:integrability}.

This  set-up  covers many important examples of discrete metric spaces such as integer lattices $\integer^d$ (with various distances), nested graphs (i.e.\ self-similar graphs formed by vertices of unbounded nested fractals embedded in 
Euclidean spaces, see e.g.\ \cite{Nieradko-Olszewski} and references therein for basic definitions), and some other infinite graphs with geodesic distance \cite{Murugan_Saloff-Coste} or general uniformly discrete metric spaces \cite{Murugan_Saloff_2}. 

\subsubsection{Subordinate discrete-time processes on infinite graphs} \label{sec:subordinate}
Let $G$ be a connected graph with finite geometry over $X$, as introduced in Section \ref{sec:NNRW}. Let $\{Z_n: n \geq 0\}$ be a nearest neighbour Markov chain on $G$, i.e.\ a time-homogeneous discrete-time Markov process with one-step transition probabilities $Q(x,y)$ satisfying the condition \eqref{eq:nnrw_def}, and  let $\{\tau_n:n \geq 0 \}$ be an arbitrary increasing random walk with values in $\nat_0$ and $\tau_0=0$ (i.e.\ the steps $\tau_{n+1}-\tau_n$, $n=0,1,2.\ldots$ are i.i.d.\ random variables with values in $\nat_0$). Moreover, we assume that  $\{\tau_n:n \geq 0 \}$  is independent of $\{Z_n: n \geq 0\}$. We call $\{\tau_n:n \geq 0 \}$ a \emph{discrete subordinator}.
The \emph{subordinate Markov chain} $\{Y_n: n \geq 0\}$ is then defined as  random time change: 
\begin{gather*}
		Y_n:= Z_{\tau_n}, \quad n=0,1,2,\dots . 
\end{gather*} 
It is not hard to see that the  process $\{Y_n: n \geq 0\}$ is again a time-homogeneous Markov chain. 

The nearest neighbour transition probabilities $Q$ are typically constructed from the symmetric conductances (weights) over $G$ in a similar way as presented at the beginning of Section \ref{sebsec:ex2}, see e.g.\ Barlow \cite{Barlow_book} or Kumagai \cite{Kumagai}.
There always exists a measure $\mu: X \to (0,\infty)$ such that $\mu(x) Q(x,y) = \mu(y) Q(y,x)$, that is the detailed balance condition holds. 

The one-step transition probabilities $P(x,y)$ of the process $\{Y_n: n \geq 0\}$ are given by
\begin{gather} \label{eq:subord_P}
	P(x,y) = \sum_{n=1}^{\infty} Q_n(x,y) \Pp (\tau_1=n),
\end{gather}
and it follows directly from this formula and the detailed balance of $Q$ that $\mu(x) P(x,y) = \mu(y) P(y,x)$. 

By taking
$\mu(x,y):= \mu(x) P(x,y)$, we can observe that we are formally still in the framework that was described at the beginning 
of Section \ref{sebsec:ex1}; in particular, \eqref{ass-i-ii} trivially holds. The assumption \eqref{A1} is satisfied as long as we know that there is some $c>0$ such that
\begin{align} \label{eq:subord_dsp}
	\Pp (\tau_2=n) \leq c \Pp (\tau_1=n), \quad n \geq 2,
\end{align}
(in this case we say that the subordinator $\{\tau_n:n \geq 0 \}$ has the DSP), and there exists  some  $n_0 \in \nat$ such that 
$\Pp(\tau_1=n)>0$, for all $n \geq n_0$ \cite[Lemma 4.3, Corollary 4.4]{CKS-ALEA}. Furthermore, \eqref{A2} follows from \eqref{eq:subord_P} under the assumption that the kernel $Q$ satisfies the condition \eqref{eq:nnrw_add_ass}, and \eqref{A3} holds if we know that e.g.\ $\inf_{x \in X} \mu(x) > 0$. The most important examples of discrete subordinators that fit our settings are stable and relativistic subordinators, which are discussed in detail in \cite[pp.~1094-1095]{CKS-ALEA}. This allows one to study the Feynman--Kac semigroups corresponding to fractional powers of the nearest-neighbour Laplacians on graphs $-(\id - Q)^{\alpha/2}$, $\alpha \in (0,2)$, and their relativistic variants $-(\id - Q+m^{2/\alpha})^{\alpha/2}+m$, $\alpha \in (0,2)$, $m>0$. 

In order to identify the sequences $(R(n))_{n \in \nat}$ and $(B_n)_{n \in \nat}$ in \eqref{eq:P_R} one has to find reasonably good bounds for the series in \eqref{eq:subord_P}. This is often possible in the case of discrete metric spaces where the estimates are given in terms of functions of a distance (see e.g.\ \cite{Cygan-Sebek,Grzywny-Trojan} and \cite[Sec.\ 4.2]{CKS-ALEA}), in a similar way as it was done in Section \ref{sec:discrete-metric}.

\subsubsection{Probability kernels on product spaces}\label{sec:product} 
Let $\left(\mu_1(x_1,y_1)\right)_{x_1,y_1\in X_1}$, $\left(\mu_2(x_2,y_2)\right)_{x_2,y_2\in X_2}$ be infinite matrices satisfying \eqref{ass-i-ii} and leading to probability kernels $P_1(x_1,y_1)$, $P_2(x_2,y_2)$ (and $\Phat_1(x_1,y_1)$, $\Phat_2(x_2,y_2)$) that satisfy assumptions \eqref{A1}-\eqref{A3} and the condition \eqref{cond:duality} with the measures $\mu_1$, $\mu_2$, as explained at the beginning of Section \ref{sebsec:ex2}. 

Then the probability kernel $P(x,y) = P_1(x_1,y_1)P_2(x_2,y_2)$, $x,y \in X_1 \times X_2$ (which can also be formally constructed from $\mu_1(x_1,y_1)\mu_2(x_2,y_2)$) satisfies \eqref{cond:duality} and the assumptions \eqref{A1}-\eqref{A3} with the measure $\mu(x)= \mu_1(x_1)\mu_2(x_2)$. Moreover, $P$ satisfies \eqref{eq:P_R} with the sequences $(R(n))_{n \in \nat}$ and $(B_n)_{n \in \nat}$ such that $R(n)=R^{(1)}(n)R^{(2)}(n)$ and $B_n = B^{(1)}_n \times B^{(2)}_n$, where $(R^{(i)}(n))_{n \in \nat}$, $(B^{(i)}_n)_{n \in \nat}$ corresponds to the kernel $P_i$, $i=1,2$. This construction easily extends to finite products of discrete spaces, including integer lattices and products of more general graphs. 

We remark that the kernels $P_i$ may be drastically different, e.g.\ $P_1$ could decay polynomially, while $P_2$ decays exponentially. This shows that our framework covers much more general examples than those discussed in Sections \ref{sec:discrete-metric}, \ref{sec:subordinate}. Examples of this type can be seen as discrete-time counterparts of the continuous-time semigroups that were recently studied by Kulczycki and Sztonyk \cite{Kulczycki-Sztonyk}.

\subsubsection{Non-reversible Markov chains} Our framework also accommodates examples for which $\mu(x,y) \neq \mu(y,x)$ or, equivalently, $\Phat(x,y) \neq P(x,y)$. Non-symmetry may enter the picture in various ways. To illustrate this, we give a simple example of a transition kernel on $\integer$, which works well in our settings. This example can be generalized in various directions. Assume that $\mu(x,y)=k(y-x)$ for a function $k : \integer \to (0,\infty)$ such that $k(-x) \neq k(x)$ for all $x\in X$, and $\left\|k\right\|_{\ell^1(\integer)} = \sum_{x \in \integer} k(x) < \infty$. Since 
\begin{gather*}
	\mu(x) 
	= \sum_{y\in \integer}k(y-x) 
	= \sum_{y\in \integer}k(y) 
	= \left\|k\right\|_{\ell^1(\integer)} < \infty, \quad x \in \integer,
\end{gather*}
it is clear that \eqref{ass-i-ii} is satisfied and $\dumu = \mu$. Moreover, 
\begin{gather*}
	P(x,y) = \frac{1}{\left\|k\right\|_{\ell^1(\integer)}} k(y-x), 
	\quad 
	\Phat(x,y) = \frac{1}{\left\|k\right\|_{\ell^1(\integer)}} k(x-y) 
	= \frac{1}{\left\|k\right\|_{\ell^1(\integer)}} k(-(y-x)),
\end{gather*}
and \eqref{cond:duality} holds. It is also straightforward to check that \eqref{A1}, \eqref{A2} hold.

We give two examples of  a  function $k$ for which \eqref{A3} is also satisfied:
\begin{itemize}
\item 
	$k(x) = l(|x+x_0|)$, for a fixed $0 \neq x_0 \in \integer$, or 
\item 
	$k(x) = \left(\eta \one_{\left\{x<0\right\}} + \one_{\left\{x=0\right\}} +  (1-\eta) \one_{\left\{x>0\right\}}\right) l(|x|)$, for some $\eta \in (0,1)\setminus\{\frac 12\}$,
\end{itemize}
where $l:\nat_0 \to (0,\infty)$ is such that $\left\|l\right\|_{\ell^1(\nat_0)} = 1$ and
\begin{gather}\label{l_DSP}
	\sum_{z\in \integer }l(|z-x)|)l(|y-z|)\leq c\,l(|y-x|), 
	\quad  x, y\in \integer,
\end{gather}
for a constant $c>0$. Depending on the type of $k$, we get with \eqref{l_DSP}
\begin{align*}
	\sum_{z\in \integer}k(z-x)k(y-z) 
	&= \sum_{z\in \integer}l(|z-(x-x_0)|)l(|y+x_0-z|)\\ 
	&\leq c l(|y-x+2x_0|)\\ 
	&= c k(y-x+x_0)
\intertext{or}
	\sum_{z\in \integer}k(z-x)k(y-z) 
	\leq \sum_{z\in \integer}l(|z-x|)l(|y-z|)
	&\leq c l(|y-x|)\\ 
	&\leq \frac{c}{\eta \wedge (1-\eta)} k(y-x).
\end{align*}
In the first case, we still have to apply \cite[Lemma 2.1]{{CKS-ALEA}} to obtain that $k(x+x_0)\asymp k(x)$, $x \in X$. With this, we see that in  either  case \eqref{A3} holds. 

\subsection{Justification of the examples presented in Section \ref{sebsec:ex1}} \label{sebsec:ex3} 

It is convenient to embed the examples of $P$ and $V$ (with the profiles $R$ and $W$, respectively, see \eqref{eq:P_R} and \eqref{eq:V_W}) in a slightly more general framework, in which the confining potential $V$ depends on the kernel $P$ in a sufficiently regular way. Set
\begin{gather*}
	r(x) := 2 R^{-2}(\alpha (x)),\quad x \in X,
\end{gather*}
so that there is a constant $C_0 \geq 1$ satisfying
\begin{gather}\label{eq:reg_P} 
  C_0^{-1} r(x) \leq \frac{1}{P(x,x_0)P(x_0,x)} \leq C_0 r(x), \quad x \in X.
\end{gather} 
In addition, we assume that there are an increasing \enquote{profile} function $h:(1,\infty) \to (0,\infty)$ such that $\lim_{r\to\infty} h(r) = \infty$ and constants $0<C_1 \leq C_2$ such that 
\begin{gather} \label{eq:reg_pot}
	C_1 h(r(x)) \leq V(x) \leq C_2 h(r(x)), \quad x \in X.
\end{gather} 
Because of \eqref{cond:duality} one has $P(x,x_0)P(x_0,x) = \Phat(x,x_0)\Phat(x_0,x)$, so \eqref{eq:reg_P} can be equivalently stated in terms of $\Phat$. 

In Sections \ref{subsec:power-type}, \ref{subsec:power-log} and \ref{subsec:power-log-log} below we analyse the cases of polynomially and logarithmically (or slower than logarithmically) growing profiles $h$, leading to aIUC and pIUC, respectively. These cases cover almost all possible situations. Based on these findings, we give in Section \ref{sec:calc-ex}  a  detailed justification for the examples  mentioned in  Section \ref{sebsec:ex1}.

\subsubsection{Power-type profiles and aIUC}\label{subsec:power-type}
Assume that 
\begin{gather*}
	h(r) = r^{\gamma}, \quad r>1, \quad\text{for some\ \ } \gamma >0.
\end{gather*}
We are going to show that this choice leads to aIUC. Indeed, by \eqref{eq:reg_P}--\eqref{eq:reg_pot} we see that
\begin{align*}
	P(x,x_0)P(x_0,x) 
	&\geq \frac{1}{C_0}\left(\frac{1}{h(r(x))}\right)^{1/\gamma}\\
	&\geq \frac{C_0^{-1}C_1^{1/\gamma}}{V(x)^{1/\gamma}}\\ 
	&= C_0^{-1}C_1^{1/\gamma}V(x)^{\left\lceil 1/\gamma \right\rceil-1/\gamma}\frac{1}{V(x)^{\left\lceil 1/\gamma \right\rceil}}, \quad x \in X.
\end{align*}
In particular,
\begin{gather*}
	\frac{C}{V(x)^{n_0-1}} \leq P(x,x_0)P(x_0,x), \quad x \in X,
\intertext{with}
	n_0 := \left\lceil 1/\gamma \right\rceil+1 \geq 2 
	\quad\text{and}\quad 
	C := C_0^{-1}C_1^{1/\gamma} \inf\left\{(1 \wedge V(x))^{\left\lceil 1/\gamma \right\rceil-1/\gamma}: x \in X\right\} >0;
\end{gather*}
this shows that the condition \eqref{eq:V} holds. Consequently, by Theorem \ref{aIUC}, both semigroups $\{\calU_n : n\in \nat_0\}$ and $\{\widehat{\calU}_n: n\in \nat_0\}$ are aIUC($n_0$).

\subsubsection{Power-logarithmic-type profiles and pIUC}\label{subsec:power-log} 
Assume that 
\begin{gather*}
	h(r) = \log^{\gamma} r, \quad r>1, \quad\text{for some\ \ } \gamma >0.
\end{gather*}
We will show that in this case the Feynman--Kac semigroups $\{\calU_n : n\in \nat_0\}$ and $\{\widehat{\calU}_n: n\in \nat_0\}$ are both pIUC, but not aIUC. In order to do this, we will bound the sets $D_n$, $n \in \nat$, which appear in Theorem \ref{th:pIUC}. For technical reasons, it is convenient to consider the sets $E_n = D_n \setminus D$, $n \geq 2$. Observe that \eqref{eq:reg_P}--\eqref{eq:reg_pot} imply $E^{(1)}_n \subseteq E_n \subseteq E^{(2)}_n$, $n \geq 2$, for the sets
\begin{align*}
	E^{(i)}_n:= \left\{z\in X: C_i \lambda_0\,\log^{\gamma} r(z) \geq  C_0^{(-1)^{i-1}} Cr(z)^{\frac{1}{n-1}} \right\}, \quad i=1,2.
\end{align*}
Let $n \geq 2$. The condition defining the sets $E_n^{(i)}$, $i=1,2$, can be expressed as
\begin{gather}\label{eq:log}
	f_n(u) := \widetilde C^{1/\gamma}\,u^{\frac{1}{\gamma(n-1)}}-\log u\leq 0, \quad u>1,
\end{gather}
where $\widetilde C = C_0C/(C_1\lambda_0)$ or $\widetilde C = C/(C_0C_2\lambda_0)$, respectively. In order to find the extrema of the function $f_n$, we differentiate it
\begin{align*}
	f_n'(u)  
	= \frac{\widetilde C^{1/\gamma}}{\gamma(n-1)}\,u^{\frac{1}{\gamma(n-1)}-1}-\frac{1}{u} 
    =\frac{1}{u}\left(\frac{\widetilde C^{1/\gamma}}{\gamma(n-1)}\,u^{\frac{1}{\gamma(n-1)}}-1\right)
\end{align*}
and solve $f_n'(u)=0$ for $u$. We get
\begin{gather*}
 u = u_n=\left(\frac{\gamma(n-1)}{\widetilde C^{1/\gamma}}\right)^{\gamma(n-1)}.
\end{gather*}
Since $f_n$ is decreasing on $(1,u_n)$ and increasing on $(u_n,\infty)$, we see that $f_n$ has a minimum at $u_n$, and $f_n(u_n)<0$. 
Moreover, $f_n(u) \to \infty$ as $u \to \infty$.
By continuity, these properties imply that there exist sequences $(a_n)_{n \geq 2}$ and $(b_n)_{n  \geq 2}$ such that $1 < a_n < u_n < b_n$ and $f_n(a_n)=f_n(b_n)=0$. Moreover, $a_n$ is decreasing and $b_n \to \infty$ as $n \to \infty$. As $\# D < \infty$, this means that $D_n \neq X$ and $\# D_n < \infty$, for all $n \geq 2$. In particular, aIUC cannot hold. Furthermore, these arguments show that we may take 
\begin{gather*}
	D:= \left\{z\in X: \lambda_0 V(z) \leq C\right\} \cup \left\{z \in X: r(z) \leq a_2\right\},
\end{gather*}
and we get 
\begin{gather*}
	\left\{z \in X: r(z) \leq u_n\right\} \subset \left\{z \in X: r(z) \leq b_n\right\} \subseteq D_n 
	\quad\text{for $n$ large enough.}
\end{gather*}
The exhausting family $(A_n)_{n\in \nat}$ for pIUC constructed in Theorem \ref{th:pIUC} can be now approximated from below by finding the maximal $k=k(n)$ such that 
\begin{gather} \label{eq:superset}
	B_k \subseteq \left\{z \in X: r(z) \leq u_n\right\}. 
\end{gather}
This shows the size of the sets $D_n$ and $A_n$. With a more detailed analysis, one can find even better approximations of $b_n$. 

\subsubsection{Power-iterated-logarithmic-type profiles and pIUC}\label{subsec:power-log-log}
We take
\begin{gather*}
	h(r) = \log^{\gamma} \log r, \quad r>e, \quad\text{for some\ \ } \gamma >0.
\end{gather*}
It is convenient to assume that $r(x) > e^2$, 
e.g.\ 
$r(x):= e^2 R^{-2}(\alpha (x))$, for $x\in X$.
Arguing as in Section \ref{subsec:power-log}, we can show that the Feynman--Kac semigroups $\{\calU_n : n\in \nat_0\}$ and $\{\widehat{\calU}_n: n\in \nat_0\}$ are both pIUC, but not aIUC. In particular, we can find an approximation of the exhausting sequence $(A_n)_{n\in \nat}$ for pIUC from Theorem \ref{th:pIUC}. We indicate the main differences only, omitting all details. The condition \eqref{eq:log}, which defines the sets $E_n^{(i)}$, $i=1,2$, now reads as follows:
\begin{gather*}
	f_n(u) := \widetilde C^{1/\gamma}\,u^{\frac{1}{\gamma(n-1)}}-\log \log u\leq 0, \quad u>e,
\end{gather*}
with an appropriate constant $\widetilde C >0$, for $n \geq 2$ large enough. We solve
\begin{gather*}
	f_n'(u)  
	= \frac{\widetilde C^{1/\gamma}}{\gamma(n-1)}\,u^{\frac{1}{\gamma(n-1)}-1}-\frac{1}{u \log u} 
    =\frac{1}{u}\left(\frac{\widetilde C^{1/\gamma}}{\gamma(n-1)}\,u^{\frac{1}{\gamma(n-1)}}-\frac{1}{\log u}\right) = 0
\end{gather*}
for $u$, getting
\begin{gather*}
 u = u_n=a^{\gamma(n-1)},
\end{gather*} 
where $a >1$ is the unique real number such that $a \log a = \widetilde C^{-1/\gamma}$. In order to get a lower approximation for $(A_n)_{n\in \nat}$, we have to look for the maximal $k=k(n)$ such that 
\begin{gather} \label{eq:superset-loglog}
	B_k \subseteq \left\{z \in X: r(z) \leq a^{\gamma(n-1)}\right\}. 
\end{gather}

\subsubsection{Detailed justification for the entries of Table \ref{Table}} \label{sec:calc-ex} We will discuss each combination of the profiles $R$ and $W$ separately. 

\medskip 

\noindent
\textbf{Polynomial transition kernels:\ \boldmath$R(n)=n^{-\beta}$.\unboldmath} \\
First observe that one can take 
$r(x) = 2R^{-2}(\alpha (x))$, for $x\in X$,
in \eqref{eq:reg_P}. 

\smallskip 

\begin{description}
\item[\boldmath\underline{$W(n)=\log^{\rho}(1+n)$}\unboldmath]
We take $h(r) = \log^{\rho} r$ so that $V(x) \asymp h(r(x))$, $x \in X$. In fact, \eqref{eq:reg_pot} holds with $C_1=2^{-\rho}$, $C_2=1$.
Therefore, the general observations in Section \ref{subsec:power-log} show that we have pIUC, but aIUC fails to hold. We will characterize the approximation of the exhausting family $(A_n)_{n\in \nat}$ for pIUC which appears in Theorem \ref{th:pIUC}. Due to \eqref{eq:superset},  we have to solve the following inequality for $k$:
\begin{gather*}
	2k^{2\beta}\leq \left(\frac{(n-1)\rho}{\widetilde C^{1/\rho}}\right)^{(n-1)\rho},
\end{gather*}
where $\widetilde C = C_0C/(C_1\lambda_0)$, with $C$ being the constant from estimate \eqref{outer_estimate} in Theorem \ref{HK:estimates}. Consequently, the exhausting sequence $(B_{k(n)})_{n \in \nat}$ with
\begin{gather*}
	k(1) = 1, 
	\qquad 
	k(n)=\left\lfloor 2^{-\frac{1}{2\beta}}\left(\frac{(n-1)\rho}{\widetilde C^{1/\rho}}\right)^{\frac{(n-1)\rho}{2\beta}}\right\rfloor \vee 1, \quad n \geq 2,
\end{gather*}
is an approximation for the family $(A_n)_{n\in \nat}$. In particular, the Feynman--Kac semigroups $\{\calU_n : n\in \nat_0\}$ and $\{\widehat{\calU}_n: n\in \nat_0\}$ are both pIUC with respect to the exhausting sequence $(B_{k(n)})_{n \in \nat}$. 
 
Let us estimate the (quasi\nobreakdash-)ergodic rate appearing in Theorem \ref{pIUC_erg}. Take $A_n = B_{k(n)}$, $n \in \nat$, in Theorem \ref{pIUC_erg}. By \eqref{eq:V_W}, we get
\begin{gather*}
	\sup_{z\in A^c_m}\frac{1}{V(z)}
	= \sup_{z\in B_{k(m)}^c}\frac{1}{V(z)}
	= \frac{1}{W(k(m))} 
	\asymp \left(\frac{1}{m \log m}\right)^{\rho},
\end{gather*}
for $m$ large enough.  Since this function decays much slower than a geometric sequence, it determines the rate of convergence in Theorem \ref{pIUC_erg}.

\item[\boldmath\underline{$W(n)=n^\rho$}\unboldmath]
It is straightforward to see that the profile function $h$ takes the form $h(r)=r^{\frac{\rho}{2\beta}}$; therefore we have aIUC, as discussed in Section \ref{subsec:power-type}.

\item[\boldmath\underline{$W(n)=\exp(\rho n)$}\unboldmath]
Since we have $n^\rho \leq c \exp(\rho n)$, $n \in \nat$, for some constant $c=c(\rho)>0$, it follows from the previous case and Section \ref{subsec:power-type} that aIUC holds as well.
\end{description}

\medskip 

\noindent
\textbf{Exponential transition kernels:\ \boldmath$R(n)=e^{-\kappa n} n^{-\beta}$.\unboldmath} \\
One can take 
$r(x) = 2R^{-2}(\alpha (x))$, $x \in X$, in \eqref{eq:reg_P}. 

\smallskip 

\begin{description}
\item[\boldmath\underline{$W(n)=n^{\rho}$}\unboldmath]
By taking $h(r)=\log^{\rho} r$, we can show that \eqref{eq:reg_pot} holds with $C_1=(2\kappa+2\beta)^{-\rho}$, $C_2=(2\kappa)^{-\rho}$. 
This implies that the Feynman--Kac semigroups $\{\calU_n : n\in \nat_0\}$ and $\{\widehat{\calU}_n: n\in \nat_0\}$ are both pIUC, but not aIUC. The next step is to find $k(n)$, which is determined by \eqref{eq:superset}. Formally, we have to solve the following inequality for $k$:
\begin{gather*}
	2e^{2\kappa\,k}k^{2\beta} 
	\leq \left(\frac{(n-1)\rho}{\widetilde C^{1/\rho}}\right)^{(n-1)\rho},
\end{gather*}
where $\widetilde C = C_0C/(C_1\lambda_0)$, with $C$ being the constant from estimate \eqref{outer_estimate} in Theorem \ref{HK:estimates}.
Since we look only for a lower approximation for large $n$ and $k$, it is more convenient to analyse the inequality
\begin{gather*}
	e^{(2\kappa+\epsilon)k}
	\leq \left(\frac{(n-1)\rho}{\widetilde C^{1/\rho}}\right)^{(n-1)\rho}.
\end{gather*}
Solving this inequality for $k$, we see that the exhausting sequence $(B_{k(n)})_{n \in \nat}$ with
\begin{gather*}
	k(1) = 1, 
	\qquad 
	k(n)=\left\lfloor \frac{(n-1)\rho \, \log\left(\frac{(n-1)\rho}{\widetilde C^{1/\rho}}\right)}{2\kappa+\epsilon}\right\rfloor \vee 1, \quad n \geq 2,
\end{gather*}
can be chosen as a bound for the family $(A_n)_{n\in \nat}$ in Theorem \ref{th:pIUC}. In particular, the Feynman--Kac semigroups $\{\calU_n : n\in \nat_0\}$ and $\{\widehat{\calU}_n: n\in \nat_0\}$ are both pIUC with respect to the exhausting sequence $(B_{k(n)})_{n \in \nat}$. 

As in the corresponding case of polynomial transition kernels, we can see that the (quasi\nobreakdash-) ergodic rate in Theorem \ref{pIUC_erg} is
\begin{gather*}
	\sup_{z\in B_{k(m)}^c}\frac{1}{V(z)}=\frac{1}{W(k(m))} \asymp \left(\frac{1}{m \log m}\right)^{\rho}.
\end{gather*}

\item[\boldmath\underline{$W(n)=\log^{\rho}(1+n)$}\unboldmath] We take $h(r) = \log^{\rho} \log r$ and proceed in the same way as in the previous case for $W(n)=n^{\rho}$. Both Feynman--Kac semigroups are pIUC, but not aIUC. By using \eqref{eq:superset-loglog} we show that a bound for the family $(A_n)_{n\in \nat}$ in Theorem \ref{th:pIUC} is given by the sequence $(B_{k(n)})_{n \in \nat}$ with
\begin{gather*}
	k(1) = 1, 
	\qquad 
	k(n)=\left\lfloor \frac{(n-1)\rho\log a}{2\kappa+\epsilon}\right\rfloor \vee 1, \quad n \geq 2,
\end{gather*}
for a suitable $a>1$ and any $\epsilon \in (0,1)$. In particular, the Feynman--Kac semigroups $\{\calU_n : n\in \nat_0\}$ and $\{\widehat{\calU}_n: n\in \nat_0\}$ are both pIUC with respect to $(B_{k(n)})_{n \in \nat}$. The (quasi\nobreakdash-)ergodic rate in Theorem \ref{pIUC_erg} takes the form
\begin{gather*}
	\sup_{z\in B_{k(m)}^c}\frac{1}{V(z)}
	= \frac{1}{W(k(m))} 
	\asymp \left(\frac{1}{\log m}\right)^{\rho},
\end{gather*}
for $m$ large enough. 

\item[\boldmath\underline{$W(n)=\exp(\rho\,n)$}\unboldmath]
We will show that we are in the aIUC regime. We take the profile $\widetilde h(r) = r^{(\rho-\epsilon)/(2\kappa)}$, for some $0<\epsilon < \rho$, and use the conclusion of Section \ref{subsec:power-type}. Thus, we obtain that the potential $\widetilde V(x) \asymp \widetilde h(r(x))$, $x \in X$, with the profile $\widetilde W(n) = e^{(\rho-\epsilon) n} n^{\beta(\rho-\epsilon)/\kappa}$
as in \eqref{eq:V_W} guarantees aIUC. Since $\widetilde W(n) \leq c e^{\rho n} = W(n)$, $n \in \nat$, for some constant $c>0$, this is also true for the original potential. 
\end{description}

\bibliographystyle{abbrv}
\bibliography{discrete_f-k_v1}

\begin{thebibliography}{10}

\bibitem{Anastassiou_Bendikov}
G.~A. Anastassiou and A.~Bendikov.
\newblock A discrete analog of {K}ac's formula and optimal approximation of the
  solution of the heat equation.
\newblock {\em Indian J. Pure Appl. Math.}, 28(10):1367--1389, 1997.

\bibitem{Banuelos}
R.~Ba\~{n}uelos.
\newblock Intrinsic ultracontractivity and eigenfunction estimates for
  {S}chr\"{o}dinger operators.
\newblock {\em J. Funct. Anal.}, 100(1):181--206, 1991.

\bibitem{Bakry-Gentil-Ledoux}
D.~Bakry, I.~Gentil, and M.~Ledoux.
\newblock {\em Analysis and geometry of {M}arkov diffusion operators}, volume
  348 of {\em Grundlehren der mathematischen Wissenschaften [Fundamental
  Principles of Mathematical Sciences]}.
\newblock Springer, Cham, 2014.

\bibitem{Barlow_book}
M.~T. Barlow.
\newblock {\em Random walks and heat kernels on graphs}, volume 438 of {\em
  London Mathematical Society Lecture Note Series}.
\newblock Cambridge University Press, Cambridge, 2017.

\bibitem{Bass_Levin}
R.~F. Bass and D.~A. Levin.
\newblock Transition probabilities for symmetric jump processes.
\newblock {\em Trans. Amer. Math. Soc.}, 354(7):2933--2953, 2002.

\bibitem{Bendikov-Saloff}
A.~Bendikov and L.~Saloff-Coste.
\newblock Random walks on groups and discrete subordination.
\newblock {\em Math. Nachr.}, 285(5-6):580--605, 2012.

\bibitem{Bennett-Sharpley}
C.~Bennett and R.~Sharpley.
\newblock {\em Interpolation of operators}, volume 129 of {\em Pure and Applied
  Mathematics}.
\newblock Boston, MA etc.: Academic Press, Inc., 1988.

\bibitem{Chen-Kaleta-Wang}
X.~Chen, K.~Kaleta, and J.~Wang.
\newblock Heat kernels and {G}reen functions for fractional {S}chr\"{o}dinger
  operators with confining potentials, 2025.
\newblock Preprint: arXiv:2502.11320.

\bibitem{Chen-Kim-wang-CMP}
X.~Chen, P.~Kim, and J.~Wang.
\newblock Intrinsic ultracontractivity and ground state estimates of non-local
  {D}irichlet forms on unbounded open sets.
\newblock {\em Comm. Math. Phys.}, 366(1):67--117, 2019.

\bibitem{Chen-Kim-wang-MAn}
X.~Chen, P.~Kim, and J.~Wang.
\newblock Two-sided {D}irichlet heat kernel estimates of symmetric stable
  processes on horn-shaped regions.
\newblock {\em Math. Ann.}, 384(1-2):373--418, 2022.

\bibitem{Chen-Wang}
X.~Chen and J.~Wang.
\newblock Intrinsic contractivity properties of {F}eynman-{K}ac semigroups for
  symmetric jump processes with infinite range jumps.
\newblock {\em Front. Math. China}, 10(4):753--776, 2015.

\bibitem{Chen_Wang_1}
X.~Chen and J.~Wang.
\newblock Intrinsic ultracontractivity of {F}eynman-{K}ac semigroups for
  symmetric jump processes.
\newblock {\em J. Funct. Anal.}, 270(11):4152--4195, 2016.

\bibitem{Collet}
P.~Collet, S.~Mart{\'{\i}}nez, and J.~San~Mart{\'{\i}}n.
\newblock {\em Quasi-stationary distributions. {Markov} chains, diffusions and
  dynamical systems.}
\newblock Probability and its Applications. Berlin: Springer, 2013.

\bibitem{Csaki}
E.~Cs\'{a}ki.
\newblock A discrete {F}eynman-{K}ac formula.
\newblock {\em J. Statist. Plann. Inference}, 34(1):63--73, 1993.

\bibitem{CKS-ALEA}
W.~Cygan, K.~Kaleta, and M.~\'{S}liwi\'{n}ski.
\newblock Decay of harmonic functions for discrete time {F}eynman-{K}ac
  operators with confining potentials.
\newblock {\em ALEA Lat. Am. J. Probab. Math. Stat.}, 19(1):1071--1101, 2022.

\bibitem{Cygan-Sebek}
W.~Cygan and S.~\v{S}ebek.
\newblock Transition probability estimates for subordinate random walks.
\newblock {\em Math. Nachr.}, 294(3):518--558, 2021.

\bibitem{Davies-mon}
E.~B. Davies.
\newblock {\em Heat kernels and spectral theory}, volume~92 of {\em Cambridge
  Tracts in Mathematics}.
\newblock Cambridge University Press, Cambridge, 1989.

\bibitem{Davies-Gross-Simon}
E.~B. Davies, L.~Gross, and B.~Simon.
\newblock Hypercontractivity: a bibliographic review.
\newblock In {\em Ideas and methods in quantum and statistical physics ({O}slo,
  1988)}, pages 370--389. Cambridge Univ. Press, Cambridge, 1992.

\bibitem{Davies-Simon-JFA}
E.~B. Davies and B.~Simon.
\newblock Ultracontractivity and the heat kernel for {S}chr\"{o}dinger
  operators and {D}irichlet {L}aplacians.
\newblock {\em J. Funct. Anal.}, 59(2):335--395, 1984.

\bibitem{Demuth-Casteren}
M.~Demuth and J.~A. van Casteren.
\newblock {\em Stochastic Spretral Theory for Selfadjoint Feller Operators. A
  functional integration approach}.
\newblock Probability and Its Applications. Birkh\"auser Verlag, Basel, 2000.

\bibitem{Diaconis}
P.~Diaconis, K.~Houston-Edwards, and L.~Saloff-Coste.
\newblock Analytic-geometric methods for finite {M}arkov chains with
  applications to quasi-stationarity.
\newblock {\em ALEA Lat. Am. J. Probab. Math. Stat.}, 17(2):901--991, 2020.

\bibitem{Diaconis-Saloff}
P.~Diaconis and L.~Saloff-Coste.
\newblock Logarithmic {S}obolev inequalities for finite {M}arkov chains.
\newblock {\em Ann. Appl. Probab.}, 6(3):695--750, 1996.

\bibitem{Grigoryan-Telcs}
A.~Grigor'yan and A.~Telcs.
\newblock Sub-{G}aussian estimates of heat kernels on infinite graphs.
\newblock {\em Duke Math. J.}, 109(3):451--510, 2001.

\bibitem{Grzywny-Trojan}
T.~Grzywny and B.~Trojan.
\newblock Subordinated {M}arkov processes: sharp estimates for heat kernels and
  {G}reen functions, 2021.
\newblock Preprint: arXiv:2110.01201v2.

\bibitem{Kaleta-Kulczycki}
K.~Kaleta and T.~Kulczycki.
\newblock Intrinsic ultracontractivity for {S}chr\"{o}dinger operators based on
  fractional {L}aplacians.
\newblock {\em Potential Anal.}, 33(4):313--339, 2010.

\bibitem{Kaleta-Kwasnicki-Lorinczi-JST}
K.~Kaleta, M.~Kwa\'{s}nicki, and J.~L\H{o}rinczi.
\newblock Contractivity and ground state domination properties for non-local
  {S}chr\"{o}dinger operators.
\newblock {\em J. Spectr. Theory}, 8(1):165--189, 2018.

\bibitem{Kaleta-Lorinczi-SPA}
K.~Kaleta and J.~L\H{o}rinczi.
\newblock Fractional {$P(\phi)_1$}-processes and {G}ibbs measures.
\newblock {\em Stochastic Process. Appl.}, 122(10):3580--3617, 2012.

\bibitem{Kaleta-Lorinczi-AoP}
K.~Kaleta and J.~L\H{o}rinczi.
\newblock Pointwise eigenfunction estimates and intrinsic
  ultracontractivity-type properties of {F}eynman-{K}ac semigroups for a class
  of {L}\'{e}vy processes.
\newblock {\em Ann. Probab.}, 43(3):1350--1398, 2015.

\bibitem{Kaleta-Schilling-JFA}
K.~Kaleta and R.~L. Schilling.
\newblock Progressive intrinsic ultracontractivity and heat kernel estimates
  for non-local {S}chr\"{o}dinger operators.
\newblock {\em J. Funct. Anal.}, 279(6):108606, 69, 2020.

\bibitem{KSch-ergodic}
K.~Kaleta and R.~L. Schilling.
\newblock Quasi-ergodicity of compact strong {F}eller semigroups on {$L^2$}.
\newblock {\em J. Inst. Math. Jussieu}, 24(2):541–584, 2025.

\bibitem{Keller_Lenz}
M.~Keller and D.~Lenz.
\newblock Dirichlet forms and stochastic completeness of graphs and subgraphs.
\newblock {\em J. Reine Angew. Math.}, 666:189--223, 2012.

\bibitem{keller_lenz_wojciechowski_2021}
M.~Keller, D.~Lenz, and R.~Wojciechowski.
\newblock {\em Graphs and Discrete Dirichlet Spaces}, volume 358 of {\em
  Grundlehren der Mathematischen Wissenschaften [Fundamental Principles of
  Mathematical Sciences]}.
\newblock Springer International Publishing, 2021.

\bibitem{Kim-Tagawa}
D.~Kim and T.~Tagawa.
\newblock Quasi-ergodic limits for moments of jumps under absorbing stable
  processes.
\newblock {\em J. Theoret. Probab.}, 38(1):Paper No. 10, 19, 2025.

\bibitem{Kim-Song-2008}
P.~Kim and R.~Song.
\newblock Intrinsic ultracontractivity of non-symmetric diffusion semigroups in
  bounded domains.
\newblock {\em T{\^o}hoku Math. J. (2)}, 60(4):527--547, 2008.

\bibitem{Kim-Song-AoP}
P.~Kim and R.~Song.
\newblock Intrinsic ultracontractivity of nonsymmetric diffusions with
  measure-valued drifts and potentials.
\newblock {\em Ann. Probab.}, 36(5):1904--1945, 2008.

\bibitem{Kim-Song-FM}
P.~Kim and R.~Song.
\newblock Intrinsic ultracontractivity for non-symmetric {L}\'{e}vy processes.
\newblock {\em Forum Math.}, 21(1):43--66, 2009.

\bibitem{Knobloch-Partzsch}
R.~Knobloch and L.~Partzsch.
\newblock Uniform conditional ergodicity and intrinsic ultracontractivity.
\newblock {\em Potential Anal.}, 33(2):107--136, 2010.

\bibitem{Kulczycki-Siudeja}
T.~Kulczycki and B.~Siudeja.
\newblock Intrinsic ultracontractivity of the {F}eynman-{K}ac semigroup for
  relativistic stable processes.
\newblock {\em Trans. Amer. Math. Soc.}, 358(11):5025--5057, 2006.

\bibitem{Kulczycki-Sztonyk}
T.~Kulczycki and K.~Sztonyk.
\newblock Intrinsic ultracontractivity for {S}chr\"odinger semigroups based on
  cylindrical fractional {L}aplacian on the plane, 2024.
\newblock Preprint: arXiv:2407.14325.

\bibitem{Kumagai}
T.~Kumagai.
\newblock {\em Random walks on disordered media and their scaling limits},
  volume 2101 of {\em Lecture Notes in Mathematics}.
\newblock Springer, Cham, 2014.
\newblock Lecture notes from the 40th Probability Summer School held in
  Saint-Flour, 2010, \'{E}cole d'\'{E}t\'{e} de Probabilit\'{e}s de
  Saint-Flour. [Saint-Flour Probability Summer School].

\bibitem{Kwasnicki}
M.~Kwa\'{s}nicki.
\newblock Intrinsic ultracontractivity for stable semigroups on unbounded open
  sets.
\newblock {\em Potential Anal.}, 31(1):57--77, 2009.

\bibitem{Metafune-Spina}
G.~Metafune and C.~Spina.
\newblock Kernel estimates for a class of {S}chr\"{o}dinger semigroups.
\newblock {\em J. Evol. Equ.}, 7(4):719--742, 2007.

\bibitem{Tweedy}
S.~Meyn and R.~L. Tweedie.
\newblock {\em Markov chains and stochastic stability}.
\newblock Cambridge: Cambridge University Press, 2nd edition, 2009.

\bibitem{Murugan_Saloff_2}
M.~Murugan and L.~Saloff-Coste.
\newblock Transition probability estimates for long range random walks.
\newblock {\em New York J. Math.}, 21:723--757, 2015.

\bibitem{Murugan_Saloff-Coste}
M.~Murugan and L.~Saloff-Coste.
\newblock Heat kernel estimates for anomalous heavy-tailed random walks.
\newblock {\em Ann. Inst. Henri Poincar\'{e} Probab. Stat.}, 55(2):697--719,
  2019.

\bibitem{Nelson}
E.~Nelson.
\newblock The free {M}arkoff field.
\newblock {\em J. Funct. Anal.§}, 12:211--227, 1973.

\bibitem{Nieradko-Olszewski}
M.~Nieradko and M.~Olszewski.
\newblock Good labeling property of simple nested fractals.
\newblock {\em J. Fractal Geom.}, 11(1-2):31--56, 2024.

\bibitem{Ouhabaz-Rhandi}
E.~M. Ouhabaz and A.~Rhandi.
\newblock Kernel and eigenfunction estimates for some second order elliptic
  operators.
\newblock {\em J. Math. Anal. Appl.}, 387(2):799--806, 2012.

\bibitem{Schaefer}
H.~H. Schaefer.
\newblock {\em Banach lattices and positive operators}.
\newblock Springer-Verlag, New York-Heidelberg, 1974.
\newblock Die Grundlehren der mathematischen Wissenschaften, Band 215.

\bibitem{Takeda}
M.~Takeda.
\newblock Existence and uniqueness of quasi-stationary distributions for
  symmetric {M}arkov processes with tightness property.
\newblock {\em J. Theoret. Probab.}, 32(4):2006--2019, 2019.

\bibitem{Zhang-Li-Liao}
H.~Zhang, H.~Li, and S.~Liao.
\newblock Intrinsic ultracontractivity and uniform convergence to the
  {$Q$}-process for symmetric {M}arkov processes.
\newblock {\em Electron. Commun. Probab.}, 28:Paper No. 44, 12, 2023.

\bibitem{Zhang}
J.~Zhang, S.~Li, and R.~Song.
\newblock Quasi-stationarity and quasi-ergodicity of general {Markov}
  processes.
\newblock {\em Sci. China, Math.}, 57(10):2013--2024, 2014.

\end{thebibliography}

\end{document}